\documentclass[11pt, reqno]{amsart}
\usepackage[english]{babel}
\usepackage{a4wide,xcolor,graphicx,amsmath,amssymb,verbatim,mathrsfs,latexsym,hyperref}

\allowdisplaybreaks
\numberwithin{equation}{section}
\newtheorem{theo}{Theorem}[section]

\newtheorem{lemma}{Lemma}[section]

\newtheorem{cor}{Corollary}[section]
\newtheorem{rem}{Remark}[section]

\newcommand{\R}{\mathbb{R}}	% Set of real numbers.
\newcommand{\N}{\mathbb{N}}	% Set of natural numbers.
\newcommand{\eps}{\varepsilon}	% Epsilon.
\newcommand{\pa}{\partial}		% Partial derivative.
\newcommand{\Div}{\textrm{div}\,}	% Divergence.
\newcommand{\na}{\nabla}		% Nabla.

\newcommand{\dd}{\mbox{d}}
\newcommand{\eye}{\mbox{Id}}

\newcommand{\bk}[1]{\langle #1 \rangle}
\newcommand{\mm}{{\bf\mathfrak{m}}}
\newcommand{\TT}{\mathcal{T}}
\newcommand{\meas}[1]{\mbox{meas}\left\{ #1 \right\}}

\title[Fuzzy Landau equation. Existence and uniqueness of global smooth solutions]{The fuzzy Landau equation: global well-posedness and Fisher information}

\author{Maria Pia Gualdani}
\address{The University of Texas at Austin, USA}
\email{gualdani@math.utexas.edu}

\author{Nestor Guillen}
\address{New York University, USA}
\email{nestor@txstate.edu}

\author{Nata\v sa Pavlovi\' c}
\address{The University of Texas at Austin, USA}
\email{natasa@math.utexas.edu}

\author{Maja Taskovi\' c}
\address{Emory University, USA}
\email{maja.taskovic@emory.edu}

\author{Nicola Zamponi}
\address{Universit\"at Augsburg, Germany}
\email{nicola.zamponi@uni-a.de}

\thanks{MPG is partially supported by NSF grant DMS-2511625, DMS-2206677 and by the SPARK Grant from the College of Natural Science at UT Austin. NG is partially supported by NSF grant DMS-214423. NP is partially supported by the NSF grants  DMS-1840314, DMS-2009549 and DMS-2052789. MT is partially supported by the NSF grant DMS-2206187. NZ thanks the University of Augsburg for financial support via the research program “Forschungspotentiale besser nutzen!”. The first author would like to thank Luis Silvestre for the  fruitful discussion.  This material is based upon work supported by the National Science Foundation under Grant No. DMS-2424139, while MPG, NP and MT  were in residence at the Simons Laufer Mathematical Sciences Institute in Berkeley, California during the Fall 2025 semester. MPG, NP and MT would like to thank the Simons Laufer Mathematical Institute for their kind hospitality and stimulating environment. Math subject classification: 35Q20, 35B65, 35B45, 82C40.}
\date{\today}

\begin{document}

\thanks{} 

%\keywords{}  

\begin{abstract} We investigate a fuzzy variant of the inhomogeneous Landau equation and establish global-in-time existence and uniqueness of smooth solutions for a range of soft potentials. The spatial delocalization introduced in the collision operator not only  prevents singularity formation and enhances regularity, but also reveals additional structural properties of the model. In particular, we show that several forms of the Fisher information decay monotonically or remain uniformly bounded in time.

\end{abstract}

\maketitle
 
\tableofcontents

\markboth{Fuzzy Landau equation. Existence and uniqueness of global smooth solutions}{M. Gualdani, N. Guillen, N. Pavlovi\' c, M. Taskovi\' c, and N. Zamponi}

%{{
%\section*{Questions/comments}
%\begin{enumerate}
%    \item MT/NP: Lemma 2.6 uses that $m \ge 1 + \frac{\lambda}{2}$. It seems to us that the same constraint is used in Lemma 2.7. However, the main Lemma 2.3 uses a different constraint $m \ge 4 + 2\lambda$
   % \item MT/NP: Here are some comments about Lemma 2.7:
   % \begin{itemize}
    %    \item We checked the proof of Lemma 2.7 carefully - it looks corrects. 
    %    \item However, we are wondering why we are introducing the Gaussian in the proof.  We quickly tried to do the proof without Gaussian and it seems doable. Are we missing something obvious?
   % \end{itemize}
   % \item MT/NP: In the proof of Lemma 2.3, on page 22, we cannot verify the formula for $Y_\eta(e_3)$ (second line).
   % \item MT/NP: Proof of parts (ii) and (iii) of Lemma 3.1. In each proof there is a black line that we could not justify that swaps $\nabla_v$ to $\nabla x$ (or vice versa).
    %\item MT/NP: In the proof of Lemma 3.4, which is in the Appendix B (towards the end),  we need to check that the commutator $[\tilde{b}_k ,  \hat{n} ]$ is zero, and maybe add that fact to Lemma 3.2.
%\end{enumerate}
%}}

\section{Introduction}

The existence of global smooth solutions remains a challenging open problem in the analysis of inhomogeneous Boltzmann and Landau equations, with the possibility of finite-time singularity formation being at the center of the difficulty. Global smooth solutions have been established in the perturbative regime and, more recently, for homogeneous initial data. On the other hand, it has recently been shown that certain smooth initial data lead to finite-time hydrodynamic singularities. Understanding the mechanisms that distinguish these two regimes is a primary motivation for this work.

 We consider a {\bf{spatially delocalized}}  inhomogeneous Landau equation:
\begin{align}\label{1}
\pa_t f + v\cdot \nabla_x f = \Div_v\int_{\R^3}\int_{\R^3}N(v-v^*)\kappa(x-x^*)(f^*\nabla_v f - f\nabla_{v^*} f^*)\dd x^*\dd v^*, 
\end{align}
 with $ (x,v) \in \R^3 \times \R^3$. We use the notation $f = f(x,v)$ and $f^* = f(x^*,v^*)$. 
The kernel $N$ is the classical Landau kernel
\begin{align}\label{N}
N(v) :=\alpha(|v|) |v|^{2}\Pi(v),
\end{align}
where 
\begin{align*}
\qquad \Pi(v) := \eye - \frac{v\otimes v}{|v|^2}, \quad \alpha(|v|) := |v|^{\gamma},\qquad -3\leq\gamma.
\end{align*}

The distinguishing feature of \eqref{1} is that {\em{collisions}} are \emph{delocalized}. That is, in contrast to traditional kinetic models,  ``collisions'' occur between particles even if they are far apart: particles at  locations $x$ and $x^*$ may undergo a ``collision'' as long as  $\kappa(x-x^*)>0$. Replacing the kernel $\kappa$ in \eqref{1} with a Dirac delta $\delta(x-x^*)$ yields the classical inhomogeneous Landau equation.  

The {{introduction of fuzziness}} expands the range of possible interactions and provides a framework for modeling collisions that obey uncertain or imprecise laws; historically,  in gas dynamics, several  authors have considered  delocalized kinetic models, viewing them as \emph{mollifications} of traditional kinetic equations. This perspective partially motivated the works of Morgenstern \cite[Section 5]{Morg55} and Povzner \cite{Povz62} (english translation in \cite{Pov1965}), who considered such mollifications specifically for the Boltzmann equation. Some decades after the  works of Morgenstern and Povzner, Cercignani \cite{Cer1983} showed such ``mollified equations'' describe the evolution of the Grad limit for a model of ``soft spheres'' where spheres are allowed to overlap each other.  Further works with delocalized collisions are considered in \cite{Povz62, Arkeryd86, Arkeryd90, Pan04,AN99, FHG11}, including the mollified Boltzmann equation, and the Enskog equation for dense gases with non-negligible particle sizes. Whether these models have solutions that are smooth for arbitrarily large times is still an open question. Another example of delocalized collisions can be found in the modeling of flocking and alignment dynamics; see \cite{BCC11, DFT10, KMT13, S23, T21} and references therein.

While we are not currently aware of specific physical systems directly modeled by the fuzzy Landau equation \eqref{1}, nor  is it immediately clear how grazing and delocalized collisions might co-occur in physical modeling, our interest in  \eqref{1}  is  from a mathematical analysis prospective,   with the hope that studying \eqref{1} helps down the line to understand the classical  inhomogeneous regimes.

From the perspective of regularity theory, fuzzy kinetic equations occupy a middle ground between the spatially homogeneous and inhomogeneous regimes. This is especially true because the spatial kernel $\kappa$  considered here is strictly positive everywhere, avoiding the analytical difficulties typically associated with vacuum regions. Thanks to this, we will show that {\bf{spatial delocalization of collisions prevents singularities}} and facilitates the finding of {\bf{additional Lyapunov functionals}}, concretely various forms of the Fisher information.

The main results of this paper are collected in the following theorem: \\

 {\bf{Main theorem.}} 
 {\em{
\begin{itemize}
    \item (A-priori estimate) 
    The Fisher information
 $$
 I(t):= \int e^{-\delta t} \left(\frac{|\nabla_x f|^2}{f}  + \frac{|\nabla_v f|^2}{f} \right) \;dxdv
 $$
 is monotonically decreasing for any solution $f$ to (\ref{1}) with   $\gamma \in (-{{\sqrt{8}}},-2]$ and  kernels $\kappa(x)>0$ that satisfy 
    \begin{align}\label{cond_kappa_Fisher_maintheo}
        \left\|  \kappa \left(\frac{|\nabla \kappa|}{\kappa}\right)^{^{-\gamma}}\right\|_{L^\infty(\mathbb{R}^3)} \le c_0 .
    \end{align} The constant $\delta >0$ depends on $c_0$ in (\ref{cond_kappa_Fisher_maintheo})  and on the mass of the initial data. \\

 \item (Global well-posedness) Equation  (\ref{1}) equipped with smooth general initial data and  kernel $\kappa$  satisfying for all $x \in \mathbb{R}^3$: 
\begin{equation}
\begin{aligned}\label{kappa}
&\kappa(x)  =\bk{x}^{-\lambda},
\end{aligned}
\end{equation}
for any $\lambda \ge 0$, has a  unique smooth solution for any  $\gamma \in (-{{\sqrt{8}}},0]$.
\end{itemize}
}}

Precise statements of the  global well-posedness and Fisher information results are given in Theorem \ref{Thm1} and  Theorem \ref{Theo_Fisher_xv} respectively. As a corollary of Theorem \ref{Theo_Fisher_xv}, we also show that the spatial Fisher information
 $$
 \int \frac{|\nabla_x f|^2}{f}   \;dxdv
 $$
 is {\em{monotone decreasing}} if the spatial kernel $\kappa \equiv 1$ for any $\gamma \ge -3$. This is the content of  Corollary \ref{Cor:Fisher_Landau_x}. Kernels satisfying condition (\ref{cond_kappa_Fisher_maintheo}) include 
 \begin{align}\label{all-kappa}
 \kappa(x) \simeq 1, \bk{x}^{-\lambda}, e^{-\bk{x}}, e^{-|x|^2}.
 \end{align}
 
For simplicity of exposition, the proof of global well-posedness in Theorem \ref{Thm1} is carried out using kernels $\kappa$ as in (\ref{kappa}). Adjusting the proof to kernels of the form $e^{-\bk{x}}$  is straightforward. Adjusting the proof to kernels of the form $e^{-|x|^2}$ requires some modifications in the coercivity and higher Sobolev regularity arguments (perhaps needing exponential weights instead of just polynomials).   
 
 In the last section of the paper we briefly recall the variational formulation of  (\ref{1}) in the GENERIC (General Equation for the nonequilibrium Reversible-Irreversible Coupling) framework. This formulation was first proposed in \cite{Duong_Li_1}. 

{\bf{Literature Review.}} To the authors' best knowledge,  equation \eqref{1} had appeared in the literature only very recently: in \cite{Duong_Li_1}, Duong and He developed a variational framework for \eqref{1}, following up on work for the fuzzy Boltzmann equation by Erbar and He \cite{Erbar24,Erbar25}. Recently, Duong and He showed in  \cite{Duong_Li_2} existence of $H$-solutions for potentials $\gamma \in (-\min\{d,4\},1]$ for any $d\ge 2$, for kernels $\kappa \sim C$ and $\kappa \sim e^{-\langle x \rangle}$. We mention also that spatial delocalization  has been a tool  to regularize numerical schemes, see for example the recent work by Bailo, Carrillo and Hu \cite{BCH24} on the Vlasov-Maxwell-Landau equations. A very recent result on conditional uniqueness and explicit stability estimates in the 2-Wasserstein distance for the fuzzy Landau equation was shown by Caja-Lopez in \cite{Caja2026uniquenessFuzzy}; similar results were previously proven for the Enskog equation  by Friesen, R\"udiger and Sundar in  \cite{FriesenMartinSundar2022BoltEnskog}. Very recently, Shvydkoy in \cite{S26} showed existence of smooth solutions for a class of fuzzy Fokker-Planck-Alignment equations and studied the decay of the mixed Fisher information.  

We briefly review some key developments regarding the space-inhomogeneous Landau equation, focusing on global existence and regularity. For the Cauchy problem, the existence of global-in-time solutions was initially established in the seminal work of Di Perna and Lions \cite{Lio1994} in the context of renormalized solutions. For initial data near equilibrium, the existence of global smooth solutions has been proven in several works, including \cite{DLSS2020, Guo02, KimGuoHwang2016}.  Also, global well-posedness as well as scattering was established for initial data close to the vacuum in \cite{luk19}.

In the space-homogeneous regime (where initial data depends only on the velocity variable), singularity formation was first ruled out by one of the authors and Silvestre \cite{GS24}. Subsequent works have demonstrated instant regularization and global smoothness for non-smooth initial data under varying assumptions. Notable contributions include works by one of the authors with Golding and Loher \cite{GoGuLo24} and with Desvillettes, Golding and Loher \cite{DGL24}, as well as studies by Ji \cite{Ji24,Ji24a} and a recent preprint by He, Ji, and Luo \cite{HJJ2025}.  Uniqueness of the solution found in \cite{GoGuLo24} was recently proven in \cite{GS26}. A recent preprint by Junne, Winter, and Yoldas \cite{JWY26} addresses global well-posedness and monotonicity of the Fisher information for the multi-species homogeneous Landau equation.  The first author and Winter in \cite{GW24} identified a specific regularizing mechanism for the homogeneous Landau equation.

Returning to the space-inhomogeneous regime, Cameron, Silvestre, and Snelson \cite{CamSilSne2017}, alongside collaborations by Henderson, Snelson, Solomon, and Tarfulea \cite{HendersonSnelson2017,HendersonSnelsonTarfulea2017,SnelSol23}, have established regularity estimates for larger initial data, conditioned on certain lower-regularity structural bounds. We highlight further the works of Silvestre and Imbert \cite{Imb_Silv2017,IS22} where $C^\infty$ regularity for the Boltzmann equation without angular cut off is shown to hold as long as the hydrodynamic quantities 
\begin{align}
\rho(x,t) := \int f \;dv ,\quad 
E(x,t) := \int f |v|^2 \;dv ,\quad 
\mathcal{E}(x,t) := \int  f \log f \;dv ,
\end{align}
stay uniformly bounded along with $\rho^{-1}$. Essentially, the results in \cite{HendersonSnelson2017,HendersonSnelsonTarfulea2017,SnelSol23, Imb_Silv2017,IS22} show that if the density $f$ develops singularities in finite time, then those singularities will be reflected up at the level of the macroscopic quantities $\rho,E$, and $H$.  Golding, Henderson and Silvestre showed in \cite{GH25} that for a range of soft potentials global well-posedness is conditional only on the weighted $L^\infty$ norm of $f$; essentially, if one can prove that $f$ stays uniformly bounded, the hydrodynamic quantities of $f$ do not form finite time singularities.  

Crucially, the conditional regularity results mentioned above do not rule out singularity formation. Establishing whether the inhomogeneous Landau or Boltzmann equation develops finite-time singularities remains a major open question.  In this context,  the first author together with Bedrossian, Chen, Ji, Vicol and Yang very recently showed in \cite{BCGJVY26} that the inhomogeneous Landau equation with very hard potentials $\gamma \in (\sqrt{3}, 2]$ has solutions that exhibit finite time hydrodynamic singularities.  

\vspace{0.2cm}

In this work we show that for solutions to  \eqref{1} the quantities $f(x,v,t)$, $\rho(x,t)$, $E(x,t)$ and $\mathcal{E}(x,t)$ remain bounded for all times $t>0$. The main observation here is that {\bf{delocalization implies hypocoercivity, and hypocoercivity prevents finite time singularities.}} Although here we only consider  soft potentials, we expect that our analysis also  applies to hard potentials. Similarly, we consider smooth initial data but the analysis can easily be extended to less regular initial data, following the approach recently developed in \cite{GoGuLo24, Ji24}. We also remark that our approach could  naturally extend to the Boltzmann equation; specifically, our results from Section \ref{sec:fisher} should hold for the fuzzy Boltzmann equation by adapting the methods developed by Imbert, Silvestre, and Villani \cite{ISV24} (see also \cite{Villani2025}) for the homogeneous Boltzmann equation. 
Assumption \eqref{kappa} could be relaxed to other types of kernels without major changes as long as the support of $\kappa$ is the whole $\mathbb{R}^3$.   The fact that $\kappa>0$ everywhere is significant in that it eliminates any analytical challenges that might arise from the presence of vacuum regions. Indeed, {\bf{since $\kappa>0$ everywhere, the diffusive effects of the delocalized collisions are felt everywhere immediately and one expects vacuum regions to disappear for positive times}}. A very interesting and challenging question is whether smooth solutions exist when the kernel $\kappa$ is compactly supported.

The proof of our main results relies on two key ingredients. The first is an estimate on the diffusion coefficient, showing that it is strictly positive everywhere. The second is the proof that $f$ remains bounded. To establish this, we employ two different strategies depending on the value of $\gamma$. When $\gamma>-2$, standard energy estimates are sufficient. In contrast, for $\gamma\le -2$, we exploit bounds provided by the Fisher information. More precisely, we prove that the Fisher information in both the spatial and velocity variables is bounded, and then use this additional energy functional together with a weighted Poincaré inequality (an $\varepsilon$-Poincaré inequality) to deduce the boundedness of $f$.

\subsection{Assumptions on the initial data:}\label{sec:initial data} In the rest of the manuscript we will always assume that the initial data $f_{in} \ge 0 $, $ f_{in}\in H^s{\cap L^\infty}(\mathbb{R}^6)$ with $s\ge 2$,  has finite mass, momentum, energy, and entropy:  
\begin{align}
\rho_0 & = \int_{\R^6}f_{in}(x,v)\dd x\dd v,  \nonumber \\
c_2 & = \int_{\R^6} (|x|^2+|v|^2) f_{in}(x,v)\dd x\dd v , \quad 0 = \int_{\R^6} v_{i} f_{in}(x,v)\dd x\dd v, \label{moments}\\
h_0 & = \int_{\R^6}f_{in}\log f_{in}\dd x\dd v.   \nonumber
\end{align}
Moreover 
\begin{align} \label{hp.ic}
c_{\mm}:= \int_{\R^6}(|x|^{\mm}+|v|^{\mm})f_{in}\dd x\dd v<\infty,
\end{align}
with 
\begin{align}\label{lambda,m}
\mm {>} \max \left\{ 3\lambda+{15},  \frac{3(s-\gamma-1)}{7+2\gamma} {+\lambda} ,\; \frac{ 3 |\gamma|}{3+\gamma}, 
%\frac{ 2 |\gamma|}{3+\gamma}, 
2|\gamma|, \frac{ 6 \lambda + 8 |\gamma|-16}{3(3+\gamma)} \right\}, 
\end{align}
{where $\lambda$ comes from \eqref{kappa}}. 
To show regularity of solution we will assume moreover that     \begin{align}\label{Hkreg.ic}
(1+ |x|^2 + |v|^2)^{s_0+s-|\alpha| - |\beta|} D_v^{\alpha} D^{\beta}_x f_{in}\in L^2(\R^3\times\R^3),
\end{align}
for any multi-index $ \alpha, \beta$ such that $|\alpha|+|\beta| \leq s $ and $s_0 \geq \frac{3\lambda + 3 -\gamma}{2}$. 

\subsection{Notation} We define
\begin{align*}
	\bk{x} := \sqrt{1 + |x|^2}\quad\forall \; x  \in\R^3,
\end{align*}
$$B_R^{x,v}:= \{ (x,v)\in\R^6~:~|x|^2+|v|^2< R^2 \},$$
$$B_R^{x,v}(x_0,v_0):= \{ (x,v)\in\R^6~:~|x-x_0|^2+|v-v_0|^2< R^2 \}.$$ 
We denote with $[a,b]$ the commutator vector between vectors $a$ and $b$, whose $i$-component is
$$
[a,b]_i = a \cdot \nabla b_i - b \cdot \nabla a_i.
%\sum_j a_j \partial_j b_i - b_j \partial_j a_i. 
$$

\subsection{Organization of the paper}

In Section \ref{ss.estim} we present relevant a-priori estimates and the proof of coercivity summarized in Lemma \ref{lem.aA.lb}. In Section \ref{sec:fisher} we calculate the evolution of the Fisher information. In Section \ref{sec:poincare} we prove the $\varepsilon$-Poincar\'e inequality. In Section \ref{sec:reg} we show that $f\in L^\infty$. Existence, uniqueness and regularity are shown in Section \ref{sec:gwp}. Finally, Section \ref{sec:generic} shows that the fuzzy Landau has a variational formalation in the GENERIC framework. 

\subsection{Use of artificial intelligence} The authors used artificial intelligence to improve the English language and exposition. 

\section{A Priori estimates: moments and coercivity.}\label{ss.estim}
In this section, we establish several a priori estimates for the solution $f$ to \eqref{1}. Throughout, we assume that $f$ is sufficiently regular to justify all computations and that it satisfies the conservation of mass and entropy dissipation conditions stated in \eqref{cons_laws} and \eqref{boltz_H}, with the spatial kernel $\kappa$ satisfying \eqref{kappa}.

To facilitate our analysis, it is useful to rewrite \eqref{1} in the following divergence form:  
\begin{align}\label{eq:Landau_Aa}
    \partial_t f + v \cdot \nabla_x f 
    = Q(f,f) :=\Div_v \big( A[f] \nabla_v f - f b[f] \big),
\end{align}
where $A[f]$ is the diffusion matrix  defined by 
 \begin{align}\label{Af}
A[f](x,v,t) &:= \int_{\R^6}N(v-v^*)\kappa(x-x^*)f(x^*,v^*,t)\dd x^*\dd v^*,
\end{align}
and 
\begin{align*}
  b[f] := \Div_v  A[f].  %\quad  a[f](x,v,t) :=  \Div_v(\Div_v  A[f]). %\mbox{Tr} \; A[f](x,v,t).
\end{align*}
Often, our estimates will involve the function 
$$
H[f](x,v,t):=\int_{\R^3}\kappa(x-y)f(y,v,t)\dd y.
$$
Most importantly, this function $H[f]$  will play a fundamental role in  the proof of the lower bound for the diffusion matrix  $A[f]$: 
\begin{align}\label{aA.lb}
\xi\cdot A[f](x,v,t) \xi \gtrsim \bk{x}^{-\lambda}\bk{v}^{\gamma}|\xi|^2,\qquad\forall\xi\in\R^3, \qquad\forall \; x, v \; \in\R^3, \quad t\ge 0.
\end{align}

This essential lower bound  will be proved in Lemma~\ref{lem.aA.lb} below. 

Like classical inhomogeneous equations, solutions to (\ref{1}) formally conserve mass, second moment, and energy along the time evolution, i.e. for any $i \in {1,2,3}$
\begin{align}\label{cons_laws}
\partial_t \iint 
\begin{pmatrix}
1 \\
v_i\\
|v|^2
\end{pmatrix} f
\;dvdx = 0,
\end{align}
and satisfy the $H$-Boltzmann theorem 
\begin{align}\label{boltz_H}
\partial_t \int \int f \log f\;dvdx \le 0. 
\end{align}

%We first provide moments estimates. 

\subsection{Moments} The presence of the spatial kernel $\kappa(x) \sim \langle x \rangle^{-\lambda}$
  in the collision operator dictates our approach to bounding the moments of $f$.

\begin{lemma}\label{ss.moments} 
Let $\mm\ge 0$, and let $f$ be a  solution to (\ref{1}) with initial datum $f_{in}$.
Assume that 
$$
\int_{\R^6}(1+ |x|^{\mm}+|v|^{\mm})f_{in}(x,v)\dd x\dd v \le c_{\mm}.
$$
 Then for every $T>0$ we have 
\begin{align}\label{mom.high}
	\sup_{t\in [0,T]}\int_{\R^6}(|x|^{\mm}+|v|^{\mm})f(x,v,t)\dd x\dd v < C_{c_{\mm},T},
\end{align}
in the following cases:\\

(i) $-2 \le \gamma \le 0$, \\

(ii) $-3 < \gamma <-2$ provided the function 
$$
F(v,t) := \int_{\mathbb{R}^3_x} f(x,v,t)\;dx
$$
is such that 
\begin{align}\label{condFL3}
\sup_T \int_{\R^3} F^{3}\;dv\le C_T
\end{align}

% $-\frac{14}{5} \le \gamma <-2$,  provided the  $L^1(0,T,L^{3/2}(\mathbb{R}^6))$ norm of $f$ is bounded: 
% \begin{align}\label{Fisher.bound.mom}
% \int_0^T \left(\int_{\R^6} f^{3/2}\;dvdy\right)^{2/3} dt\le C_T,
% \end{align}

% (iii) $-3 < \gamma <-\frac{14}{5}$, provided \eqref{Fisher.bound.mom} holds, and  $\lambda \ge 1 $ in (\ref{kappa}). 
\end{lemma}

\begin{proof}
%{\color{teal} Q: Comment on $\mm \in [0,2]$.} 
We first recall that conservation of mass and second moment yields 
$$
\int_{\R^6}(1+|v|^2) f(x,v,t)\dd x\dd v = const.
$$

We use $|x|^s$ as test function in \eqref{1}. It is enough to consider $s\in\N\cap [2,\mm]$. In fact, $s=0$ is the conservation of mass. If $s=2$ we obtain 
\begin{align*}
	\frac{d}{dt}\int_{\R^6}|x|^2 f(x,v,t)\dd x\dd v = 2 \int_{\R^6}
	x\cdot v f(x,v,t)\dd x\dd v \le \int_{\R^6} |x|^2 f \;dxdv + \int_{\R^6} |v|^2 f \;dxdv 
\end{align*}
and Gronwall's inequality yields
$$
\sup_{t\in [0,T]} \int_{\R^6}|x|^2 f(x,v,t)\dd x\dd v \le C_{c_2,T}.
$$
For $s\ge 2$ we get
\begin{align*}
	\frac{d}{dt}\int_{\R^6}|x|^s f(x,v,t)\dd x\dd v =
	s\int_{\R^6}|x|^{s-2}x\cdot v f(x,v,t)\dd x\dd v.
\end{align*}
Young's inequality yields %{\color{teal} {MT/NP: we'll double-check this: Maria: Use Young's: $ab \le \frac{a^p}{p}+ \frac{a^{p'}}{p'}$ with $a = s |x|^{s-1}$, $b = s|v|$ and $p=\frac{s}{s-1}$.}}
\begin{align*}
	\frac{d}{dt}\int_{\R^6}|x|^s f(x,v,t)\dd x\dd v\leq 
	(s-1)\int_{\R^6}|x|^s f(x,v,t)\dd x\dd v
	+\int_{\R^6}|v|^s f(x,v,t)\dd x\dd v,
\end{align*}
implying via Gronwall's Lemma that 
\begin{align}\label{Gro.mom.x}
\int_{\R^6}|x|^s f(x,v,t)\dd x\dd v\leq e^{st} \int_{\R^6}|x|^s f_{in}(x,v)\dd x\dd v
+\int_0^t e^{s(t-\tau)}\int_{\R^6}|v|^s f(x,v,\tau)\dd x\dd v\dd\tau. 
\end{align}
In particular, $\int_{\R^6}|x|^\mm f(x,v,t)\dd x\dd v$ can be controlled once $\int_{\R^6}|v|^\mm f(x,v,t)\dd x\dd v$ is bounded. For this reason we turn our attention to 
\begin{align}\label{dt.mom}
\frac{d}{dt}\int_{\R^6}|v|^s f(x,v,t)\dd x\dd v = 
	\int_{\R^6}|v|^s Q(f,f)\dd x\dd v.
\end{align}
Let us now denote, for the sake of brevity
\begin{align*}
M_s(t) :=\int_{\R^6}|v|^s f(x,v,t)\dd x\dd v.
\end{align*}
Let us compute the right-hand side of \eqref{dt.mom}:
\begin{align*}
	&\int_{\R^6}\frac{|v|^s}{s} Q(f,f)\dd x\dd v \\
	&= 
	-\int_{\R^{12}}\kappa(x-y)
	|v|^{s-2}v\cdot N(v-w)(f(y,w)\nabla_{v}f(x,v) - f(x,v)\nabla_{w}f(y,w))\dd x\dd y\dd v\dd w. 
\end{align*}
A simple exchange of variable and integration by parts yield 
\begin{align}
	&\int_{\R^6}\frac{|v|^s}{s} Q(f,f)\dd x\dd v  \nonumber\\
	&= 
	-\int_{\R^{12}}\kappa(x-y)
	(|v|^{s-2}v - |w|^{s-2}w)\cdot N(v-w)f(y,w)\nabla_{v}f(x,v)\dd x\dd y\dd v\dd w \nonumber \\
	&= 
	\int_{\R^{12}}\kappa(x-y)\Div_v\left(
	N(v-w)(|v|^{s-2}v - |w|^{s-2}w)
	\right)
	f(y,w)f(x,v)\dd x\dd y\dd v\dd w. \label{Qnonso}
\end{align}
Note that (\ref{Qnonso}) is identically zero if $s=2$. Therefore 
\begin{align}
   \sup_{t\in [0,T]} \int (|v|^2+ |x|^2) f \; dxdv \le C(T). \label{sec_mom}
\end{align}
We focus our attention on the term
\begin{align*}
	\Div_v &\left(N(v-w)(|v|^{s-2}v - |w|^{s-2}w)\right)\\
	&=(\Div_v N(v-w))\cdot (|v|^{s-2}v - |w|^{s-2}w) + 
	\mbox{Tr}(N(v-w)\nabla_{v}(|v|^{s-2}v))\\
    &=:I_1+ I_2.
\end{align*}
The term $I_1$ is negative: recall that 
\begin{align*}
	\Div_v N(v-w) = -2|v-w|^{\gamma}(v-w),
\end{align*}
 and therefore the monotonicity of the function $|v|^{s-2}v$ yields
$$
I_1 = -2|v-w|^{\gamma}(v-w)\cdot (|v|^{s-2}v - |w|^{s-2}w)\leq 0.
$$
Let us consider now $I_2$. Since 
\begin{align*}
	\nabla_{v}(|v|^{s-2}v) = 
	|v|^{s-2}\left( \eye + (s-2)\frac{v\otimes v}{|v|^2} \right),
\end{align*}
we rewrite 
\begin{align*}
I_2 = &|v|^{s-2}\mbox{Tr}(N(v-w))  + (s-2)|v|^{s-2}\mbox{Tr}\left(N(v-w)\frac{v\otimes v}{|v|^2} \right)\\
= & \; 2 |v|^{s-2}|v-w|^{\gamma+2}+ (s-2)|v|^{s-2}\;  \frac{v}{|v|}\cdot N(v-w)\frac{v}{|v|}.
    \end{align*}
Since $\frac{v}{|v|}$ has modulus 1, 
$\frac{v}{|v|}\cdot N(v-w)\frac{v}{|v|}\leq \mbox{Tr}(N(v-w)) = 2 |v-w|^{\gamma+2}$ and therefore 
\begin{align*}
I_2 \le  & \; 2 |v|^{s-2}|v-w|^{\gamma+2}+ 2 (s-2)|v|^{s-2}|v-w|^{\gamma+2}\\
 = & \; 2 (s-1)|v|^{s-2}|v-w|^{\gamma+2}.
\end{align*}
Substituting the estimates for $I_1$ and $I_2$ in (\ref{Qnonso}) we obtain 
\begin{align*}
	&\int_{\R^6}\frac{|v|^s}{s} Q(f,f)\dd x\dd v 
	\leq 2(s-1)\int_{\R^{12}}\kappa(x-y)|v|^{s-2}|v-w|^{\gamma+2}
	f(y,w)f(x,v)\dd x\dd y\dd v\dd w	.
\end{align*}
The combination of the above estimate with  \eqref{dt.mom} yields
\begin{align}\label{Ms.dt}
\frac{d}{dt}M_s(t)\leq C\int_{\R^6} |v|^{s-2}f(x,v)\left[
	\int_{\R^6}\kappa(x-y)|v-w|^{\gamma+2}
	f(y,w)\dd y\dd w
	\right]\dd x\dd v
\end{align}
for $s=2,\ldots,\mm$.
We distinguish now two cases according to $\gamma$.\smallskip \\
{\em{Case 1:}} $ -2 \le \gamma \le 0$. 
Since $0\le \gamma+2\leq 2$,  the integral in square brackets inside \eqref{Ms.dt} is bounded by the mass and second moment of $f$ (which are bounded thanks to (\ref{cons_laws})), and therefore
\begin{align*}
\frac{d}{dt}M_s(t)
& \le C_{in} \left(1+ M_s(t)\right)
\end{align*}
for $s=2,\ldots,\mm$. In particular \eqref{mom.high} holds for case {\em{(i)}}.
\smallskip \\ 
{\em{Case 2:}} $ \gamma < -2$. 

We first note that 
\begin{align}
    \int_{\R^3} |v-w|^{\gamma+2}
	F(w)\dd w \le & \|F\|_{L^1(\mathbb{R}^3)} + \|F\|_{L^3(\mathbb{R}^3)}\||v|^{\gamma+2}\|_{L^{\frac{3}{2}}(B_v(0,1))} \nonumber\\
   \lesssim &  \|F\|_{L^1(\mathbb{R}^3)}  + \|F\|_{L^3(\mathbb{R}^3)}. \label{est_inter_F3}
\end{align}
Let us consider the right-hand side of \eqref{Ms.dt}:
\begin{align*}
\int_{\R^6} &|v|^{s-2}f(x,v)\left[
	\int_{\R^6}\kappa(x-y) |v-w|^{\gamma+2}
	f(y,w)\dd y\dd w
	\right]\dd x\dd v\\ 
   \le & \|\kappa\|_{L^\infty}\int_{\R^6} |v|^{s-2}f(x,v)\left[
	\int_{\R^3} |v-w|^{\gamma+2}
	F(w)\dd w
	\right]\dd x\dd v\\
    \lesssim &  \|\kappa\|_{L^\infty} \left(\|F\|_{L^1(\mathbb{R}^3)}  + \|F\|_{L^3(\mathbb{R}^3)}\right)\int_{\R^6} |v|^{s-2}f(x,v)\dd x\dd v    \\
    \le &  c_{\|\kappa\|_{L^\infty} , \rho_0, \|F\|_{L^3(\mathbb{R}^3)}}M_{s-2}(t),
% \leq & \int_{\R^{12}\cap\{ |v-w|\geq 1 \}}
% \kappa(x-y)|v|^{s-2}|v-w|^{\gamma+2}f(y,w)f(x,v)\dd v\dd w\dd x\dd y\\ 
% %
% &+\int_{\R^{12}\cap\{ |v-w|< 1 \}}
% \kappa(x-y)|v|^{s-2}|v-w|^{\gamma+2}f(y,w)f(x,v)\dd v\dd w\dd x\dd y. 
% \end{align*}
% The integral outside the ball can be easily estimated as 
% \begin{align*}
% \int_{\R^{12}\cap\{ |v-w|\geq 1 \}}
% \kappa(x-y) |v|^{s-2}|v-w|^{\gamma+2}f(y,w)f(x,v)\dd v\dd w\dd x\dd y \leq \kappa_2 M_0 M_{s-2}.
\end{align*}
using (\ref{est_inter_F3}) and (\ref{condFL3}). Thus, we get the inequality 
$$
\frac{d}{dt}M_s(t) \le  c_{\|\kappa\|_{L^\infty} , \rho_0, \|F\|_{L^3(\mathbb{R}^3)}} M_{s-2}(t). 
$$

\end{proof}

\vspace{0.3cm}
\subsection{Bounds for $H[f]$ }

We now establish several critical inequalities for the spatial average of $f(x,v,t)$:
\begin{align*}
H[f](x,v) := \int\kappa(x-x^*)f(x^*,v)\;d x^*. 
%\label{average_H}
\end{align*}
These inequalities will be used extensively in subsequent sections to bound $a[f]$, $A[f]$ and their derivatives. Recall in fact that $A[f]$ defined in (\ref{Af}) can be rewritten as 
$$
A[f](x,v) = \int_{\R^6}N(v-w)H[f](x,w)\; \dd w,
$$
and 
$$
\nabla a[f]:= \textrm{div}_v A[f]= b[f].
$$
%$$
%a[f](x,v) = \int_{\R^6}\textrm{Tr}\; N(v-w) H[f](x,w)\; \dd w= 2 H[f]\ast_v |v|^{\gamma+2}. 
%$$

The next lemma and its two corollaries contain  general inequalities that relate the regularity of $H[f]$ with the regularity of $f$. These inequalities depend on $v$ and $x$ in a controllable way.

\begin{lemma}\label{lem.est.a}
Let $H[f]$ be defined as above with $f\ge 0$ and $r\in (0,3)$. Assume \eqref{kappa} holds. Let $\alpha\in\N^3$ be a multi-index of rank $k\geq 0$. Then for every $\epsilon>0$ there exists a constant $C(\epsilon)>0$ depending on $\epsilon$, $r$, $\alpha$ such that for all $x,v\in\R^3$ the following inequality holds: % exist such that
\begin{align}
\nonumber
&\left|D_v^\alpha\int_{\R^3}\frac{H[f](x,w)}{|v-w|^r}\dd w\right|\\ 
\label{DH.est.1}
&\leq
\frac{\kappa_2}{\bk{v}^r\bk{x}^\lambda}
\Big(
\epsilon
\|\bk{y}^{\lambda + \sigma} D_v^\alpha f\|_{L^p(\R^6)} 
+ C(\epsilon)\int_{\R^6}
\bk{w}^{\frac{3(p-1)(r+|\alpha|)}{(3-r)p-3}}
%\left(\bk{w}^{\frac{3(p-1)(r+|\alpha|)}{(3-r)p-3}}+\bk{y}^{ \frac{\lambda(p-1)(3-rq)}{r(p+q-pq)} }\right)
\bk{y}^{\lambda}f(y,w)\dd y\dd w
\Big),
\end{align}
for any arbitrary choice of $p > \frac{3}{3-r}$ and  $\sigma > \frac{3(p-1)}{p}$.
\end{lemma}
\begin{proof}
We first prove the result for $\epsilon=1$.
Let
$$ q  \in \left(\frac{3}{\sigma}\,,\, \frac{3}{r}\right), \quad q' = \frac{q}{q-1}, \quad \omega := \frac{rq}{3-rq}, \quad R = 2^{\frac{1}{1+\omega}}.$$ 
We introduce the following cutoff functions:
\begin{align*}
\overline{\eta}\in C^\infty_c(B_2),\quad\overline{\eta}\equiv 1~~\mbox{on }B_1,\quad 
0\leq\overline{\eta}\leq 1~~\mbox{on }\R^3,\\
\eta_v(w)\equiv \overline{\eta}(4|v|^{-1}w),\quad 
\xi_v(w)\equiv \overline{\eta}(4|v|^\omega w),\quad v\neq 0,~~w\in\R^3.
\end{align*}
Case 1: $|v|\leq R$.  H\"older inequality with exponents $\{q,q'\}$, $q,q'<+\infty$, yields
\begin{align}
\left| D_v^\alpha
\int_{\R^3}\frac{H(x,w)}{|v-w|^r}\dd w\right| &=
\left| \int_{\R^3}{D_v^\alpha H(x,v-w)}\frac{\overline{\eta}(w)}{{|w|^r}}\dd w + 
\int_{\R^3}{H(x,w)}D_v^\alpha \frac{(1-\overline{\eta}(v-w))}{|v-w|^r}
\dd w\right| \nonumber  \\
{} \nonumber \\
&\lesssim \|D_v^\alpha H(x,\cdot)\|_{L^{q'}_v(\R^3)} + \|H(x,\cdot)\|_{L^{1}_v(\R^3)}. \label{est.H.lem0} 
\end{align}

Case 2: $|v|\geq R$. 
We estimate the integral after dividing it in three domains via cutoffs:
\begin{align}\nonumber
D_v^\alpha \int_{\R^3} \frac{H(x,w)}{|v-w|^r}\dd w
=& \; 
D_v^\alpha\int_{\R^3}\frac{H(x,w)}{|v-w|^r}
\xi_v(v-w)\eta_v(v-w)\dd w \\ \nonumber
&+ D_v^\alpha\int_{\R^3}\frac{H(x,w)}{|v-w|^r}
(1-\xi_v(v-w))\eta_v(v-w)\dd w \\ \nonumber
&+ D_v^\alpha\int_{\R^3}\frac{H(x,w)}{|v-w|^r}
(1-\eta_v(v-w))\dd w \\
=&:  J_1+J_2+J_3. \label{est.H.lem}
\end{align}
We first estimate $J_3$:
\begin{align*}
|J_3|\leq \|H(x,\cdot)\|_{L^1_v(\R^3)}
\sup_{w\in\R^3}| D_v^\alpha(|v-w|^{-r}(1-\eta_v(v-w))|.
\end{align*}
Notice that, for $\beta\in\N^3$, $|\beta|\leq|\alpha|=k$, the quantity
$$
D_v^\beta(1-\eta_v(v-w)) = D_v^\beta(1-\overline{\eta}(4|v|^{-1}(v-w)))
$$
is bounded and supported outside the ball (in $w$) of center $v$ and radius $|v|/2$, since we are in the case $|v|>R$ and the
$v-$derivatives of $\overline{\eta}(4|v|^{-1}(v-w))$ are not zero only when $|v|/4\leq |v-w|\leq |v|/2$. It follows
\begin{align}\label{lem.J3.est}
|J_3|\lesssim |v|^{-r} \|H(x,\cdot)\|_{L^1_v(\R^3)}.
\end{align}
We consider now $J_1$. Since $|v|>R = 2^{\frac{1}{1+\omega}}$, we have $|v|^{-\omega}\leq |v|/2$ and therefore $\xi_v(v-w)\eta_v(v-w)=\xi_v(v-w)$, hence
\begin{align*}
J_1=\sum_{|\beta|\leq|\alpha|}
\begin{pmatrix}
\alpha\\ \beta
\end{pmatrix}
\int_{\R^3} |w|^{-r}D_v^\beta H(x,v-w)\, D_v^{\alpha-\beta}\xi_v(w)
\dd w .
\end{align*}
Once again, since we are in the case $|v|>R$ and
the $v-$derivatives of $\xi_v(w)$ are not zero only for $|v|^{-\omega}/4<|w|<|v|^{-\omega}/2$, we infer that all $v-$derivatives appearing in the above integral are bounded. Hence
\begin{align*}
|J_1| &\lesssim 
\sum_{|\beta|\leq|\alpha|}
\int_{\{|w|\leq |v|^{-\omega}\} }\frac{D_v^\beta H(x,v-w)}{|w|^r}\dd w\\
&\lesssim
\sum_{|\beta|\leq|\alpha|}\|D_v^\beta H(x,\cdot)\|_{L^{q'}}\left( 4\pi\int_0^{|v|^{-\omega}} \rho^{2-rq}\dd\rho \right)^{\frac{1}{q}},
\end{align*}
that is
\begin{align}\label{lem.J1.est}
|J_1| \lesssim \sum_{|\beta|\leq|\alpha|}\|D_v^\beta H(x,\cdot)\|_{L^{q'}}
|v|^{-r},
\end{align}
using that $\omega = \frac{rq}{3-rq}$.

Finally let us consider $J_2$:
\begin{align*}
J_2 &= \int_{\R^3}H(x,w)D_v^\alpha\left(
\frac{(1-\xi_v(v-w))\eta_v(v-w)}{|v-w|^r}\right)\dd w.
\end{align*}
It holds
\begin{align*}
&\left| 
D_v^\alpha\left(
\frac{(1-\xi_v(v-w))\eta_v(v-w)}{|v-w|^r}\right) \right|\\
&= \left| \sum_{|\beta|\leq|\alpha|}
\begin{pmatrix}
\alpha\\ \beta
\end{pmatrix}D_v^\beta(|v-w|^{-r})\, 
D_v^{\alpha-\beta}((1-\xi_v(v-w))\eta_v(v-w))\right| \\ 
&\lesssim\sum_{|\beta|\leq|\alpha|}
|v-w|^{-r-|\beta|}\left| 
D_v^{\alpha-\beta}((1-\xi_v(v-w))\eta_v(v-w))\right| .
\end{align*}
The integrand in $J_2$ is not zero only on  
$$\left\{w\in\R^3~~:~~ 
\frac{1}{4|v|^{\omega}}\le |v-w|\leq \frac{|v|}{2}
\right\}
$$ implying
$$
\frac{|v|}{2}\le |w| \quad \textrm{and} \quad 
 \frac{1}{|v-w|}\leq 4|v|^{\omega} = 4\frac{|v|^{(\omega+1)}}{|v|}\lesssim
\frac{|w|^{(\omega+1)}}{|v|}=\frac{|w|^{ \frac{3}{3-rq} }}{|v|}. 
$$
Given that $|v|\ge R$, for $|\beta|\leq|\alpha|=k$ it follows
$$
\frac{1}{|v-w|^{r+|\beta|}}\lesssim 
\frac{|w|^{ \frac{3(r+k)}{3-rq} }}{|v|^r}.
$$
In the same way as before one shows that 
$$\left| 
D_v^{\alpha-\beta}((1-\xi_v(v-w))\eta_v(v-w))\right| $$
is bounded. Therefore  
\begin{align}\label{lem.J2.est}
|J_2|\lesssim 
|v|^{-r}\int_{\R^3} |w|^{ \frac{3(r+k)}{3-rq} } H(x,w)\dd w.
\end{align}
Combining (\ref{est.H.lem0})--(\ref{lem.J2.est}) we get 
\begin{align}\label{est.H.lem.2}
\left| D_v^\alpha
\int_{\R^3}\frac{H(x,w)}{|v-w|^r}\dd w\right|
\leq \frac{C}{\bk{v}^r}\left(
\|D_v^\alpha H(x,\cdot)\|_{L^{q'}_v(\R^3)} 
+ \int_{\R^3}\left(1 + |w|^{\frac{3(r+k)}{3-rq}}\right)H(x,w)\dd w
\right),
\end{align}
where we used a standard Gagliardo-Nirenberg-Sobolev inequality to estimate the terms involving the lower order $v-$derivatives of $H$ via the $L^{q'}$ norm of $D_v^\alpha H$ and the mass of $H$.

We estimate now $\|D_v^\alpha H(x,\cdot)\|_{L^{q'}_v(\R^3)}$ in terms of norms of $f$. Bochner integral inequality and H\"older inequality yield
\begin{align}\nonumber
\bk{x}^\lambda\|D_v^\alpha H(x,\cdot)\|_{L^{q'}_v(\R^3)} &\lesssim 
\int_{\R^3}\bk{x}^\lambda\bk{x-y}^{-\lambda}
\|D_v^\alpha f(y,\cdot)\|_{L^{q'}_v(\R^3)}\dd y\\ \label{est.1st.way}
&\lesssim \int_{\R^3}\bk{y}^\lambda
\|D_v^\alpha f(y,\cdot)\|_{L^{q'}_v(\R^3)}\dd y\\ \nonumber
&=\int_{\R^3}\bk{y}^{ -\sigma } \bk{y}^{\lambda + \sigma}\|D_v^\alpha f(y,\cdot)\|_{L^{q'}_v(\R^3)}\dd y\\ \nonumber
&\leq \left( 
\int_{\R^3}\bk{y}^{ -\sigma q }\dd y
\right)^{1/q}\left(
\int_{\R^6}\bk{y}^{\lambda q' + \sigma  q'}
|D_v^\alpha f|^{q'}\dd v\dd y
\right)^{1-1/q}
\end{align}
for $ \sigma>0$. The first integral on the right-hand side is convergent since $q > \frac{3}{\sigma}$. The choice of $q = \frac{p}{p-1}$ yields the statement for $\epsilon=1$.

To show the statement for a generic $\epsilon>0$ it suffices to repeat the above argument with the cutoff $\overline{\eta}$ replaced by $\overline{\eta}_\epsilon(x):=\overline{\eta}(x/\epsilon)$.

This finishes the proof of the Lemma.
\end{proof}
The case $\alpha =0$ is summarized in the following corollary: 
\begin{cor}\label{coro.DH.est.1b}
Same assumptions as Lemma \ref{lem.est.a}.
Let $\alpha = 0$ (hence $k=0$). For every $\epsilon>0$
there exists a constant $C(\epsilon)>0$ depending on $\epsilon$, $r$ such that for $x,v\in\R^3$ the following inequality holds:
\begin{align}
&\int_{\R^3} \frac{H(x,w)}{|v-w|^r}\dd w\nonumber\\
&\leq \frac{\kappa_2}{\bk{v}^r\bk{x}^\lambda}\left(
\epsilon\|f\|_{L^{ \frac{p-\theta}{1-\theta} }(\R^6)} 
+ C(\epsilon)\int_{\R^6}\left( 
1 + |y|^{ \frac{p(\lambda+\sigma)}{\theta} } 
+ |w|^{\frac{3(p-1)r}{(3-r)p-3}
{\frac{p}{p-\theta}}
} \right)
f(y,w)\dd y\dd w
\right),\label{DH.est.1b}
\end{align}
for any arbitrary choice of $\theta\in (0,1)$, $p > \frac{3}{3-r}$ and $\sigma>0$ such that $ \sigma > \frac{3(p-1)}{p}$.
\end{cor}
\begin{proof}
The claim follows immediately via H\"older {and Young inequalities} applied to \eqref{DH.est.1}.
\end{proof}

\subsection{Coercivity}
The lower bound on the diffusion is stated next. We emphasize the assumptions on $\kappa$ are crucial here, for they  guarantee a {\bf{non trivial diffusion in velocity \emph{for any} $x$}} as long as $f$ has positive mass \emph{somewhere} in the domain. Put differently, as collisions can occur between distant particles, the diffusive effect of collisions is felt everywhere instantaneously, even in vacuum regions. Analytically, this instantaneous diffusion causes the delocalized inhomogeneous equation to behave much more like a kinetic equation with linear diffusion. The main coercivity result is stated in the following lemma.

\begin{lemma}\label{lem.aA.lb}
Let $f\ge 0$ be a function that has finite mass, entropy and moments of order $\mm$ with {{$\mm \ge 4+ 2\lambda$}}:
\begin{align*}
    \rho_0 & = \int_{\R^6}f(x,v) \; \dd x\dd v, \\
	c_{\mm} &\ge  \int_{\R^6}(|x|^\mm + |v|^\mm)f(x,v)\; \dd x\dd v,\\
	h_0 & \ge\int_{\R^6} f\log f \; \dd x\dd v.
\end{align*}

There exists a constant $\theta>0$ that only depends on $\gamma$, $\rho_0$, $c_{\mm}$ and $ h_0$ such that 
	
    \begin{align}\label{coerc_lem1}
    \xi\cdot A[f](x,v)\xi \geq \theta \;  \frac{\bk{v}^{\gamma}}{\bk{x}^{\lambda}} \; |\xi|^2,\qquad\forall \; \xi, \; v, \; x\in\R^3, 
    \end{align}
    where 
    $$
    A[f](x,v) := \int_{\R^6}N(v-v^*)\kappa(x-x^*)f(x^*,v^*)\dd x^*\dd v^*.
    $$
\end{lemma}
Even though the coercivity (\ref{coerc_lem1}) is stated in terms of the function $f$, the objects of our calculations will be {\em{averages of $f$ in the physical space}}. After showing that these averages of $f$ satisfy certain bounds, the proof of Lemma \ref{lem.aA.lb} is mainly inspired by the one of the homogeneous case.
%We will, however, not work with $f$, but with the average of $f$ in the physical space.  

There are several preparatory lemmas.  We start by recalling a classical statement (see for example, \cite[Section 1.1]{de15}), followed by  a technical lemma. 
\begin{rem}\label{rem:entr-bound}
    Let  $g (x)\ge 0$.   There exists a positive constant $C$ that only depends on $\int g \ln g \; dx$, $\int g\;dx$ and $\int g|x|^2 \;dx$ such that 
    $$
    \int g | \ln g| \; dx \le C. 
    $$
\end{rem} 
The next technical lemma relates the entropy of a function to the entropy of the average of such function: 

\begin{lemma}\label{lem:lifting_entropy}
Let $g(x,y)\ge 0$ with $\int \int g(x,y) \;dxdy = \theta_0 >0 $.   Denote with 
$$G(x):= \int g(x,y)\;dy, %\qquad G(y):= \int g(x,y)\;dx.
$$ 
and let $M(y)$ be a positive function such that $ \int M(y) \;dy =\theta_0$.

The entropy of $g$ and $G$ satisfy the following relation: 
{{\begin{align*}
     \int G(x) \ln G(x) \;dx +  \int \int g(x,y) \ln M(y) \;dxdy\le \int \int g(x,y) \ln g(x,y)\;dxdy + \theta_0\ln \theta_0 .
\end{align*}}}
\end{lemma} 
\begin{proof}
    Consider 
    \begin{align*}
\int g(x,y) \ln\left(\frac{g(x,y)}{M(y)}\right)\;dy  =&  ~ \theta_0\int \frac{g(x,y) }{M(y)}\ln\left(\frac{g(x,y)}{M(y)}\right)\; \frac{M(y)}{\theta_0}dy  \\
\ge & ~\theta_0\int \frac{g(x,y)}{M(y)}\frac{M(y)}{\theta_0}dy \ln \int\frac{g(x,y)}{M(y)}\frac{M(y)}{\theta_0}dy ,
         \end{align*}
    after applying Jensen's inequality with the function  $ z \ln z$ and the measure $\frac{M(y)}{\theta_0}dy$. After simplification one gets 
    $$
    G(x) \ln \frac{G(x)}{\theta_0} \le \int g(x,y) \ln\left(\frac{g(x,y)}{M(y)}\right)\;dy  = \int g(x,y) \ln g(x,y) \;dy  - \int g(x,y) \ln M(y) \;dy.
    $$
    Integrate in $dx$ and get 
    $$
    \int G(x) \ln {G(x)} \;dx - \theta_0 \ln \theta_0 \le \int \int g(x,y) \ln g(x,y) \;dy \;dx - \int\int g(x,y) \ln M(y)\;dydx, 
    $$
    which finishes the proof. 
\end{proof}

{
With Remark \ref{rem:entr-bound} and Lemma \ref{lem:lifting_entropy} in mind, we now turn our attention to averages of $f$, and consider the following functions: 
\begin{align}
 H[f](x,v) &:= \int_{\R^3}\kappa(x-y)f(y,v) \; \dd y, \nonumber  \\
 F(v) &:= \int_{\R^3}f(x,v)\; \dd x, \label{eq:def H}\\
 \rho(x) &:= \int_{\R^3}f(x,v)\; \dd v. \nonumber
\end{align}

We recall that throughout  this section  $f$ satisfies the hypothesis of Lemma \ref{lem.aA.lb}. An immediate consequence of Lemma \ref{lem:lifting_entropy} is the bound for the entropy for $F$, $\rho$ and $H[f]$: 
\begin{cor}
There exist positive constants that only depend on $\rho_0$, $h_0$, and $\| \kappa\|_{L^\infty}$ such that 
 \begin{align}
& {{\int F \ln F \; dv , \int \rho \ln \rho \; dx \le C_{h_0,c_2,\rho_0}}}, \label{ent_F}\\
 %& \int \rho \ln \rho \; dx \le C_{\rho_0, h_0},\label{ent_Rho}\\
 \sup_{ x\in \mathbb{R}^3} & \int H[f] \ln H[f] \; dv \le C_{ \|\kappa\|_{L^\infty},\rho_0, h_0,c_{\mm}}.  \label{ent_H}
 \end{align}
\end{cor}
\begin{proof}
To show the first inequality use Lemma \ref{lem:lifting_entropy} with respectively $g=f$, $G=F$ and $M(x) = \rho_0 \pi^{-3/2}e^{-|x|^2}$ to obtain
$$
 \int F(v) \ln F(v) \;dv \le \int \int f(x,v) \ln f(x,v)\;dxdv + \rho_0\ln \rho_0 + \int \int f(x,v) |x|^2 \;dxdy \le C_{h_0,c_2,\rho_0}
$$
and $g=f$, $G=\rho$ and $M(v) = e^{-|v|^2}$ to obtain
$$
 \int \rho(x) \ln \rho(x) \;dx \le \int \int f(x,v) \ln f(x,v)\;dxdv + \rho_0\ln \rho_0 + \int \int f(x,v) |v|^2 \;dxdy\le C_{h_0,c_2,\rho_0}
$$

For the third inequality we use Lemma \ref{lem:lifting_entropy} with 
$$g(x,y,v) = \kappa(x-y) f(y,v),\quad \theta_0(x) = \iint \kappa(x-y) f(y,v) \; dydv,$$ and get 
\begin{align*}
    \int H[f] \ln H[f] \;dv \;\le  & \iint \kappa(x-y) f(y,v) \ln [\kappa(x-y) f(y,v)] \; dydv \\
   &  + \iint \kappa(x-y) f(y,v) \; dydv  \ln\left(\iint \kappa(x-y) f(y,v) \; dydv\right)\\
   & - \iint \kappa(x-y) f(y,v) \ln M(y) \; dydv\\
    =: & \;  I_1 + I_2 + I_3. 
\end{align*}
Note that 
\begin{align*}
    I_1  = & \iint \kappa(x-y) f(y,v) \ln  f(y,v) \; dydv + \iint \kappa(x-y) f(y,v) \ln [\kappa(x-y)] \; dydv\\
    \le &  \; \| \kappa\|_{L^\infty}\iint  f(y,v) |\ln  f(y,v)| \; dydv + \| \kappa\|_{L^\infty} \ln(\max\{ 1, \| \kappa\|_{L^\infty}\}) \; \rho_0 \\
    \le & \;  C_{\| \kappa\|_{L^\infty}, \rho_0, c_{\mm}, h_0},
\end{align*}
$$
I_2 \le \| \kappa\|_{L^\infty} \rho_0 \ln (\| \kappa\|_{L^\infty} \rho_0).
$$
and 
$$
I_3\le \| \kappa\|_{L^\infty} c_2.
$$
 
\end{proof}

}

The next lemma shows the existence  of a set in $\mathbb{R}^3 \times  \mathbb{R}^3$ with non-trivial measure in which $f$ is uniformly bounded from below. The proof is inspired by the one of the homogeneous case \cite{DesVil2000a, Silvestre2015}:

\begin{lemma}\label{lem.aA.1}
	There exist constants $\ell$, $R$, $\TT$ depending on $c_2$, $\rho_0$, $h_0$ such that
	$$
	\meas{\{f\geq\ell\}\cap B_{R}^{x,v}}\geq \frac{5\rho_0}{8\TT}.
	$$
\end{lemma}
\begin{proof}
	Fix $R\geq 0$. The estimate on the second moment of $f$ yields
	\begin{align*}
		\int_{\R^6\backslash B_R^{x,v}}f\dd x\dd v \leq R^{-2}c_2. 
	\end{align*}
	Fix $\ell\geq 0$ and consider
	\begin{align*}
		\int_{B_R^{x,v}\cap\{f\leq\ell\}}f\dd x\dd v\leq\ell ~ \meas{B_R^{x,v}}\leq \omega_6\ell R^6,
	\end{align*}
	where $\omega_6$ is the measure of the unit ball in $\R^6$.
	Fix $\TT\geq\max\{1,\ell\}$. Then we have
	\begin{align*}
		\int_{f\geq \TT}f\dd v\dd x \leq \frac{1}{\log\TT}\int_{f\geq \TT}f\log f\dd v\dd x
		 \leq \frac{1}{\log\TT}\int_{\R^6}f|\log f|\dd v\dd x
		\leq \frac{C(h_0,c_2)}{\log\TT}.
	\end{align*}
    For the last inequality, see Remark \ref{rem:entr-bound}.
	Make the following choices of parameters:
    \begin{equation} \label{eq:le-meas-par}
		R = \sqrt{\frac{8 c_2}{\rho_0}},\quad 
		\ell = \frac{\rho_0}{8\omega_6 R^6},\quad 
		 \TT = \exp\left(\frac{8 C(h_0,c_2)}{\rho_0}\right).
	\end{equation}
    With these choices of parameters one obtains
{
    \begin{align*}
		\int_{B_R^{x,v}\cap\{\ell\leq f\leq\TT\}}f\dd x\dd v
        & = \int_{\mathbb{R}^6} f\dd x\dd v
        - \int_{(B_R^{x,v})^c\cup\{\ell\leq f\}^c \cup (\{f \leq\TT\})^c}f\dd x\dd v\\
        & = \int_{\mathbb{R}^6} f\dd x\dd v
        - \int_{(B_R^{x,v})^c~\cup~\{f < \ell\} ~\cup~ (\{f >\TT\})}f\dd x\dd v\\
         & = \int_{\mathbb{R}^6} f\dd x\dd v
        - \int_{(B_R^{x,v})^c~\cup~(B_R^{x,v} \cap \{f < \ell\}) ~\cup ~(\{f >\TT\})}f\dd x\dd v\\
        & \ge \rho_0 - \frac{3}{8}\rho_0 = \frac{5}{8}\rho_0,
    \end{align*}
    } 
    which implies
	\begin{align*}
		\frac{5\rho_0}{8\TT}\leq \frac{1}{\TT}\int_{B_R^{x,v}\cap\{\ell\leq f\leq\TT\}}f\dd x\dd v\leq \meas{ B_R^{x,v}\cap\{\ell\leq f\leq\TT\} }\leq 
		\meas{ B_R^{x,v}\cap\{\ell\leq f\} }.
	\end{align*}

    This finishes the proof of the lemma.
\end{proof}

A consequence of the previous lemma is a bound on the mass of $H[f]$ pointwise in $x$. This bound depends on $x$ in a controllable way.  This, and other estimates on $H[f]$, is the content of the next lemma:

\begin{lemma}\label{lem.aA.2}
	Let ${\mm\ge 2 + {\lambda}}$. Then there exist constants only dependent on $\rho_0$, $c_{\mm}$, $h_0$, and $\kappa$ such that: 
	\begin{align*}
    & (i) \quad  C_{\rho_0, c_2,h_0}\; \bk{x}^{-\lambda} \le    \int_{\R^3}H[f](x,v)\dd v \le \|\kappa\|_\infty \rho_0,\\
	& (ii) \quad ~H[f](x,v)  \leq C_{\lambda, \kappa_2}\bk{x}^{-\lambda}\int_{\R^3}\bk{y}^{\lambda}f(y,v)\dd y,\\
	& (iii) \quad \int_{\R^3}\bk{v}^2 H[f](x,v)\dd v \leq C_{\rho_0,c_\mm}\bk{x}^{-\lambda},\\
	& (iv)  \quad  \int_{\R^3}H[f]|\log H[f]|\dd v \leq C_{\rho_0,c_\mm, h_0}.
	\end{align*}
\end{lemma}

\begin{proof} Let us show ($i$).
	{Let $\ell$ and $R$ be as in Lemma \ref{lem.aA.1}. Then we have
	\begin{align*}
		\| \kappa \|_{\infty} \rho_0 \ge  \int H(x,v)\;dv= \int_{\R^6}\kappa(x-y)f(y,v)\dd y\dd v
		\geq \ell\int_{ \{ f\geq\ell \}\cap B_R^{y,v} }\kappa(x-y)\dd y\dd v.
	\end{align*}
     To estimate the bound from below, we distinguish two cases: $|x|\le 4R$ and $|x|\ge 4R$. In the set $\{|x|\le 4R\} \cap  B_R^{y,v} $ the kernel   $\kappa(x-y)$ is bounded from below by a constant that only depends on $R$. Hence for these $x$  we have
    \begin{align}\label{stima11}
		\int H(x,v)\;dv \ge & \; \ell\int_{ \{ f\geq\ell \}\cap B_R^{y,v}}\kappa(x-y)\dd y\dd v \\
        \geq & \; c_R \ell \; \meas{
			\{ f\geq\ell \}\cap B_R^{x,v}}  \nonumber\\
            \geq & \;  \tilde{c}_R\ell \langle x \rangle^{-\lambda}\; \meas{
			\{ f\geq\ell \}\cap B_R^{x,v} }. \nonumber
	\end{align}
    For $|x|\ge 4R$ and $(y,v) \in B_R^{y,v}$, we have that 
    $$
    |x-y|^2 \lesssim |x|^2 + R^2, \quad \textrm{and} \quad  \bk{x-y}^2\le C_R \bk{x}^2, 
    $$
    implying 
    \begin{align}\label{stima22}
		\int_{\R^3}H(x,v)\dd v\geq  & \; \ell\int_{ \{ f\geq\ell \}\cap B_R^{y,v} }\kappa(x-y)\dd y\dd v \\
        \ge & \; \tilde{C}_R\ell \bk{x}^{-\lambda}\meas{
			\{ f\geq\ell \}\cap B_R^{x,v} }. \nonumber
        \end{align}
    Summarizing we obtain 
    $$
    \int_{\R^3}H(x,v)\;\dd v\geq  {C}_R\bk{x}^{-\lambda} \frac{5\rho_0\ell}{8\TT},
    $$
	thanks to Lemma \ref{lem.aA.1}, which proves ($i$).
	}
    
    \vspace{2mm}
    
	Let us now prove ($ii$). {From \eqref{kappa} we get
	\begin{align*}
		H(x,v) = \int_{\R^3}\kappa(x-y)f(y,v)\dd y
		\leq \kappa_2\int_{\R^3}\bk{x-y}^{-\lambda} f(y,v)\dd y.
	\end{align*}
	Via the standard inequality $\bk{x}^\lambda\leq C(\lambda)(\bk{x-y}^\lambda + \bk{y}^\lambda)$ we obtain
	\begin{align*}
		\bk{x}^\lambda H(x,v)
        & \lesssim\kappa_2 \int_{\R^3}f(y,v)\dd y + 
		\kappa_2 \int_{\R^3}\bk{x-y}^{-\lambda} \bk{y}^\lambda f(y,v)\dd y\\
        & \le    C_{\kappa} \int_{\R^3} \bk{y}^\lambda f(y,v)\dd y,
	\end{align*}
	}
    where $C_{\kappa}$ depends on $\lambda$ and $\kappa_2$.  This proves ($ii$).
	
	From ($ii$) we obtain
	\begin{align*}
		\bk{x}^{\lambda}\int_{\R^3} & H(x,v)\bk{v}^2 \dd v \leq C_{\kappa}\int_{\R^6}\bk{v}^2\bk{y}^{\lambda} f(y,v)\dd y \dd v\\
		\lesssim  & \; C_{\kappa} \int_{\R^6}(1 + |v|^2 + |y|^2)^\mm \; f(y,v)\dd y\dd v \leq  C(\rho_0, c_{\mm}),
	\end{align*}
	by taking 
    $$
    {\mm\ge 2 + {\lambda}.}
    $$
    Hence ($iii$) follows.

   \vspace{2mm} 

   Inequality ($iv$) follows from ($iii$), (\ref{ent_H}) and Remark \ref{rem:entr-bound}.  This finishes the proof.

\end{proof}

Next we show that $H[f]$ is uniformly bounded from below by a constant in a non-trivial set in the velocity space. This property is fundamental for the proof of (\ref{coerc_lem1}). 
\begin{lemma}\label{lem.aA.3}
Let $\mm\ge 2 + {\lambda}$. 
	Let $H[f]$ satisfy $(i)$, $(ii)$, and $(iii)$ of Lemma \ref{lem.aA.2}. There exist positive constants $\tilde\ell$, $\tilde R$, $\tilde\mu$  dependent only on $\rho_0$, $c_\mm$, $h_0$, $\lambda$, $\gamma$, and $\kappa_2$, such that
	\begin{align*}
		\meas{ 
			v\in\R^3~:~
			H[f](x,v)\geq\tilde\ell \bk{x}^{-\lambda},~~ |v|< {\tilde R} }\geq \tilde\mu,\qquad\mbox{a.e.~$x\in\R^3$.}
	\end{align*}
\end{lemma}
\begin{proof}

We first define a new function 
$$H_\eta(x,v) := \int_{\R^3}\kappa(x-y)\min(f(y,v),
    \eta e^{-|y|^2-|v|^2})\dd y.$$
    Since $H_\eta(x,v) \le H[f](x,v) $ for all $x,v$, it is sufficient to show that 
\begin{align}\label{H1.goal}
		\meas{ 
			v\in\R^3~:~
		    H_\eta(x,v)
            \geq\tilde\ell \bk{x}^{-\lambda},~~ |v|< \tilde{R} }\geq \tilde{\mu}\qquad\mbox{a.e.~$x\in\R^3$},
	\end{align}
    for some $\tilde{\mu}>0$, $\tilde\ell$ and $\tilde R$.

%The estimate of the second moment of $H$ (see $(iii)$ in Lemma \ref{lem.aA.2}) yields

Moreover
	\begin{align}\label{stima111}
	\int_{\R^3\backslash B_{\tilde{R}}^{v}}H_\eta(x,v)\dd v \leq
    \int_{\R^3\backslash B_{\tilde{R}}^{v}}\int_{\R^3}
    \eta \kappa(x-y) e^{-|y|^2-|v|^2}\dd y\dd v 
    \leq C_{\kappa,0}\eta {\tilde{R}}^{-2}\bk{x}^{-\lambda}.
	\end{align}
	For $\tilde{\ell}>0$ to be defined later we have 
	\begin{align}\label{stima112}
		\int_{B_{\tilde{R}}^{v}\cap\{H_\eta\leq \; \tilde{\ell}\; \bk{x}^{-\lambda}\}}H_\eta\dd v\leq \tilde{\ell}\; \bk{x}^{-\lambda} \; \meas{B_{\tilde{R}}^{v}} = 
		\frac{4}{3}\pi  {\tilde{R}}^3\;\tilde\ell\bk{x}^{-\lambda}.
	\end{align}
Next we show that  $H_\eta$ satisfies a lower bound similar to the one in $(i)$, Lemma \ref{lem.aA.2}.
  With $R$ and $\ell$ as in Lemma \ref{lem.aA.1} we have 
    \begin{align*}
     \int_{\R^3}H_\eta(x,v)\dd v  & =\int_{\R^6}\kappa(x-y)\min(f(y,v),\eta e^{-|y|^2-|v|^2})\; \dd y\dd v\\
	&\geq \ell \int_{B_R^{y,v}\cap\{ \min(f(y,v),\eta 
    e^{-|y|^2-|v|^2})\geq\ell \}  }\kappa(x-y)\; \dd y\dd v\\
    &= \ell \int_{ B_R^{y,v}\cap\{ f\geq\ell \}}\kappa(x-y)\; \dd y\dd v,
    \end{align*}
    if we choose $\eta$ such that 
    $$
    \eta e^{-|y|^2-|v|^2}\geq\ell, \quad {\textrm{for}} 
    \; |y|^2+|v|^2\le R^2. 
    $$
    We fix therefore
    $$
    \eta = \ell e^{R^2}.
    $$
    Arguing like in (\ref{stima11}) and (\ref{stima22}, we conclude
    \begin{align}\label{lb.H1}
    \int_{\R^3}H_\eta(x,v)\dd v \ge \ell\int_{ \{ f\geq\ell \}\cap B_R^{y,v}}
    \kappa(x-y)\dd y\dd v\geq {C}_R\bk{x}^{-\lambda} \frac{5\rho_0 \ell}{8\TT}.
    \end{align}

   Combining (\ref{lb.H1}) with  (\ref{stima112}) and (\ref{stima111})we have the following lower bound: 
   \begin{align*}
&\int_{B_{\tilde{R}}^{v}\cap\{H_\eta\geq \; \tilde{\ell}\; \bk{x}^{-\lambda}\}}H_\eta (x,v)\dd v \\
= &\int_{\R^3}H_\eta(x,v)\dd v - \int_{B_{\tilde{R}}^{v}\cap\{H_\eta\leq \; \tilde{\ell}\; \bk{x}^{-\lambda}\}}H_\eta\dd v - \int_{\R^3\backslash B_{\tilde{R}}^{v}}H_\eta(x,v)\dd v \\
\ge &  \; \bk{x}^{-\lambda}\left( {C}_R \frac{5\rho_0 \ell}{8\TT} - \frac{4}{3}\pi  {\tilde{R}}^3\;\tilde\ell 
-C_{\kappa,0}\eta {\tilde{R}}^{-2}\right).
       \end{align*}
We now chose $\tilde R $ and $\tilde \ell$ such that 
       \begin{align}\label{condi-l}
        {C}_R \frac{5\rho_0 \ell}{8\TT} - \frac{4}{3}\pi  {\tilde{R}}^3\;\tilde\ell -C_{\kappa,0}\eta {\tilde{R}}^{-2}
        \ge {C}_R \frac{5\rho_0 \ell}{16\TT} . 
       \end{align}
Lastly, we have the following $L^\infty$ norm for $H_\eta(x,v)$ (arguing like in the proof of Lemma \ref{lem.aA.2} $(ii)$):  
\begin{align*}
H_\eta(x,v)  
\leq C_\kappa \; \bk{x}^{-\lambda}\int_{\R^3}\bk{y}^{\lambda}
\min(f(y,v),\eta e^{-|y|^2})\dd y
\leq \;  \eta\; C_\kappa \; \bk{x}^{-\lambda}.
\end{align*}
Thus we obtain
	\begin{align*}
\frac{{C}_R \frac{5\rho_0 \ell}{16\TT} }{ \eta\; C_\kappa}=& \; \frac{{C}_R \frac{5\rho_0 \ell}{16\TT} \bk{x}^{-\lambda}}{ \eta\; C_\kappa\bk{x}^{-\lambda}}   \\
\le& \; \frac{1}{ \eta\; C_\kappa\bk{x}^{-\lambda}}\int_{B_{\tilde{R}}^{v}\cap\{H_\eta\geq \; \tilde{\ell}\; \bk{x}^{-\lambda}\}}H_\eta(x,v)\dd v\\
\le & \; \meas{ B_{\tilde{R}}^{v}\cap\{H_\eta\geq \; \tilde{\ell}\; \bk{x}^{-\lambda}\}},
	\end{align*}
meaning that \eqref{H1.goal} holds with $\tilde{\mu} := \frac{{C}_R \frac{5\rho_0 \ell}{16\TT} }{ \ell e^{R^2}\; C_\kappa}$ and $\tilde\ell$ and $\tilde R$ satisfying (\ref{condi-l}). This finishes the proof.

\end{proof}

We are now ready to show  (\ref{coerc_lem1}). %Thanks to the previous lemma, we can follow the steps of the homogeneous case. 
\begin{proof}[Proof of Lemma \ref{lem.aA.lb}]
	We start from the definition \eqref{Af} of $A[f]$:
	\begin{align*}
		A[f] &= \int_{\R^3}\kappa(x-y)\int_{\R^3}f(y,v-w)|w|^{\gamma+2}\Pi(w)\dd w\dd y
		= \int_{\R^3}H[f](x,v-w)|w|^{\gamma+2}\Pi(w)\dd w,
	\end{align*}
	where $H[f](x,v) = \int_{\R^3}\kappa(x-y)f(y,v)\dd y$. 
	Let $e\in \R^3$, $|e|=1$ arbitrary. It holds
	\begin{align*}
		e\cdot A[f]e &= \int_{\R^3}H[f](x,v-w)|w|^{\gamma+2}\, e\cdot\Pi(w)e\, \dd w\\
		&\geq \int_{B_R(v)}H[f](x,v-w)|w|^{\gamma+2}e\cdot\Pi(w)e\, \dd w,
	\end{align*}
where $R>0$ is as in Lemma \ref{lem.aA.3}. Recall that $H[f]$ satisfies the properties of Lemma \ref{lem.aA.3} and Lemma \ref{lem.aA.2}.

Let also 
$$\ell(x) := \tilde\ell \bk{x}^{-\lambda}$$ and  $$\mathcal{D}_R(v):=\{w:~|v-w|<R,~ H[f](x,v-w)\geq\ell(x) \}$$ 
for brevity. According to Lemma \ref{lem.aA.3}, we can choose $R$, $\tilde{\ell}, \mu$ such that 
$$
\meas{\mathcal{D}_R(v)}\geq\mu>0.
$$
It holds
\begin{align}\label{inizion_modif}
e\cdot A[f]e &\geq \tilde\ell \bk{x}^{-\lambda}\int_{\mathcal{D}_R(v)}
|w|^{\gamma+2}e\cdot\Pi(w)e\,\dd w .
\end{align}
We observe that, for $w\in \mathcal{D}_R(v)$
	\begin{align*}
		|w|\leq |v-w| + |v|
		\leq R + |v|
        \leq (1+R)\bk{v},
	\end{align*}
hence 
$$|w|^{\gamma}\gtrsim C\bk{v}^{\gamma}.$$ 
The last inequality  leads to
\begin{align}\label{almost-there}
e\cdot A[f]e &\geq \tilde{\ell}\bk{x}^{-\lambda}\bk{v}^{\gamma}\int_{\mathcal{D}_R(v)}
(|w|^{2} - (e\cdot w)^2)\dd w.
\end{align}
Our final goal is to show that
\begin{align}\label{A.lb.goal}
\inf_{(e,v,x)\in \mathbb{S}^2\times\R^3\times\R^3}\int_{\mathcal{D}_R(v)}
(|w|^{2} - (e\cdot w)^2)\dd w>0.
\end{align}
The lemma follows from (\ref{A.lb.goal}) and (\ref{almost-there}) with $\theta$ equal to $\tilde{\ell}$ times the above infimum.

Let us then prove (\ref{A.lb.goal}). We start by noticing that 
\begin{align}\label{A.lb.2}
\int_{\mathcal{D}_R(v)}
(|w|^{2} - (e\cdot w)^2)\dd w
&\geq 
\eta^2 \meas{D_R(v)\backslash X_\eta(e,v)}\\ \nonumber
&\geq \eta^2 \big(\meas{D_R(v)} -  \meas{X_\eta(e,v)} \big)\\ \nonumber
&\geq\eta^2\big(\mu - \meas{X_\eta(e,v)}\big),
\end{align}
with 
\begin{align*}
X_\eta(e,v)\equiv \{ w\in\R^3 ~ : ~ |w|^2 - (w\cdot e)^2 < \eta^2,~~ |w-v|<R \}.
\end{align*}
We point out that
\begin{align*}
X_\eta(e,v) = \mathcal{C}_\eta(e)\cap B_R(v),\quad 
\mathcal{C}_\eta(e) \equiv \{ w\in\R^3 ~ : ~ |w|^2 - (w\cdot e)^2 < \eta^2\}.
\end{align*}
Since $\mathcal{C}_\eta(e)$ is a cylinder centered at $0$ with radius $\eta<R$ and axis $e$, the measure of the above intersection is maximized (for fixed $v\in\R^3\backslash\{0\}$) when $e$ and $v$ are parallel:
\begin{align*}
\meas{X_\eta(e,v)}\leq \meas{X_\eta(v/|v|,v)}\equiv \meas{\tilde X_\eta(v)}\quad\mbox{for }v\in\R^3\backslash\{0\}.
\end{align*}
Assume now $v\neq 0$. Let $w\in \tilde X_\eta(v)$. A change of variable $w = v + \rho\hat{z}$, $|\hat{z}|=1$, $\rho\in [0,R)$ and straightforward computations lead to
	\begin{align*}
		|w|^2 - (w\cdot\hat{v})^2 = \rho^2(1-(\hat{v}\cdot\hat{z})^2).
	\end{align*}
	This means that $\tilde X_\eta(v) = v + Y_\eta(v)$ where
	\begin{align*}
		Y_\eta(v) = \{ z\in\R^3~:~|z|<R,~
		|z|^2(1-(\hat{v}\cdot\hat{z})^2)<\eta^2 \}.
	\end{align*}
	In particular, $\meas{\tilde X_\eta(v)}=\meas{Y_\eta(v)}$. Furthermore, we observe that $\meas{Y_\eta(v)}$ 
   is independent of $|v|$
    and it is invariant under rotations of $\hat{v}$. Therefore, it is  independent of $v$, which implies that
\begin{align*}
\meas{X_\eta(e,v)}\leq \meas{Y_\eta(e_3)}\quad\mbox{for $v\neq 0$, $|e|=1$},
\end{align*}
with $e_3 = (0,0,1)$. On the other hand, if $v=0$ then 
$$ \meas{X_\eta(e,v)} = \meas{\mathcal{C}_\eta(e)\cap B_R}
\le \meas{\mathcal{C}_\eta(e_3)\cap B_R}
= \meas{Y_\eta(e_3)}$$ via a symmetry argument. 
Therefore
\begin{align*}
\sup_{(e,v)\in\mathbb{S}^2\times\R^3}\meas{X_\eta(e,v)}\leq 
\meas{Y_\eta(e_3)}.
\end{align*}
It holds
\begin{align*}
Y_\eta(e_3) &= \{ z\in\R^3~:~|z|<R,~|z|^2(1-\hat{z}_3^2)<\eta^2 \}\\
&= \{ z\in\R^3~:~ z_1^2 + z_2^2 + z_3^2<R^2,~ z_1^2 + z_2^2<\eta^2 \},
\end{align*}
where we use notation that $\hat{z} = (\hat z_1, \hat z_2, \hat z_3)$.
Since $Y_\eta({e_3})$ is the intersection between a sphere and cylinder, both centered in zero, with the cylinder radius equal to $\eta$, it follows that the measure of $Y_\eta(e_3)$ tends to zero as $\eta\to 0$, implying
$$
\lim_{\eta\to 0}\sup_{(e,v)\in\mathbb{S}^2\times\R^3}\meas{X_\eta(e,v)} = 0.
$$
	Choosing $\eta>0$ small enough and employing \eqref{A.lb.2} yields \eqref{A.lb.goal}. This finishes the proof of the lemma.
\end{proof}

\section{Fisher information, $\gamma \in (-\sqrt{8},-2]$} \label{sec:fisher}

In a previous work \cite{GS24}, one of the authors and Silvestre established the monotonicity of the Fisher information for solutions to the space-homogeneous Landau equation (see also the recent expository article \cite{GuiSil25} and the survey by Villani \cite{Villani2025}). In this section, we adapt these techniques to control the \emph{full} (spatial and velocity) Fisher information for the fuzzy Landau equation \eqref{1}. Our main result ({\em{a-priori estimates}}) is summarized in the following theorem:

\begin{theo}\label{Theo_Fisher_xv}

Let $f\ge 0$ be a smooth solution to (\ref{1}) in the time interval $[0,T]$ with  $\kappa = \kappa(|x|)$ and   
{{$$
\left\|  \kappa \left(\frac{|\nabla \kappa|}{\kappa}\right)^{^{-\gamma}}\right\|_{L^\infty(\mathbb{R}^3)} \le c_0 
$$}}
and potentials $\alpha = \alpha(|v-w|)$ such that 
\begin{align}\label{cond_alpha_Fish}
\sup_{r\in \mathbb{R}_+} \left ( {{\frac{(\alpha'(r))^2r^2 }{\alpha^2(r)}}}\right) < {{8}}, 
\quad 
\alpha(r)  \le r^{\gamma}, \quad \gamma \le -2. 
%\sup_{r\ge \theta} r^2 \alpha(r)  < +\infty,
\end{align}
Assume moreover that the initial data has bounded Fisher information:
$$
c_{in}:=\int \frac{|\nabla_v f_{in}|^2}{f_{in}}+  \frac{|\nabla_x f_{in}|^2}{f_{in}} \;dxdv <\infty .
$$
Then there exists a positive constant $\delta>0$ that only depends on the kernel $\kappa$ and on the mass of $f$ such that  
\begin{align*}
    \frac{d}{dt} \int  e^{-\delta t} \left(\frac{|\nabla_v f|^2}{f}+  \frac{|\nabla_x f|^2}{f}\right) \;dxdv & \le 0. 
\end{align*}
In particular for any $0<T<+\infty$ we have 
\begin{align*}
\sup_{t\in [0,T]}\; \int \frac{|\nabla_v f|^2}{f}+  \frac{|\nabla_x f|^2}{f} \;dxdv & \le C(T, \delta). 
\end{align*}
\end{theo}

%{{The assumption of this theorem and the "Main Theorem" in the Introduction should be the same. In particular, \eqref{kappa} should be stated explicitly in Theorem 3.1, and we should decide if we use radial or even functions.}}

%{ We should mention the space where the solution exist and and the role of $T$.}

\begin{rem}
For $\alpha(r) = r^\gamma$, conditions (\ref{cond_alpha_Fish}) imply ${{ -\sqrt{8}< \gamma\le -2}}$.
\end{rem}

To track the time evolution of the Fisher information, we begin by rewriting the collision kernel \eqref{N}, adopting the formulation introduced in \cite{GS24}:

\begin{align}\label{kernelN-bk}
N(v-w)&=\alpha(|v-w|) |v-w|^2 \Pi(v-w) \nonumber\\
&=  \alpha(|v-w|)\sum_{k=1}^3 b_k(v-w) \otimes b_k(v-w).
\end{align}
Here, $b_1,b_2,b_3$ are the vector fields generating rotations about the axes in $\mathbb{R}^3$ given by the canonical basis. That is  
\begin{align*}
b_1 := \left( \begin{array}{ccc} 0 \\
                             w_3-v_3\\
                             v_2-w_2
              \end{array}  \right),\;\; b_2 := \left( \begin{array}{ccc} v_3-w_3 \\
                             0\\
                             w_1-v_1
              \end{array}  \right),\;\; b_3 := \left( \begin{array}{ccc} w_2-v_2 \\
                                                   v_1-w_1\\
                                                   0\\
              \end{array}  \right).\end{align*}
			  
The approach in \cite{GS24} relies on doubling the velocity variables, which involves lifting vector fields from  $\mathbb{R}^3$ to $\mathbb{R}^6$. In our space-inhomogeneous setting, this lifting must encompass both velocity and spatial variables, effectively extending vectors from $\mathbb{R}^{3} \times \mathbb{R}^{3}$ to $\mathbb{R}^{6} \times \mathbb{R}^{6}$. More specifically, let $z \in \mathbb{R}^{12}$ be the vector 
$$
z := (x,v,y,w) \in \mathbb{R}^{12}
$$
 representing the positions and velocities of a pair of particles. Thus, we will use $\tilde{b}_k\in \mathbb{R}^{12}$  to denote the vector 
\begin{align}\label{vec:bkt}
\tilde{b}_k(v-w):=\left( \begin{array}{ccc} 0\\b_k \\0\\
                             -b_k  \end{array} \right ),  \quad k=1,2,3.
\end{align}			  

We will also refer frequently to the following vectors 	  
\begin{align}\label{vec:exith}
  \tilde{e}_i :=\left( \begin{array}{ccc} e_i\\0\\
                             e_i \\0
              \end{array}  \right),\;\; \hat{e}_i :=\left( \begin{array}{ccc} e_i\\0\\
                             -e_i \\0
              \end{array}  \right),\;\; \tilde{\xi}_i :=\left( \begin{array}{ccc} 0\\ e_i\\
                          0\\   e_i 
              \end{array}  \right),\;\; \hat{\xi}_i :=\left( \begin{array}{ccc} 0\\e_i\\
                            0\\ -e_i 
              \end{array}  \right), 	
\end{align}				  
where $\{e_i\}_{i=1,2,3}$, are the basis vectors for $\mathbb{R}^3$. 

 For the rest of this subsection, $\nabla$ and $\Div$ will represent the gradient and divergence with respect to $z$. Note that for any function $h(z)$ we have 
\begin{align*}
| \nabla_x h|^2 + | \nabla_y h|^2  =  \frac{1}{2} \sum_{i=1}^3 \; | \tilde e_i  \cdot \nabla h|^2 + | \hat e_i \cdot \nabla h|^2,\\
| \nabla_v h|^2 + | \nabla_w h|^2 = \frac{1}{2} \sum_{i=1}^3 \; | \tilde \xi_i \cdot  \nabla h|^2 + | \hat \xi_i \cdot  \nabla h|^2.
\end{align*}

In Lemma \ref{Lem:Fisher_Landau} and Lemma \ref{Lem:Fisher_Landau_v} we compute the time derivative  of the Fisher information along \eqref{1}. A key observation is that these time derivatives can be expressed as integrals in double the number of variables (so, $z$ instead of just $(x,v)$ and $(y ,w)$). These integrals involve the function $z\mapsto f(x,v)f(y,w)$, and a degenerate elliptic operator in $\mathbb{R}^{12}$ acting on a function $F(z) = F(x,y,v,w)$. 
 To motivate the way this elliptic operator is defined,  we begin by noting that the collision operator in the Fuzzy landau equation \eqref{1}, which we now denote by $q(f)$, can be rewritten as follows 
\begin{align}\label{qF}
q(f)(x,v) := & \Div_v\int_{\R^3}\int_{\R^3}N(v-v^*)\kappa(x-x^*)(f^*\nabla f - f\nabla f^*)\dd x^*\dd v^* \nonumber\\
 =& \Div_v\int_{\R^3}\int_{\R^3}N(v-v^*)\kappa(x-x^*)(\nabla_v - \nabla_{v^*})\big(f(x,v) f(x^*, v^*)\big)\dd x^*\dd v^*.
\end{align}
Now, inspired by \cite{GS24},   we define the lifted operator acting on a function $F(z) = F(x,y,v,w)$ according to
\begin{align}
  Q(F)(x,v,y,w) 
  := &(\Div_v - \Div_w) N(v-w)\kappa(x-y)(\nabla_v - \nabla_{w})F(x,v,y,w) \notag\\
   =&  \kappa(x-y)  \sum_{k=1}^3  \sqrt{\alpha} \tilde{b}_k \cdot \nabla(\sqrt{\alpha}  \tilde{b}_k \cdot \nabla F), \label{e:lifted operator}
\end{align}  
where the last line will be proved in the next lemma.
% In Lemma \ref{Lem:Fisher_Landau} and Lemma \ref{Lem:Fisher_Landau_v} we compute the time derivative  of the Fisher information along \eqref{1}. A key observation is that these time derivatives can be expressed as integrals in double the number of variables (so, $z$ instead of just $(x,v)$ and $(y ,w)$). These integrals involve the function $z\mapsto f(x,v)f(y,w)$, and a degenerate elliptic operator in $\mathbb{R}^{12}$ acting on a function $F(z) = F(x,y,v,w)$ according to\footnote{a spatially homogeneous version of this operator has been introduced in \cite{GS24}}
% \begin{align}\label{e:lifted operator}
%   Q(F)(x,v,y,w) :=\kappa(x-y)  \sum_{k=1}^3  \sqrt{\alpha} \tilde{b}_k \cdot \nabla(\sqrt{\alpha}  \tilde{b}_k \cdot \nabla F).
% \end{align}  

From now on the function $F$ represents the product 
\begin{align*}
    F(z) = F(x,v,y,w) = f(x,v) f(y,w).
\end{align*}

The next lemma provides preparatory identities. 
\begin{lemma}
Denote by $F$ the product 
$$
F(x,v,y,w,t):= f(x,v,t) f(y,w,t),
$$ 
where $f(x,v,t)$ is a solution to (\ref{1}). We have the following identities: 
\begin{align}
(i) \quad \Div_{v-w}  \left[ \kappa(x-y)N(v-w) (\nabla_v-\nabla_w) F\right] &= \kappa(x-y) \sum_{k=1}^3  \sqrt{\alpha} \tilde{b}_k \cdot \nabla(\sqrt{\alpha}  \tilde{b}_k \cdot \nabla F) \label{Q(F)}\\
&= \sum_{k=1}^3  \sqrt{\alpha} \tilde{b}_k \cdot \nabla(\sqrt{\alpha} \kappa(x-y)  \tilde{b}_k \cdot \nabla F), \label{Q(F)2}
\end{align}
\begin{align}
{} \nonumber \\
(ii) \quad  \frac{d}{dt}  \int \frac{(e_i \cdot \nabla_x f)^2 }{f} \;dvdx =&  \frac{1}{2} \int \int  \frac{\tilde e_i \cdot  \nabla F }{F}  (\tilde e_i \cdot \nabla Q(F)) +   \frac{\hat e_i \cdot  \nabla F }{F}  (\hat e_i \cdot \nabla Q(F))\;dz  \nonumber \\
 & - \frac{1}{4} \int\int  \left( \frac{|\tilde e_i \cdot  \nabla F|^2 }{F^2}  +   \frac{|\hat e_i \cdot  \nabla F|^2 }{F^2} \right) Q(F) \;dz\label{Equiv} ,\\
 {} \nonumber \\
 (iii) \quad \frac{d}{dt}  \int \frac{(e_i \cdot \nabla_v f)^2 }{f} \;dvdx = &\; \frac{1}{2} \int \int  \frac{\tilde \xi_i \cdot  \nabla F }{F}  (\tilde \xi_i \cdot \nabla Q(F)) +   \frac{\hat \xi_i \cdot  \nabla F }{F}  (\hat \xi_i \cdot \nabla Q(F))\;dz   \nonumber \\
 & - \frac{1}{4} \int\int  \left( \frac{|\tilde \xi_i \cdot  \nabla F|^2 }{F^2}  +   \frac{|\hat \xi_i \cdot  \nabla F|^2 }{F^2} \right) Q(F) \;dz \label{Equiv_v} \\
 &- 8 \int (e_i \cdot \nabla_v \sqrt{f})(e_i \cdot  \nabla_x \sqrt{f}) \;dxdv,   \nonumber
\end{align}
 where $Q(F)$ is the linear operator defined in \eqref{e:lifted operator}. 
\end{lemma}

\begin{proof}

({\em i})  To obtain this part,  we use \eqref{kernelN-bk} and calculations from \cite{GS24}, as well as the fact that the operator $Q$ acts in spatial variables in a multiplicative way via $\kappa$. More precisely,
using \eqref{kernelN-bk} and  $ b_k \otimes b_k=   \{{b_k}_i {b_k}_j \}_{i,j}$ we have that 
\begin{align*}
  \textrm{div}_{v-w} \left[\kappa(x-y)N(v-w)(\nabla_v-\nabla_w) F \right]
\end{align*}
equals
\begin{align*}
& \kappa(x-y)\sum_k \sum_i (\partial_{v_i}  - \partial_{w_i}) \left[\alpha \;  \sum_j {b_k}_i  {b_k}_j (\partial_{v_j}  - \partial_{w_j}) F\right]\\
& = \kappa(x-y) \sum_k \sum_i (\partial_{v_i}  - \partial_{w_i}) \left[ \alpha \; {b_k}_i \tilde{b}_k \cdot \nabla F  \right]\\
& = \kappa(x-y)\sum_k \textrm{div}_{v,w} ( \alpha \; \tilde{b}_k\tilde{b}_k\cdot \nabla F ) \\
& = \kappa(x-y) \sum_k  \tilde{b}_k \cdot \nabla( \alpha \; \tilde{b}_k\cdot \nabla F ) \\
&=  \kappa(x-y) \sum_k \sqrt{\alpha} \;\tilde{b}_k \cdot \nabla( \sqrt{\alpha}  \; \tilde{b}_k\cdot \nabla F ).
\end{align*}
Here we used that $\Div \tilde{b}_k =0$ and that $\nabla \sqrt{\alpha} \cdot \tilde{b}_k =0$ for $k=1,2,3$.  Note that, more generally,  
\begin{align}\label{trick}
\nabla {g} \cdot \tilde{b}_k =0, \quad \forall \;k=1,2,3,
\end{align}
for any function $g = g(|v-w|)$. This proves \eqref{Q(F)} and  \eqref{Q(F)2}.

\vspace{3mm}

({\em ii}) For this identity, we expand the time derivative and replace  $f_t$ with $-v\cdot \nabla_x f + q(f)$: 
\begin{align*}
\frac{d}{dt}  \int \frac{|e_i \cdot  \nabla_x f|^2 }{f} \;dvdx =& \; 2 \int \frac{e_i \cdot \nabla_x f }{f} (e_i \cdot  \nabla_x f_t) \;dvdx - \int  \frac{|e_i \cdot  \nabla_x f|^2 }{f^2} f_t\;dvdx  \\
=&  -2 \int \frac{e_i \cdot \nabla_x f }{f} (e_i \cdot  \nabla_x (v\cdot \nabla_x f) ) \;dvdx + \int  \frac{|e_i \cdot  \nabla_x f|^2 }{f^2} (v\cdot \nabla_x f) \;dvdx  \\
& +2 \int \frac{e_i \cdot \nabla_x f }{f} (e_i \cdot  \nabla_x q(f) ) \;dvdx - \int  \frac{|e_i \cdot  \nabla_x f|^2 }{f^2} q(f)\;dvdx\\
=:& \; J_1+J_2+ I_1+ I_2. 
\end{align*}
Integrating by parts yields 
$$
J_1+J_2 = -8 \int \nabla_x \sqrt{f} \cdot \nabla_x ( v \cdot \nabla_x \sqrt{f})\;dxdv = 0.
$$
Let us now consider the term $I_1$: 
\begin{align*}
I_1 = & \;2 \int \int  \frac{e_i \cdot  \nabla_x f }{f}  (e_i \cdot \nabla_x \left( \textrm{div}_v  \kappa (x-y)N(v-w) (\nabla_v-\nabla_w) F \right))\;dvdxdwdy \\
 =&   \;2 \int \int  \frac{e_i \cdot  \nabla_x f }{f}  (e_i \cdot \nabla_x \left(\textrm{div}_{v-w} \kappa (x-y)N(v-w) (\nabla_v-\nabla_w) F\right))\;dz\\ 
 =&  \;2 \int \int  \frac{e_i \cdot  \nabla_x f }{f}  (e_i \cdot \nabla_x Q(F))\;dz \\
 =& \;2 \int \int  \frac{e_i \cdot  \nabla_x F }{F}  (e_i \cdot \nabla_x Q(F))\;dz  ,
 \end{align*} 
  using (\ref{Q(F)}) and the fact that $ \frac{\nabla_x f }{f} = \frac{\nabla_x F }{F} $. We now proceed with the symmetrization by applying the change of variables $(x,v) \to (y,w)$ in the second integral:
 \begin{align*}
I_1 = & \int \int  \frac{e_i \cdot  \nabla_x F }{F}  (e_i \cdot \nabla_x Q(F))\;dz+ \int \int  \frac{e_i \cdot  \nabla_y F }{F}  (e_i \cdot \nabla_y Q(F))\;dz \\
=&  \frac{1}{2} \int \int  \frac{\tilde e_i \cdot  \nabla F }{F}  (\tilde e_i \cdot \nabla Q(F))\;dz +\frac{1}{2} \int \int  \frac{\hat e_i \cdot  \nabla F }{F}  (\hat e_i \cdot \nabla Q(F))\;dz. 
\end{align*}
Similarly, for $I_2$ we have  
 \begin{align*}
- \int  \frac{|e_i \cdot  \nabla_x f|^2 }{f^2} q(f)\;dxdv 
=&  - \int\int  \frac{|e_i \cdot  \nabla_x f|^2 }{f^2} \left( \textrm{div}_v  \kappa (x-y)N(v-w) (\nabla_v-\nabla_w) F \right))\;dvdwdxdy \\
=&  - \int\int  \frac{|e_i \cdot  \nabla_x f|^2 }{f^2} \left( \textrm{div}_{v-w}  \kappa (x-y)N(v-w) (\nabla_v-\nabla_w) F \right))\;dvdwdxdy \\
=& { - \frac{1}{2} \int\int  \frac{|e_i \cdot  \nabla_x F|^2 }{F^2}  Q(F) \;dz - \frac{1}{2} \int\int  \frac{|e_i \cdot  \nabla_y F|^2 }{F^2}  Q(F) \;dz}\\
=& - \frac{1}{4} \int\int  \frac{|\tilde e_i \cdot  \nabla F|^2 }{F^2}  Q(F) \;dz - \frac{1}{4} \int\int  \frac{|\hat e_i \cdot  \nabla F|^2 }{F^2}  Q(F) \;dz.
\end{align*} 
({\em iii}) We follow similar computations to get (\ref{Equiv_v}), with the exception of the contribution from the transport term:
$$
 \int -2 \frac{e_i \cdot \nabla_v f }{f} (e_i \cdot  \nabla_v (v\cdot \nabla_x f) )  + \frac{|e_i \cdot  \nabla_v f|^2 }{f^2} (v\cdot \nabla_v f) \;dvdx  = -8  \int (e_i \cdot \nabla_v \sqrt{f})(e_i \cdot  \nabla_x \sqrt{f}) \;dxdv.
$$
This concludes the proof. 

\end{proof}
Now we compute the time evolution of the Fisher information, considering the spatial and velocity components separately.  Very frequently we will use the identity 
$$
a \cdot \nabla (b \cdot \nabla c) = b \cdot \nabla (a \cdot \nabla c) + [a,b] \cdot \nabla c,
$$
where $[a,b]$ is the vector with $i$-th component $[a,b]_i := \sum_{j} a_j\partial_j b_i - b_j \partial_j a_i$. \\

The next lemma gathers some commutator identities.
%{ We need to add  quick proof of the first two statements in the lemma below.}

\begin{lemma}\label{lem_commutators}
We have 
\begin{align}\label{comm1}
[\tilde{e}_i, \sqrt{\alpha}\tilde{b}_k] = [\hat{e}_i, \sqrt{\alpha} \tilde{b}_k]=[\tilde{\xi}_i, \tilde{b}_k] = [\hat{\xi}_k, \tilde{b}_k]= 0 , \quad \forall \; i,k=1,2,3,
\end{align}
and 
\begin{align}\label{comm-xi}
 [\hat{\xi}_1, \tilde{b}_2] =-[\hat{\xi}_2, \tilde{b}_1]=  -2\hat{\xi}_3 , \quad  [\hat{\xi}_1, \tilde{b}_3]=- [\hat{\xi}_3, \tilde{b}_1]=  2\hat{\xi}_2, \quad  [\hat{\xi}_3, \tilde{b}_2]= -[\hat{\xi}_2, \tilde{b}_3] =2 \hat{\xi}_1 . 
 \end{align}
  Moreover, for $\alpha = \alpha(r) $ with $r = |v-w|$, we have
 \begin{align}\label{comm3}
[\tilde{\xi}_i, \sqrt{\alpha} \;\tilde{b}_k] & = 0, \\
\quad [\hat{\xi}_i,  \sqrt{\alpha} \; \tilde{b}_k] &=\sqrt{\alpha}[\hat{\xi}_i,  \; \tilde{b}_k] + {{ {{\frac{\alpha'}{r \sqrt{\alpha}} (v_i-w_i) \;\tilde{b}_k}}}}. \label{comm4}
 \end{align}
If $\alpha(r) = r^\gamma$ then 
\begin{align}\label{r_to_gamma}
\frac{\alpha'}{r \sqrt{\alpha}} =  {{ \gamma \;\frac{\sqrt{\alpha}}{r^2} }}.
\end{align}

Define
\begin{align*}
\hat n 
:=\left( \begin{array}{ccc} 0\\v-w \\0\\- (v-w)
\end{array}  \right) 
%= \sum_i  \left( \begin{array}{ccc} 0\\(v_i-w_i)e_i\\0\\- (v_i-w_i)e_i
%\end{array}  \right)
%=\sum_i (v_i - w_i) \hat{\xi}_i
\end{align*}
we have that 
$$
[\hat n, \tilde{b}_k ] =0, \quad \forall \; k=1,2,3.
$$
\end{lemma}
\begin{proof}
Recall the definitions (\ref{vec:exith}) and (\ref{vec:bkt}). 
Since vectors $\tilde{e}_i$ have nonzero entries in the first and third position, while vectors $\tilde{b}_k$ have nonzero entries in the second and fourth positions, we have  $[\tilde{e}_i, \sqrt{\alpha}\tilde{b}_k] = [\hat{e}_i, \sqrt{\alpha} \tilde{b}_k]=0$.

Since $\tilde{\xi}_i$ are constant, we have that 
\begin{align*}
    [\tilde{\xi}_i, \tilde{b}_k]_j = &\tilde{\xi}_i \cdot \nabla \; {\tilde{b}_{k,j}}(v-w) \\
   = & ( \partial_{v_i} + \partial_{w_i})\;  {\tilde{b}_{k,j}}(v-w)\\
   = & 0. 
\end{align*}
Moreover  
\begin{align*}
 [\hat{\xi}_k, \tilde{b}_k]_j = &\hat{\xi}_k \cdot \nabla \; {\tilde{b}_{k,j}}(v-w) \\
   = & ( \partial_{v_k} - \partial_{w_k})\;  {\tilde{b}_{k,j}}(v-w)\\
   = & 0,
\end{align*}
because the vector $\tilde{b}_k$ does not depend on $v_k$ and $w_k$.  

We now prove the first identity in (\ref{comm-xi}), the others follow similarly: 
\begin{align*}
 [\hat{\xi}_1, \tilde{b}_2]_j   = &\hat{\xi}_1 \cdot \nabla \; {\tilde{b}_{2,j}}(v-w)\\
 = & ( \partial_{v_1} - \partial_{w_1})\;  {\tilde{b}_{2,j}}(v-w)\\
 = & 2  \partial_{v_1}\;  {\tilde{b}_{2,j}}(v-w). 
\end{align*}
Note that $2  \partial_{v_1}\;  {\tilde{b}_{2,6}}(v-w) = -2$, $2  \partial_{v_1}\;  {\tilde{b}_{2,12}}(v-w) = 2$ and $2  \partial_{v_1}\;  {\tilde{b}_{2,j}}(v-w)=0$ for all other $j$, which implies $ [\hat{\xi}_1, \tilde{b}_2] = - 2 \hat{\xi}_3$. 

To prove \eqref{comm3} and \eqref{comm4}, we note that
    \begin{align*}
[\tilde{\xi}_i, \sqrt{\alpha} \;\tilde{b}_k] &= \sqrt{\alpha}{\underbrace{[\tilde{\xi}_i,  \; \tilde{b}_k] }_{=0}}+ {\underbrace{(\tilde{\xi_i}\cdot \nabla \sqrt{\alpha})}_{=0}} \tilde{b}_k = 0, \\
\quad [\hat{\xi}_i,  \sqrt{\alpha} \; \tilde{b}_k] &=  \sqrt{\alpha}[\hat{\xi}_i,  \; \tilde{b}_k] + (\hat \xi_i \cdot \nabla \sqrt{\alpha}) \tilde{b}_k = \sqrt{\alpha}[\hat{\xi}_i,  \; \tilde{b}_k] + {{ {{\frac{\alpha'}{r \sqrt{\alpha}} (v_i-w_i) \;\tilde{b}_k}}}}.
 \end{align*}

We now show that $[\hat n, \tilde{b}_k]=0$. Note that 
$$
\hat n = %\sum_i  \left( \begin{array}{ccc} 0\\(v_i-w_i)e_i\\0\\- (v_i-w_i)e_i
%\end{array}  \right)
\sum_i (v_i - w_i) \hat{\xi}_i
$$
and therefore 
\begin{align*}
    [\hat n, \tilde{b}_k] =&  \sum_i [(v_i - w_i) \hat{\xi}_i, \tilde{b}_k ] \\
    =& \sum_i  (v_i - w_i)  [ \hat{\xi}_i, \tilde{b}_k ] - \tilde{b}_k  \cdot \nabla (v_i - w_i) \; \hat{\xi}_i. 
\end{align*}
If $i=k$ then 
$$
(v_k - w_k)  {\underbrace{[ \hat{\xi}_k, \tilde{b}_k ]}_{=0}} - {\underbrace{\tilde{b}_k  \cdot \nabla (v_k - w_k)}_{=0}} \; \hat{\xi}_k =0. 
$$
For $i\neq k$ we have 
$$
\sum_{i\neq k}(v_i - w_i)  {\underbrace{[ \hat{\xi}_i, \tilde{b}_k ] }_{\textrm{(\ref{comm-xi})}}} - {\underbrace{\tilde{b}_k  \cdot \nabla (v_i - w_i)}_{2{b}_k \cdot \nabla_v v_i}} \; \hat{\xi}_i =0. 
$$
 
\end{proof}

We have all ingredients to compute the derivative of the spatial Fisher information:

\begin{lemma} \label{Lem:Fisher_Landau}

Let $f(x,v)$ be a solution to (\ref{1}). 
Denote by $F$ the product $F(x,v,y,w):= f(x,v) f(y,w)$. We have
\begin{align*}
\frac{d}{dt} \int   \frac{|\nabla_x f|^2 }{f} \; dvdx =    \sum_{i,k=1}^3 &\Big\{ - \frac{1}{2}  \int  {{\kappa}}  F  |  \tilde{e}_i \cdot   \nabla ( \sqrt{\alpha} \tilde{b}_k  \cdot \nabla \ln F)|^2  \;dz \\
& -  \frac{1}{2}  \int {{\kappa}}  F  |  \hat{e}_i \cdot   \nabla (  \sqrt{\alpha}\tilde{b}_k  \cdot \nabla \ln F)|^2  \;dz\\
&  {{ - \frac{1}{2}   \int  F  ( \sqrt{\alpha} \; \tilde{b}_k \cdot \nabla \ln F  )  \hat{e}_i \cdot \nabla (  \sqrt{\alpha} \tilde{b}_k  \cdot \nabla \ln  F)  (  \hat{e}_i \cdot \nabla  \kappa)  \;dz}}
\Big\}.
\end{align*}

\end{lemma}
In the above identity all the terms except the one with $\nabla \kappa$ are negative. Hence, the following corollary is an immediate consequence of the previous lemma.
\begin{cor}\label{Cor:Fisher_Landau_x}
    If $\kappa \equiv constant$, the functional 
    $$
     \int   \frac{|\nabla_x f|^2 }{f} \; dvdx
    $$
    is monotone decreasing in time for any $\gamma$ for which (\ref{1}) makes sense. 
\end{cor}
We are ready to prove Lemma \ref{Lem:Fisher_Landau}. Our proof is based on straightforward integration by parts. 
\begin{proof}[Proof of Lemma \ref{Lem:Fisher_Landau}] 
Using (\ref{Equiv}) we get 
\begin{align}
\frac{d}{dt}  \int \frac{|\nabla_x f|^2 }{f} \;dvdx = \sum_{i=1}^3 & \;  \frac{1}{2}  \int   \frac{\tilde{e}_i \cdot  \nabla F }{F}  (\tilde{e}_i\cdot \nabla Q(F)) \;dz -  \frac{1}{4} \int    \frac{(\tilde{e}_i \cdot \nabla F )^2 }{F^2} Q(F) \;dz \nonumber\\
& \; + \frac{1}{2}  \int   \frac{\hat{e}_i \cdot  \nabla F }{F}  (\hat{e}_i\cdot \nabla Q(F)) \;dz -  \frac{1}{4} \int    \frac{(\hat{e}_i \cdot \nabla F )^2 }{F^2} Q(F) \;dz,\label{mm}
\end{align}
where
\begin{align*}
Q(F)=  \sum_{k=1}^3 \sqrt{\alpha} \tilde{b}_k \cdot \nabla(  {{\kappa(x-y)}}\sqrt{\alpha} \tilde{b}_k \cdot \nabla F) .
\end{align*}

We first focus on the terms with $\tilde{e}_i$: we use the formulation 
\begin{align*}
\int   \frac{1}{2} \frac{\tilde{e}_i \cdot  \nabla F }{F}  (\tilde{e}_i\cdot \nabla Q(F)) & -  \frac{1}{4}  \frac{(\tilde{e}_i \cdot \nabla F )^2 }{F^2} Q(F) \;dz\\
= &\; \frac{1}{2} \int  \frac{1}{2}  Q(F)  (  \tilde{e}_i \cdot \nabla \ln F)^2 + F (\tilde{e}_i \cdot \nabla \ln F) \left(  \tilde{e}_i \cdot \nabla \left(\frac{Q(F)}{F}\right)\right)\;dz.
\end{align*}
We substitute
$$
\frac{Q(F)}{F}  =  Q(\ln F) +  \sum_k {{\kappa}} | \sqrt{\alpha} \tilde{b}_k \cdot \nabla \ln F|^2
$$
in the last term, and obtain
\begin{align*} 
\frac{1}{2} \int   \frac{1}{2}Q(F) |\tilde{e}_i \cdot   \nabla \ln F|^2   & +    F (\tilde{e}_i \cdot \nabla \ln F) \left(\tilde{e}_i  \cdot \nabla Q(\ln F)\right)  \\
&  +  \sum_k  F (\tilde{e}_i \cdot \nabla \ln F) \left(\tilde{e}_i  \cdot \nabla ( {{\kappa}} \;\alpha |    \tilde{b}_k \cdot \nabla \ln F|^2) \right)\;dz =:\frac{1}{2} (\tilde{I}_1 + \tilde I_2+ \tilde I_3).
\end{align*}
Similar computations hold for the terms in (\ref{mm}) with $\hat {e}_i$, and imply 
$$
\frac{d}{dt}  \int \frac{|\nabla_x f|^2 }{f} \;dvdx = \sum_i \frac{1}{2}( \tilde{I}_1 + \tilde I_2+ \tilde I_3) + \frac{1}{2}( \hat{I}_1 + \hat I_2+ \hat I_3). 
$$
We have simplified the notation by hiding  the subindex $i$ from $I_1, I_2, I_3$. \\

We start with $\tilde I_2$. Using $[\tilde{e}_i,  \sqrt{\alpha}\tilde{b}_k]=0$, $\Div \tilde{b}_k =0$ and \eqref{trick}, we have  
\begin{align*}
\tilde I_2 =&   \sum_{k=1}^3   \int    F (\tilde{e}_i \cdot \nabla \ln F) \left(\tilde{e}_i  \cdot \nabla ( \sqrt{\alpha} \tilde{b}_k \cdot \nabla( {{\kappa}}\sqrt{\alpha}\tilde{b}_k \cdot \nabla \ln F) )\right) \;dz  \\
 =&   \sum_{k=1}^3   \int    F (\tilde{e}_i \cdot \nabla \ln F) \left( \sqrt{\alpha}  \tilde{b}_k   \cdot \nabla (  \tilde{e}_i \cdot \nabla( {{\kappa}} \sqrt{\alpha}\tilde{b}_k \cdot \nabla \ln F) )\right)  \;dz  \\
 =&  -  \sum_{k=1}^3  \int  \sqrt{\alpha} \tilde{b}_k \cdot \nabla (\tilde{e}_i \cdot \nabla  F)  (  \tilde{e}_i \cdot \nabla( {{\kappa}} \sqrt{\alpha} \;\tilde{b}_k \cdot \nabla \ln F) )  \;dz\\
 =& -  \sum_{k=1}^3  \int    \tilde{e}_i \cdot \nabla (\sqrt{\alpha} \tilde{b}_k  \cdot \nabla  F)  (  \tilde{e}_i \cdot \nabla( {{\kappa}}\sqrt{\alpha}  \tilde{b}_k \cdot \nabla \ln F) )  \;dz.
 \end{align*}
 Next we separate the kernel ${{\kappa}}$ from $(\sqrt{\alpha} \; \tilde{b}_k \cdot \nabla \ln F)$ in the last term and get 
 \begin{align*}
\tilde I_2 =  -  \sum_{k=1}^3 & \int   {{\kappa}}  \tilde{e}_i \cdot \nabla ( \sqrt{\alpha} \tilde{b}_k  \cdot \nabla  F)  (  \tilde{e}_i \cdot \nabla(  \sqrt{\alpha} \; \tilde{b}_k \cdot \nabla \ln F) )  \;dz \\
&{{- \int   ( \sqrt{\alpha} \; \tilde{b}_k \cdot \nabla \ln F  )  \tilde{e}_i \cdot \nabla (  \sqrt{\alpha} \tilde{b}_k  \cdot \nabla  F)  (  \tilde{e}_i \cdot \nabla  {{\kappa}})  \;dz}} \\
 = -  \sum_{k=1}^3 &  \int  F {{\kappa}}  \left|\tilde{e}_i \cdot \nabla ( \sqrt{\alpha} \tilde{b}_k  \cdot \nabla \ln F) \right|^2\;dz \\
 &- \frac{1}{2}    \int  {{\kappa}}  (\tilde{e}_i \cdot \nabla F) \tilde{e}_i \cdot \nabla| \sqrt{\alpha}  \tilde{b}_k  \cdot \nabla \ln F|^2    \;dz \\
&{{  -   \int   ( \sqrt{\alpha} \; \tilde{b}_k \cdot \nabla \ln F  )  \tilde{e}_i \cdot \nabla (  \sqrt{\alpha} \tilde{b}_k  \cdot \nabla  F)  (  \tilde{e}_i \cdot \nabla  {{\kappa}})  \;dz.}} 
\end{align*}

 We now look at $\tilde I_1$. Using the fact that $\textrm{div}({{\kappa}} \sqrt{\alpha} \tilde{b}_k)=0$, we obtain
 \begin{align*} 
\tilde I_1 =& \; \frac{1}{2} \sum_{k=1}^3 \int {{\kappa}}  \sqrt{\alpha}   \tilde{b}_k \cdot \nabla( \sqrt{\alpha} \tilde{b}_k \cdot \nabla F) |\tilde{e}_i \cdot   \nabla \ln F|^2  \;dz \\
=& - \frac{1}{2} \sum_{k=1}^3 \int  {{\kappa}}  ( \sqrt{\alpha} \tilde{b}_k \cdot \nabla F) \sqrt{\alpha} \tilde{b}_k \cdot \nabla (  |\tilde{e}_i \cdot   \nabla \ln F|^2) \;dz \\
=& - \sum_{k=1}^3 \int  {{\kappa}}   ( \sqrt{\alpha} \tilde{b}_k \cdot \nabla F) (\tilde{e}_i \cdot   \nabla \ln F)  \sqrt{\alpha} \tilde{b}_k \cdot \nabla (  \tilde{e}_i \cdot   \nabla \ln F) \;dz  \\
=& - \sum_{k=1}^3 \int {{\kappa}}    (\sqrt{\alpha}  \tilde{b}_k \cdot \nabla \ln F) (\tilde{e}_i \cdot   \nabla F)  \tilde{e}_i  \cdot \nabla ( \sqrt{\alpha}  \tilde{b}_k  \cdot   \nabla \ln F) \;dz \\
=& - \frac{1}{2} \sum_{k=1}^3 \int {{\kappa}}  ( \tilde{e}_i  \cdot \nabla F) \tilde{e}_i \cdot \nabla| \sqrt{\alpha}  \tilde{b}_k  \cdot \nabla \ln F|^2 \;dz,
\end{align*}
which is similar to the second term in $\tilde I_2$.  Hence we have that 
\begin{align*}
\tilde I_1+\tilde I_2 =  -  \sum_{k=1}^3  &\int  F {{\kappa}}  \left|\tilde{e}_i \cdot \nabla ( \sqrt{\alpha} \tilde{b}_k  \cdot \nabla \ln F) \right|^2\;dz \\
 &-    \int  {{\kappa}}  (\tilde{e}_i \cdot \nabla F) \tilde{e}_i \cdot \nabla| \sqrt{\alpha}  \tilde{b}_k  \cdot \nabla \ln F|^2    \;dz \\
&{{  -   \int   ( \sqrt{\alpha} \; \tilde{b}_k \cdot \nabla \ln F  )  \tilde{e}_i \cdot \nabla (  \sqrt{\alpha} \tilde{b}_k  \cdot \nabla  F)  (  \tilde{e}_i \cdot \nabla  {{\kappa}})  \;dz}}  \\
=  -  \sum_{k=1}^3 &  \int  F {{\kappa}}  \left|\tilde{e}_i \cdot \nabla ( \sqrt{\alpha} \tilde{b}_k  \cdot \nabla \ln F) \right|^2\;dz \\
 &-    \int  {{\kappa}}  (\tilde{e}_i \cdot \nabla F) \tilde{e}_i \cdot \nabla| \sqrt{\alpha}  \tilde{b}_k  \cdot \nabla \ln F|^2    \;dz \\
&{{  -   \int  F  ( \sqrt{\alpha} \; \tilde{b}_k \cdot \nabla \ln F  )  \tilde{e}_i \cdot \nabla (  \sqrt{\alpha} \tilde{b}_k  \cdot \nabla  \ln F)  (  \tilde{e}_i \cdot \nabla  {{\kappa}})  \;dz}}\\
 & {{ -\int   (\tilde{e}_i \cdot \nabla  F) (\tilde{e}_i  \cdot \nabla  {{\kappa}})  |\sqrt{\alpha}    \tilde{b}_k \cdot \nabla \ln F|^2 \;dz.}}
\end{align*}
The term $ \tilde{I}_3$  can be rewritten as 
\begin{align*}
\tilde I_3 =  \sum_{k=1}^3 & \int  {{\kappa}}  (\tilde{e}_i \cdot \nabla F) \tilde{e}_i  \cdot \nabla | \sqrt{\alpha}    \tilde{b}_k \cdot \nabla \ln F|^2 \;dz \\
& {{ + \int   (\tilde{e}_i \cdot \nabla  F) (\tilde{e}_i  \cdot \nabla  {{\kappa}})  |\sqrt{\alpha}    \tilde{b}_k \cdot \nabla \ln F|^2 \;dz.}}
\end{align*}
The second and fourth term of $\tilde I_1+\tilde I_2$ cancel with $\tilde I_3$. Summarizing 
\begin{align*}
\tilde I_1+\tilde I_2+\tilde I_3 =  -  \sum_{k=1}^3  &  \int  F {{\kappa}}  \left|\tilde{e}_i \cdot \nabla ( \sqrt{\alpha} \tilde{b}_k  \cdot \nabla \ln F) \right|^2\;dz \\
&{{  -   \int  F  ( \sqrt{\alpha} \; \tilde{b}_k \cdot \nabla \ln F  )  \tilde{e}_i \cdot \nabla (  \sqrt{\alpha} \tilde{b}_k  \cdot \nabla \ln  F)  (  \tilde{e}_i \cdot \nabla  {{\kappa}})  \;dz.}} 
\end{align*}
Similar computations hold for the terms with $\hat{e}_i$. 
To conclude, we notice that $(  \tilde{e}_i \cdot \nabla  {{\kappa}})  =0$ from the symmetry of $\kappa$.

\end{proof}

Next we compute the derivative of the Fisher information with respect to the velocity variable:
\begin{lemma}\label{Lem:Fisher_Landau_v} 
Let $f$ be as in Lemma \ref{Lem:Fisher_Landau}. We have 
\begin{align*}
\frac{d}{dt} \int   \frac{|\nabla_v f|^2 }{f} \;dvdx   =&  -  \frac{1}{2} \sum_{k=1, i=1}^3  \int  {{\kappa}}  F \left(  |\tilde \xi_i \cdot  \nabla (  \sqrt{\alpha} \tilde{b}_k  \cdot \nabla \ln F)|^2  +   |\hat \xi_i \cdot  \nabla ( \sqrt{\alpha}  \tilde{b}_k  \cdot \nabla \ln F)|^2 \right)\;dz \\
&+ \frac{1}{2}\sum_{k=1}^3 \int  {{{\kappa}}}{{\frac{(\alpha')^2}{\alpha^2}}}\;F( \sqrt{\alpha}  \tilde{b}_k  \cdot \nabla \ln F)^2 \;dz -8  \int \nabla_v \sqrt{f}  \cdot  \nabla_x \sqrt{f} \;dxdv.
\end{align*}
\end{lemma}
\begin{proof}
The proof can be found in \cite[Lemma 8.2]{GS24}, with the exception of the last term that arises from the transport.  In this calculation, the kernel ${{\kappa}}$ acts as a multiplying parameter. 

In Appendix \ref{appendix_Fisher_Landau_v} we provide an alternative proof that follows the computations of Lemma \ref{Lem:Fisher_Landau}, with $\tilde e_i$ and $\hat e_i$ replaced by $\tilde \xi_i$ and $\hat \xi_i$ respectively. We will have terms with commutators $[ \hat \xi_i, \sqrt{\alpha} \tilde b_k]$. These commutator terms  simplify to zero if $\alpha$ is constant. For $\alpha$ general,  the commutator terms   produce the second term 
$$
\frac{1}{2}\sum_{k=1}^3 \int  {{{\kappa}}}{{\frac{(\alpha')^2}{\alpha^2}}}\;F( \sqrt{\alpha}  \tilde{b}_k  \cdot \nabla \ln F)^2 \;dz.
$$ See Appendix \ref{appendix_Fisher_Landau_v} for the detailed calculation. 
\end{proof}

\begin{proof}{\em(Proof of Theorem \ref{Theo_Fisher_xv})}

The previous two lemmas show that 
\begin{align*}
\frac{d}{dt} \int   \frac{|\nabla_v f|^2 }{f}+ \frac{|\nabla_x f|^2 }{f}  \;dvdx 
\end{align*}
equals
\begin{align*} 
  \sum_{k, i=1}^3 & -  \frac{1}{2}  \int    {{\kappa}}\;  F \left(  |\tilde \xi_i \cdot  \nabla (  \sqrt{\alpha} \tilde{b}_k  \cdot \nabla \ln F)|^2  +   |\hat \xi_i \cdot  \nabla ( \sqrt{\alpha}  \tilde{b}_k  \cdot \nabla \ln F)|^2 \right)\;dz \\
&-  \frac{1}{2}   \int  {{\kappa}}\;    F \left(  |\tilde e_i \cdot  \nabla (  \sqrt{\alpha} \tilde{b}_k  \cdot \nabla \ln F)|^2  +   |\hat e_i \cdot  \nabla ( \sqrt{\alpha}  \tilde{b}_k  \cdot \nabla \ln F)|^2 \right)\;dz \\
& \; {{  - \frac{1}{2}   \int  F  ( \sqrt{\alpha} \; \tilde{b}_k \cdot \nabla \ln F  )  \hat{e}_i \cdot \nabla (  \sqrt{\alpha} \tilde{b}_k  \cdot \nabla \ln  F)  (  \hat{e}_i \cdot \nabla  {{\kappa}})  \;dz}}\\
+ \sum_{k=1}^3 &\;  \frac{1}{2} \; \int {{{\kappa}}}\; {{\frac{(\alpha')^2}{\alpha^2}}}\;F( \sqrt{\alpha}  \tilde{b}_k  \cdot \nabla \ln F)^2 \;dz \\
-8  \int & \nabla_v \sqrt{f}  \cdot  \nabla_x \sqrt{f} \;dxdv \\
=: B_1 &+ ... +B_5.
\end{align*}

It remains to bound $B_3$ and $B_4$. For $B_4$ we use the following integral inequality, which  is proven in Appendix \ref{appendix_inequality}:
\begin{align} 
\sum_{k=1}^3 \;   & \int  \frac{\kappa}{|v-w|^2}\;F( \sqrt{\alpha}  \tilde{b}_k  \cdot \nabla \ln F)^2 \;dz \label{sob_ineq} \\
  &  \le   \frac{\Lambda_0}{2} \sum_{i,k=1}^3  \int   \kappa  F \left(  |\tilde \xi_i \cdot  \nabla (  \sqrt{\alpha} \tilde{b}_k  \cdot \nabla \ln F)|^2  +   |\hat \xi_i \cdot  \nabla ( \sqrt{\alpha}  \tilde{b}_k  \cdot \nabla \ln F)|^2 \right)\;dz.  \nonumber
\end{align}
{{The constant $\Lambda_0 =\frac{1}{4}$. }}  Inequality  (\ref{sob_ineq}) for symmetric functions $F$ was first proven in \cite{GS24} (see also Theorem 1.1 in \cite{Ji24a} and  Theorem 2.4 in \cite{Ji24}). In Appendix \ref{appendix_inequality} we show (\ref{sob_ineq}) for non-symmetric functions $F$ by adapting the proof of \cite{GS24, Ji24a}.  

We start with $B_4$. Inequality (\ref{sob_ineq}) implies that 
\begin{align*}
B_4 = &\frac{1}{2}\sum_{k=1}^3\int {{{\kappa}}}\; {{\frac{(\alpha')^2}{\alpha^2}}}\;F( \sqrt{\alpha}  \tilde{b}_k  \cdot \nabla \ln F)^2 \;dz \\
=& \; \frac{1}{2}\sum_{k=1}^3 \int {{\frac{\kappa}{|v-w|^2}}}\; {{\frac{(\alpha')^2|v-w|^2 }{\alpha^2}}}\;F( \sqrt{\alpha}  \tilde{b}_k  \cdot \nabla \ln F)^2 \;dz\\
\le \; & \frac{\Lambda_0}{4}\sup_{r\in \mathbb{R}_+} \left ( {{\frac{(\alpha'(r))^2r^2 }{\alpha^2(r)}}}\right) \sum_{i,k=1}^3   \int   \kappa  F \left(  |\tilde \xi_i \cdot  \nabla (  \sqrt{\alpha} \tilde{b}_k  \cdot \nabla \ln F)|^2  +   |\hat \xi_i \cdot  \nabla ( \sqrt{\alpha}  \tilde{b}_k  \cdot \nabla \ln F)|^2 \right)\;dz \\
=& -\frac{\Lambda_0}{2}\sup_{r\in \mathbb{R}_+} \left ( {{\frac{(\alpha'(r))^2r^2 }{\alpha^2(r)}}}\right) B_1.
    \end{align*}
  
Let's now estimate $B_3$: Young's inequality yields 
\begin{align*}
B_3 = &\sum_{k, i=1}^3 \;  {{  - \frac{1}{2}   \int  F  ( \sqrt{\alpha} \; \tilde{b}_k \cdot \nabla \ln F  )  \hat{e}_i \cdot \nabla (  \sqrt{\alpha} \tilde{b}_k  \cdot \nabla \ln  F)  (  \hat{e}_i \cdot \nabla  {{\kappa}})  \;dz}}  \\
\le & \sum_{k, i=1}^3 \; \frac{1}{8}\int  F \kappa |\hat{e}_i \cdot \nabla (  \sqrt{\alpha} \tilde{b}_k  \cdot \nabla \ln  F)|^2   \;dz + \frac{1}{2} \int  F ( \sqrt{\alpha} \; \tilde{b}_k \cdot \nabla \ln F  )^2 \frac{(  \hat{e}_i \cdot \nabla  {{\kappa}})^2 }{\kappa} \;dz \\
\le &  -\frac{1}{4} B_2+ \frac{1}{2} \sum_{k, i=1}^3  \int  F ( \sqrt{\alpha} \; \tilde{b}_k \cdot \nabla \ln F  )^2\frac{(  \hat{e}_i \cdot \nabla  {{\kappa}})^2 }{\kappa} \;dz.
\end{align*}
Before we proceed we note that  
\begin{align}\label{fact1}
 \sum_{k=1}^3 |\tilde{b}_k \cdot \nabla \ln F  | ^2 &\le 3|v-w|^2 \left( |\nabla_v \ln F|^2 +|\nabla_w \ln F|^2 \right),
 %\left| (\nabla_v-\nabla_w)\ln F\right|^2
 %= \sum_{k=1}^3 |({b}_k \otimes b_k) \nabla \ln F  | ^2 \\
% &= |v-w|^2 \left| \Pi(v-w)(\nabla_v-\nabla_w)\ln F\right|^2,
\end{align}
and 
\begin{align}\label{fact2}
     \int  F \left| \nabla_v \ln F\right|^2 \;dz =  \; &  \int  f(x,v)\;dxdv \int f | \nabla_v \ln f|^2 \;dxdv = \rho_0 \int \frac{| \nabla_v f|^2}{f} \;dxdv.
    % =  & \; 2 \rho_0 \int \frac{| \nabla_v f|^2}{f} \;dxdv,
\end{align}
%using the fact that $\|\Pi\|_{L^\infty}\le 1$. 
We now distinguish two cases: $\gamma =-2$ and $\gamma<-2$. 

{\em{Case 1: $\gamma =-2$. }} Assuming that 
$$
\left\| \frac{|\nabla \kappa|^2}{\kappa}\right\|_{L^\infty(\mathbb{R}^6)}<+\infty
$$
we obtain, recall that $\alpha(|v-w|) = |v-w|^{-2}$, 
\begin{align*}
B_3 \le &  -\frac{1}{4} B_2+ \frac{1}{2} \left\| \frac{|\nabla \kappa|^2}{\kappa}\right\|_{L^\infty(\mathbb{R}^6)} \int \frac{ F}{|v-w|^2}  (  \; \tilde{b}_k \cdot \nabla \ln F  )^2 \;dz  \\
\le &  -\frac{1}{4} B_2+ \rho_0  \left\| \frac{|\nabla \kappa|^2}{\kappa}\right\|_{L^\infty(\mathbb{R}^6)} \int \frac{| \nabla_v f|^2}{f} \;dxdv,
\end{align*}
using (\ref{fact1}) and (\ref{fact2}). 

{\em{Case 2: $\gamma <-2$. }}
We show that $\frac{|\nabla \kappa|^2}{\kappa}\;  \alpha $ can be interpolated as 

\begin{align}\label{interp_kernel}
\frac{|\nabla \kappa|^2}{\kappa}\;  \alpha \le  {{\varepsilon \frac{\kappa \alpha}{r^2} }}+ {{\frac{1}{\varepsilon} c(\kappa) \frac{1}{r^2}}}
\end{align}

Assume for the moment that (\ref{interp_kernel}) holds, taking into account that $\sum_i (  \hat{e}_i \cdot \nabla  {{\kappa}})^2 = 4|\nabla \kappa|^2$ then we get 
\begin{align*}
    \frac{1}{2} \int  F ( \sqrt{\alpha} \; \tilde{b}_k \cdot \nabla \ln F  )^2 \frac{(  \hat{e}_i \cdot \nabla  {{\kappa}})^2}{\kappa}  \;dz= & {2} \int  F \; \frac{|\nabla \kappa|^2}{\kappa}\;  \alpha \;  |\tilde{b}_k \cdot \nabla \ln F|^2  \;dz\\
    \le &  \frac{\varepsilon}{2} {\underbrace{\int \frac{\kappa}{|v-w|^2} F  ( \sqrt{\alpha} \; \tilde{b}_k \cdot \nabla \ln F  )^2\;dz}_{\textrm{left hand side of (\ref{sob_ineq})}}} \\
    & + \frac{1}{\varepsilon} \|c(\kappa)\|_{L^\infty(\mathbb{R}^6)}\int  \frac{F}{|v-w|^2} \;  |\tilde{b}_k \cdot \nabla \ln F|^2  \;dz \\
    \le &  - \frac{\varepsilon}{4} B_1 + \frac{2\rho_0}{\varepsilon} \|c(\kappa)\|_{L^\infty(\mathbb{R}^6)} \int \frac{| \nabla_v f|^2}{f} \;dxdv, 
\end{align*}
using again (\ref{fact1}) and (\ref{fact2}). 

We now show (\ref{interp_kernel}). We first write 
\begin{align*}
  \frac{|\nabla \kappa|^2}{\kappa}\;  \alpha = \left( \frac{\kappa \alpha}{r^2}\right)^{\theta} \left( \frac{c(\kappa)}{r^2}\right)^{1-\theta}  
\end{align*}
and pick $\theta$ such that 
$$
\alpha \le  \frac{\alpha^\theta}{r^2}
$$
which implies, for $\alpha \le  r^{\gamma}$, 
$$
1> \theta = \frac{\gamma+2}{\gamma}>0 \quad \textrm{for} \; \gamma< -2. 
$$
With this choice we obtain 
$$
c(\kappa) = \left( \frac{|\nabla \kappa|^2}{\kappa^{1+\theta}}\right)^{\frac{1}{1-\theta}}= \kappa \left(\frac{|\nabla \kappa|}{\kappa}\right)^{^{-\gamma}}.  
$$
Summarizing, we obtain 
\begin{align*}
    B_3 \le -\frac{1}{4}B_2 - \frac{\varepsilon}{4} B_1 + \frac{2\rho_0}{\varepsilon} \|c(\kappa)\|_{L^\infty(\mathbb{R}^6)} \int \frac{| \nabla_v f|^2}{f} \;dxdv. 
\end{align*}
The estimates for $B_3$ and $B_4$ yield 
\begin{align*}
\frac{d}{dt} \int  & \frac{|\nabla_v f|^2 }{f} + \frac{|\nabla_x f|^2 }{f} \;dvdx 
%\label{line1}  
\\ 
\le  \sum_{k,i=1}^3 & -  \Gamma_0  \int    {{\kappa}} F \left(  |\tilde \xi_i \cdot  \nabla (  \sqrt{\alpha} \tilde{b}_k  \cdot \nabla \ln F)|^2  +   |\hat \xi_i \cdot  \nabla ( \sqrt{\alpha}  \tilde{b}_k  \cdot \nabla \ln F)|^2 \right)\;dz \nonumber \\
 &-  \frac{1}{4}   \int  {{\kappa}}   F \left(  |\tilde e_i \cdot  \nabla (  \sqrt{\alpha} \tilde{b}_k  \cdot \nabla \ln F)|^2  +   |\hat e_i \cdot  \nabla ( \sqrt{\alpha}  \tilde{b}_k  \cdot \nabla \ln F)|^2 \right)\;dz \label{line2} \\
& + \frac{2\rho_0}{\varepsilon} \|c(\kappa)\|_{L^\infty(\mathbb{R}^6)} \int \frac{| \nabla_v f|^2}{f} \;dxdv\\
& -8  \int \nabla_v \sqrt{f}  \cdot  \nabla_x \sqrt{f} \;dxdv, \nonumber 
\end{align*}
with 
\begin{align*}
  \Gamma_0 =  \left( \frac{1}{2} - \frac{ \Lambda_0 }{4}\sup_{r\in \mathbb{R}_+} \left ( {{\frac{(\alpha')^2r^2 }{\alpha^2}}}\right) -\frac{\varepsilon}{8} \right), \quad \Lambda_0 = \frac{1}{4}. 
\end{align*}
We choose $\varepsilon$ small enough and $\alpha$ such that $\Gamma_0 \ge 0$ which implies 
$$
\sup_{r\in \mathbb{R}_+} \left ( {{\frac{(\alpha')^2r^2 }{\alpha^2}}}\right)<8. 
$$
Young's inequality implies 
\begin{align*}
\frac{d}{dt} \int   \frac{|\nabla_v f|^2 }{f} + \frac{|\nabla_x f|^2 }{f} \;dvdx  \le C(\kappa, \rho_0) \int   \frac{|\nabla_v f|^2 }{f} + \frac{|\nabla_x f|^2 }{f} \;dvdx.
\end{align*}
 Gronwall's inequality concludes the proof. 
 
\end{proof}

\section {An $\varepsilon$-Poincare inequality}\label{sec:poincare}
In this section, we establish a functional inequality valid for nonnegative distribution functions characterized by bounded moments and bounded Fisher information. We begin with a preliminary lemma that establishes uniform bounds for the spatial and velocity marginals.

 \begin{lemma}
     Let $f(x,v)\ge 0$ be such that 
\begin{equation}  \label{bound}
\iint | \nabla_x \sqrt{f}|^2 +  | \nabla_v \sqrt{f}|^2 \;dxdv \le C_0. 
\end{equation}
Then 
\begin{itemize}
    \item[(i)] The function  $$F(v) := \int f(x,v) \;dx$$ is bounded in $L^3_v(\mathbb{R}^3)$, and 
\begin{align}\label{boundh}
\| F \|_{L^3_v(\mathbb{R}^3)} \le C_0. 
\end{align}
\item[(ii)] The function 
$$
H[f](x,v) := \int \kappa(x-y)f(y,v) \;dy
$$
 is bounded in $ L^\infty_{x} L^3_v(\mathbb{R}^6)$ if $\kappa \in L^\infty(\mathbb{R}^3)$, and 
 \begin{align}\label{boundH}
 \| H[f] \|_{L_x^\infty L^3_v(\mathbb{R}^6)} \le \| \kappa\|_{L^\infty}C_0.
 \end{align}
\item[(iii)] The mass 
$$
\rho(x) := \int f(x,v) \;dv
$$
is bounded in $L^3_x(\mathbb{R}^3)$ and $f$ is bounded in $L^{3/2}_{x,v}(\mathbb{R}^6)$.
\end{itemize}
 \end{lemma}

\begin{proof}
Assumption (\ref{bound}) and Sobolev embedding in $3$-dimensions yield 
\begin{align*}
C_0 \ge \int  \left( \int( \sqrt{f})^6 \;dv\right)^{1/3}\;dx =  \int  \| f \|_{L^3_v}\;dx  \ge \left\| \int f \;dx \right\|_{L^3_v} =  \| F \|_{L^3_v}.
\end{align*}
Here we have used Bochner inequality to justify the second inequality. With similar argument, just by switching the role of $x$ and $v$ in (\ref{boundh}), the bound of the Fisher information for $f$ yields 
$$
\| \rho \|_{L^3_x} \le C_0.
$$

The function $H$ is bounded in $ L^\infty_{x} L^3_v(\mathbb{R}^6)$, and  
\begin{align*}
\sup_{x}  \| H \|^3 _{ L^3_v(\mathbb{R}^3)} =&   \sup_{x} \int \left(  \int \kappa (x-y) f(y,v) \;dy \right)^3 \;dv  \nonumber \\
\le &   \| \kappa \|^3_{L^\infty_x }   \int  \left( {\underbrace{ \int f(y,v) \;dy}_{=F}} \right)^3 \;dv = \| \kappa \|^3_{L^\infty_x }\| F \|^3_{L^3_v} \le \| \kappa \|^3_{L^\infty_x } C_0^3,
\end{align*}
using (\ref{boundh}). 

Finally, Sobolev inequality in $\mathbb{R}^6$ implies 
\begin{align}\label{boundf}
\iint f^{\frac{3}{2}} \;dxdv \le C_0.
\end{align}

\end{proof}

Building on these uniform bounds, we now establish an $\varepsilon$-Poincare inequality. This result will be fundamental for proving the $L^\infty$  bounds for solutions in the regime of very soft potentials, which we address in Section \ref{sec:reg}.

\begin{theo} \label{thm:eps_Poinc}
{{Let $\gamma\in (-3, 0)$}}. Let $\kappa(x)$ be a smooth kernel satisfying conditions (\ref{kappa}). 
Let $f$ be any sufficiently smooth  nonnegative function such that 
\begin{eqnarray}  \label{bound1}
\iint | \nabla_x \sqrt{f}|^2 +  | \nabla_v \sqrt{f}|^2 \;dxdv \le C_0, \\
\iint f \;dxdv \le \rho_0,\\
\iint  (\langle v \rangle^{n_1} +  {{\langle x \rangle^{n_2 }}}) f  \;dxdv \le C_{n_1,n_2},
\end{eqnarray}
with 
$$
{n}_1 > \max\left\{ \frac{ 3 |\gamma|}{3+\gamma}, \; \frac{ 2 |\gamma|}{3+\gamma}, \; 2|\gamma|\right\}, 
\quad {n}_2 > \max\left\{\frac{ 6 \lambda + 8 |\gamma|-16}{3(3+\gamma)},\;  2\lambda \; \right\}. 
$$

Then for any fixed $\varepsilon >0$   and $p>1$ there exists a positive constant $D_0$ that only depends on $\varepsilon$, $C_0$, $\rho_0$, $C_{n_1,n_2}$, $\gamma$ and $ \| \kappa \|_{L^\infty_x } $, such that the following inequality holds:
\begin{align}
\iint f^p (H[f] \ast_v|v|^{\gamma}) \;dxdv  \le \varepsilon \iint   {{\langle x \rangle^{- \lambda}}} \left|  \nabla_v ( \langle v \rangle^{\frac{\gamma}{2}} f^{\frac{p}{2}})\right|^2 \;dvdx + D_0  \iint f^p  \langle v \rangle^{{\gamma}}  \;dxdv, \label{eps_poinc_b}
\end{align}
where 
$$
H(x,v)[f] := \int \kappa(x-y)f(y,v) \;dy.
$$

\end{theo}

\begin{proof}
For notational simplicity, we drop the velocity subscript from the convolution operator, writing $\ast$ instead of $\ast_v$.  
We start with a general $\varepsilon$-Poincare inequality proven in  \cite{ABDL24} (see Proposition 1.7) 
\begin{align}\label{eps_P}
\int_{\mathbb{R}^3} \phi^2 (g \ast |v|^{\gamma})  dv  \le  \varepsilon \int \left|  \nabla_v ( \langle v \rangle^{\frac{\gamma}{2}}\phi)\right|^2 \;dv + c_0 \left( \| g \|_{L^1} + \varepsilon ^{-\frac{s}{1-s}} \| \langle \cdot \rangle^{|\gamma|}g \|_{L^q}^{\frac{1}{1-s}}\right) \int_{\mathbb{R}^3} \phi^2  \langle v \rangle^{\gamma} \;dv,
\end{align}
for any $-3 \le \gamma < 0 $ and $\frac{3}{3+ \gamma} > q > \max \{ 1, \frac{3}{5+ \gamma}\}$ and $s = \frac{3-q(3+ \gamma)}{2q} \in (0,1)$ and $c_0$ is a constant that only depends on $\gamma$ and $q$. 

Fix 
$$ \min\left\{\frac{3}{3+ \gamma}, \frac{3}{2} \right\}> q > \max\left\{ 1, \frac{3}{5+ \gamma}\right\}$$  in (\ref{eps_P}), which implies $\gamma >-3$.
Substitute in (\ref{eps_P})
$$g :=  {{\langle x \rangle^{\lambda}}} H[f], \quad \quad  \phi :=  {{\langle x \rangle^{-\frac{\lambda}{2}}}} f^{\frac{p}{2}}.$$
Integrating the resulting inequality  in $ dx$ yields
\begin{align}\label{eps_PP}
\iint f^p (H \ast |v|^{\gamma})  \;dvdx  = & \iint f^p  {{\langle x \rangle^{-\lambda}}}( {{\langle x \rangle^{\lambda}}}H \ast |v|^{\gamma})  \;dvdx \\
 \le   &  \; \varepsilon \iint  {{\langle x \rangle^{-\lambda}}} \left|  \nabla_v ( \langle v \rangle^{\frac{\gamma}{2}}f^{\frac{p}{2}})\right|^2 \;dvdx  \nonumber \\
& + c_0  \int \left( \|  {{\langle x \rangle^{\lambda}}} H \|_{L_v^1} + \varepsilon ^{-\frac{s}{1-s}} \| \langle \cdot \rangle^{|\gamma|} {{\langle x \rangle^{\lambda}}} H \|_{L_v^q}^{\frac{1}{1-s}}\right) \int_{\mathbb{R}^3} f^p  \langle v \rangle^{\gamma} {{\langle x \rangle^{-{\lambda}}}}\;dvdx \nonumber \\
\le   &  \; \varepsilon \iint   {{\langle x \rangle^{-\lambda}}}\left|  \nabla_v ( \langle v \rangle^{\frac{\gamma}{2}}f^{\frac{p}{2}})\right|^2 \;dvdx \nonumber  \\
& + c_0  \sup_{x}  \left( \|  {{\langle x \rangle^{\lambda}}} H \|_{L_v^1} + \varepsilon ^{-\frac{s}{1-s}} \| \langle \cdot \rangle^{|\gamma|} {{\langle x \rangle^{\lambda s}}} H \|_{L_v^q}^{\frac{1}{1-s}}\right) \iint f^p  \langle v \rangle^{\gamma} \;dvdx,  \nonumber
\end{align}
with $s = \frac{3-q(3+ \gamma)}{2q}$.   The proof of (\ref{eps_poinc_b}) is complete after we show that 
\begin{align}\label{last_bound}
\sup_{x}  \left( \|  {{\langle x \rangle^{\lambda}}}H \|_{L_v^1} + \varepsilon ^{-\frac{s}{1-s}} \|  {{\langle x \rangle^{\lambda s}}} \langle v \rangle^{|\gamma|}H \|_{L_v^q}^{\frac{1}{1-s}}\right)   <+\infty. 
\end{align}
Using (\ref{kappa}) we observe that
$$  {{\langle x \rangle^{\lambda}}} \le C_\lambda (  {{\langle x-y \rangle^{\lambda}}} +  {{\langle y \rangle^{\lambda}}}) \lesssim C_\lambda (\frac{1}{\kappa(x-y)} + {\langle y \rangle^{\lambda}}).$$ This inequality yields
\begin{align*}
\sup_{x}  \|  {{\langle x \rangle^{\lambda}}} H \|_{L_v^1} & = \sup_{x}  \iint  {{\langle x \rangle^{\lambda}}} \kappa(x-y)f(y,v,t) \;dy \;dv \\
& \le   C_\lambda  \iint f(y,v,t) \;dy \;dv +   C_\lambda \| \kappa \|_{L^\infty_x }  \iint   {{\langle y \rangle^{\lambda}}} f(y,v,t) \;dy \;dv\\
& \lesssim    \rho_0 +  \| \kappa \|_{L^\infty_x } C_{n_1, n_2}. 
\end{align*}

For the  term $\|  {{\langle x \rangle^{\lambda s}}} \langle v \rangle^{|\gamma|}H \|_{L_v^q}$ we use again the inequality $  {{\langle x \rangle^{\lambda}}} \le C (\frac{1}{\kappa(x-y)} + {\langle y \rangle^{\lambda}})$ and get 
\begin{align*}
 {{\langle x \rangle^{\lambda s}}}  \langle v \rangle^{|\gamma|} H(x,v) \lesssim & \;   \langle v \rangle^{|\gamma|} \int f(y,v,t) \;dy + \int \kappa(x-y)\langle v \rangle^{|\gamma|} {{\langle y \rangle^{\lambda }}} f(y,v,t)  \;dy \\
 = & \;  \langle v \rangle^{|\gamma|}  F(v) + \int \kappa(x-y)\langle v \rangle^{|\gamma|} {{\langle y \rangle^{\lambda }}} f(y,v,t)  \;dy,
\end{align*}
with $F(v) = \int f\;dx$. 
Recall that 
\begin{align}\label{condq}
q \in \left({{\max\left\{ 1, \frac{3}{5+ \gamma}\right\}, {
\min \left\{ \frac{3}{3+\gamma},c\frac{3}{2}  \right\} }}}\right).
\end{align}

The $L_v^q$ norm of $\langle v \rangle^{|\gamma|}  F(v)$ can be interpolated between the $L^3$-norm $\| F \|_{L^3_v(\mathbb{R}^3)}$ (bounded thanks to (\ref{boundh})) and the weighted $L^1_v$-norm  $\| \langle v \rangle^{m}  F \|_{L^1_v(\mathbb{R}^3)}$ for $m = \frac{2q |\gamma|}{3-q}$. 
The $L^q$ norm of $\int \kappa(x-y)\langle v \rangle^{|\gamma|} {{\langle y \rangle^{\lambda }}} f(y,v,t)  \;dy$ can be estimated thanks to Jensen's inequality with measure ${{\langle y \rangle^{-4 }}}$:
\begin{align*}
\int &\left( \int \kappa(x-y)\langle v \rangle^{|\gamma|} {{\langle y \rangle^{\lambda }}} f(y,v,t)  \;dy\right)^q \;dv \\
& \le C(\lambda) \int  \int \kappa^q(x-y)\langle v \rangle^{q|\gamma|} {{\langle y \rangle^{q\lambda + 4q -4 }}} f^q(y,v,t)  \;dy dv   \\
& \le C(\lambda) \| \kappa \|^q_{L^\infty_x } \int  \int  \left( \langle v \rangle^{2 q|\gamma|} +  {{\langle y \rangle^{2q\lambda + 8q -8}}} \right) f^q(y,v,t)  \;dy dv .   
\end{align*} 
The $L^q$ norm $ \|  \left( \langle v \rangle^{2 |\gamma|} +  {{\langle y \rangle^{2\lambda + 8 -\frac{8}{q}}}} \right)  f\|_{L^q( \mathbb{R}^6)}$ can be interpolated between the $L^{\frac{3}{2}}( \mathbb{R}^6)$ of $f$ (bounded by (\ref{boundf})) and weighted $L^1( \mathbb{R}^6)$ norm $ \| ( \langle v \rangle^{\tilde{n}_1} +  {{\langle y \rangle^{\tilde{n}_2 }}}) f\|_{L^1( \mathbb{R}^6)}$ for 
$$\tilde{n}_1 := \frac{2|\gamma|q }{3-2q},\quad \tilde{n}_2:= \frac{(2q \lambda + 8q  -{8})}{3-2q}.$$ 

Using (\ref{condq}) we obtain
$$
\tilde{n}_1 > \max\left\{ \frac{ 2 |\gamma|}{3+\gamma}, 2|\gamma|\right\},\quad \tilde{n}_2 > \max\left\{\frac{ 6 \lambda + 8 |\gamma|-16}{3(3+\gamma)},2\lambda\right\}. 
$$
\end{proof}

\section{Regularity of solutions}\label{sec:reg}
This section is devoted to  a-priori higher regularity estimates for the solutions to \eqref{1} with $-\sqrt{7}< \gamma\le 0$ and $\kappa$ as in (\ref{kappa}). We start with  the $L^\infty$ bound and then show regularity in Sobolev spaces.  Because the proof of the $L^\infty$ bound requires different approaches depending on whether the potential is moderately or very soft, we treat these cases separately. Finally, the higher regularity estimates are consolidated for all soft potentials in Theorem \ref{thm.reg}.  
  
\subsection{Boundedness of the solution.}\label{sec:bound}

\begin{lemma}[Boundedness, $-2<\gamma\le 0$] \label{lem.Linf}
Let $f\ge 0$ be a smooth solution to \eqref{1} with initial condition $f_{in}$ that satisfies the following bounds: 
\begin{align*}
    \rho_0 & = \int_{\R^6}f_{in}(x,v) \; \dd x\dd v, \\
	c_{\mm} &=  \int_{\R^6}(|x|^\mm + |v|^\mm)f_{in}(x,v)\; \dd x\dd v,\\
	h_0 & = \int_{\R^6} f_{in}\log f_{in} \; \dd x\dd v,
\end{align*}
for {{$\mm \ge 4+ 2\lambda$}},
{ $\mm \geq \frac{3|\gamma|}{3-|\gamma|}$. }
For each $T\ge 0$ we have: 
\begin{align}\label{g.neg.Linf}
	\sup_{t\in [0,T]}\|f(t)\|_{L^\infty_{x,v}}\leq C(T)\|f_{in}\|_{L^\infty_{x,v}}. 
\end{align}
\end{lemma}
\begin{proof}
For a generic $p>1$ we employ $f^{p-1}$ as test function in \eqref{eq:Landau_Aa}, obtaining after  some integrations by parts
\begin{align}\label{g.neg1}
\frac{d}{dt}\int_{\R^6}& f^p\dd x\dd v
+\frac{4(p-1)}{p}\int_{\R^6}
\nabla_{v}f^{p/2}\cdot A[f]\nabla_{v} f^{p/2}\dd x\dd v
=-(p-1)\int_{\R^6}f^{p}\Delta a[f]\dd x\dd v.
\end{align}
The second integral on the left-hand side can be controlled via \eqref{aA.lb}  as follows
\begin{align}\label{g.neg.1}
\int_{\R^6}
\nabla_{v}f^{p/2}\cdot A[f]\nabla_{v} f^{p/2}\dd x\dd v\geq 
\theta\int_{\R^6}\bk{x}^{-\lambda}
\bk{v}^{\gamma}|\nabla_{v}f^{p/2}|^2\dd x\dd v.
\end{align}
We turn now our attention to the integral on the right-hand side of \eqref{g.neg}:
\begin{align*}
	-\int_{\R^6} f^{p}\Delta a[f]\dd x\dd v 
	&= -\int_{\R^6}f^{p}\bk{v}^{\gamma}\bk{x}^{-\lambda}\, \bk{x}^{\lambda}\bk{v}^{-\gamma}\Delta a[f]\dd x\dd v.
\end{align*}
We observe that 
\begin{align}\label{Delta.a}
-\Delta a[f] &= 2(\gamma+d)\int_{\R^6}|v-w|^{\gamma}\kappa(x-y)f(y,w)\dd y\dd w\\ \nonumber
&\leq 2(\gamma+d)\kappa_2\int_{\R^6}|v-w|^{\gamma}\bk{x-y}^{-\lambda} f(y,w)\dd y\dd w .
\end{align}
We employ the straightforward relations
\begin{align*}
\bk{\xi}\leq \bk{\zeta} + |\xi-\zeta|,\quad 
\bk{\xi}\leq |\zeta| + \bk{\xi-\zeta},
\end{align*} 
which imply
\begin{align*}
\bk{x}^{\lambda}\leq C(\bk{x-y}^{\lambda} + |y|^{\lambda}),\quad 
\bk{v}^{-\gamma}\leq C(|v-w|^{-\gamma} + \bk{w}^{-\gamma}).
\end{align*}
Thus we obtain
\begin{align*}
-&\bk{x}^{\lambda}\bk{v}^{-\gamma}\Delta a[f] \\
&\leq
C\int_{\R^6}\left[
1 + |v-w|^{\gamma}\bk{w}^{-\gamma} + |y|^\lambda\bk{x-y}^{-\lambda} 
+ |v-w|^{\gamma}\bk{x-y}^{-\lambda}\bk{w}^{-\gamma}|y|^\lambda
\right]f(y,w)\dd y\dd w,
\end{align*}
which leads to
\begin{align}\label{bad.1}
-\bk{x}^{\lambda}\bk{v}^{-\gamma}\Delta a[f] 
\leq C +
C\int_{\R^6}|v-w|^{\gamma}\bk{w}^{-\gamma}\bk{y}^\lambda
f(y,w)\dd y\dd w
\end{align}
thanks to mass conservation and \eqref{mom.high}. Consider now
\begin{align*}
&\int_{\R^6}|v-w|^{\gamma}\bk{w}^{-\gamma}\bk{y}^\lambda
f(y,w)\dd y\dd w \\
&= \int_{\R^{d}}
\int_{B_1(v)}|v-w|^{\gamma}\bk{w}^{-\gamma}f(y,w)\dd w~
\bk{y}^\lambda \dd y 
+ 
\int_{\R^{d}}
\int_{\R^{d}\backslash B_1(v)}|v-w|^{\gamma}\bk{w}^{-\gamma}f(y,w)\dd w~
\bk{y}^\lambda \dd y.
\end{align*}
Given that $\mm\geq \lambda +3$,
the second integral can be estimated via Young's inequality as follows
\begin{align*}
\int_{\R^{d}}
\int_{\R^{d}\backslash B_1(v)}|v-w|^{\gamma}\bk{w}^{-\gamma}f(y,w)\dd w~
\bk{y}^\lambda \dd y
\leq C\int_{\R^6}
(1 + \bk{w}^\mm + |y|^\mm)f(y,w)\dd w \dd y,
\end{align*}
which is bounded thanks to \eqref{mom.high}. Hence \eqref{bad.1} becomes
\begin{align*}
-\bk{x}^{\lambda}\bk{v}^{-\gamma}\Delta a[f] 
\leq C +
C\int_{\R^{d}}
\int_{B_1(v)}|v-w|^{\gamma}\bk{w}^{-\gamma}f(y,w)\dd w~
\bk{y}^\lambda \dd y ,
\end{align*}
that is
\begin{align*}
	-\bk{x}^{\lambda}\bk{v}^{-\gamma}\Delta a[f] 
	\leq C +
	C\int_{\R^{d}}g(y,v) \bk{y}^\lambda \dd y,\quad \\
    \end{align*}
    with 
    \begin{align*}
g(y,v):= \int_{B_1(v)}|v-w|^{\gamma}\bk{w}^{-\gamma}f(y,w)\dd w .
\end{align*}
It follows
\begin{align*}
-\int_{\R^6} & f^p \Delta a[f]\dd x\dd v\\
\leq &
C\int_{\R^6}f^p \bk{v}^\gamma\bk{x}^{-\lambda}\dd x\dd v + 
C\int_{\R^9}f^p(x,v) \bk{v}^\gamma\bk{x}^{-\lambda}g(y,v)\bk{y}^\lambda \dd x\dd y\dd v.
\end{align*}
Fix now $q$ such that $\frac{3}{3+\gamma}<q<3$ (remember that $\gamma>-2$). H\"older inequality in the velocity variable with exponent $\{q, q'\}$ yields
\begin{align*}
-\int_{\R^6} & f^p \Delta a[f]\dd x\dd v\\
\leq &\; 
C\int_{\R^6}f^p \dd x\dd v + 
C\int_{\R^6}\|f^p(x,\cdot) \bk{\cdot}^\gamma\|_{L^q_v}
\|g(y,\cdot)\|_{L^{\frac{q}{q-1}}_v}\bk{y}^\lambda \bk{x}^{-\lambda}\dd x\dd y\\
=& 
\; C\int_{\R^6}f^p \dd x\dd v + 
C\left[\int_{\R^{d}}\|f^p(x,\cdot) \bk{\cdot}^\gamma\|_{L^q_v}\bk{x}^{-\lambda}\dd x\right]
\left[\int_{\R^{d}}\|g(y,\cdot)\|_{L^{\frac{q}{q-1}}_v}\bk{y}^\lambda\dd y\right].
\end{align*}

Let us now estimate $h$. The bound
\begin{align*}
0\leq g(y,v)\leq \left(\int_{B_1(v)}|v-w|^{\gamma \frac{q}{q-1}}\bk{w}^{-\gamma}f(y,w)\dd w \right)^{\frac{q-1}{q}}\left(\int_{B_1(v)}\bk{w}^{-\gamma}f(y,w)\dd w\right)^{\frac{1}{q}}
\end{align*}
implies
\begin{align*}
\int_{\R^{d}}& g(y,v)^{\frac{q}{q-1}}\dd v\\
&\leq 
\left(\int_{\R^3}\bk{w}^{-\gamma}f(y,w)\dd w\right)^{\frac{1}{q-1}}
\int_{\R^{d}}\int_{B_1(v)}|v-w|^{\gamma \frac{q}{q-1}}\bk{w}^{-\gamma}f(y,w)\dd w\dd v\\
&=
\left(\int_{\R^3}\bk{w}^{-\gamma}f(y,w)\dd w\right)^{\frac{1}{q-1}}
\int_{\R^{d}}{\underbrace{\left(\int_{B_1(w)}|v-w|^{\gamma \frac{q}{q-1}}\dd v \right)}_{=C_{\gamma,q}}}~ \bk{w}^{-\gamma}f(y,w)\dd w\\
&=C_{\gamma,q}\left(\int_{\R^3}\bk{w}^{-\gamma}f(y,w)\dd w\right)^{\frac{q}{q-1}}.
\end{align*}
This implies
\begin{align*}
\int_{\R^{d}}\|g(y,\cdot)\|_{L^{\frac{q}{q-1}}_v}\bk{y}^\lambda\dd y
\leq C_{\gamma,q}\int \bk{y}^\lambda \bk{w}^{-\gamma}f(y,w)\dd w\dd y \le \tilde{C}_{\gamma,q},
\end{align*}
once again via Young's inequality and \eqref{mom.high}.
Summarizing, we obtain
\begin{align*}%\label{bad.3}
-\int_{\R^6} & f^{p}\Delta a[f]\dd x\dd v 
\leq C\int_{\R^6}f^p \dd x\dd v
+C\int_{\R^{d}}\|f^p(x,\cdot)\bk{\cdot}^\gamma\|_{L^q_v}\bk{x}^{-\lambda}\dd x .
\end{align*}
Let us now turn our attention to 
\begin{align*}
\|f^p(x,\cdot)\bk{\cdot}^\gamma\|_{L_v^q} = 
\|f^{\frac{p}{2}}(x,\cdot)\bk{\cdot}^{\frac{\gamma}{2}}\|_{L_v^{2q}}^2.
\end{align*}
Gagliardo-Nirenberg-Sobolev inequality leads to (remember that $q<3$)
\begin{align*}
\|f^p(x,\cdot)\bk{\cdot}^\gamma\|_{L_v^q}\leq 
C\|\nabla_v(f^{\frac{p}{2}}(x,\cdot)\bk{\cdot}^{\frac{\gamma}{2}})\|_{L_v^2}^{2\theta}\;
\|f^{\frac{p}{2}}(x,\cdot)\bk{\cdot}^{\frac{\gamma}{2}}\|_{L_v^2}^{2(1-\theta)}
\end{align*}
with 
$$
\frac{1}{{{2}}q} = \frac{\theta}{6} + \frac{(1-\theta)}{2} = \frac{1}{2} -\frac{\theta}{3},
$$ 
or equivalently 
$$
\theta = \frac{3(q-1)}{2q}.
$$
Young's inequality allows us to deduce
\begin{align*}
\|f^p(x,\cdot)\bk{\cdot}^\gamma\|_{L_v^q}\leq &\; 
\varepsilon\|\nabla_v(f^{p/2}(x,\cdot)\bk{\cdot}^{\gamma/2})\|_{L_v^2}^{2} + 
C(\varepsilon)\|f^{p/2}(x,\cdot)\bk{\cdot}^{\gamma/2}\|_{L_v^2}^{2}\\
\leq & \; \varepsilon\|\bk{\cdot}^{\gamma/2}\nabla_v f^{p/2}(x,\cdot)\|_{L_v^2}^{2} + 
C(\varepsilon)\int_{\R^{d}}f^p(x,v)\dd v,
\end{align*}
therefore
\begin{align}\label{g.neg.2}
-\int_{\R^6} f^{p}\Delta a[f]\dd x\dd v 
\leq C(\varepsilon)\int_{\R^6}f^p \dd x\dd v 
+ \varepsilon\int_{\R^6}|\nabla_{v}f^{p/2}|^2 \bk{v}^{\gamma}\bk{x}^{-\lambda}\dd x\dd v.
\end{align}
Combining \eqref{g.neg1} with \eqref{g.neg.1} and \eqref{g.neg.2}, and choosing 
$\varepsilon = \frac{2\theta}{p}$ and $C(\varepsilon) \simeq {\left(\frac{p}{2\theta}\right)^{ \frac{3(q-1)}{3-q} } }$ lead to
\begin{align}\label{g.neg.3}
	\frac{d}{dt}\int_{\R^6}& f^p\dd x\dd v
	+\frac{2(p-1)}{p}\theta\int_{\R^6}
	|\nabla_{v}f^{p/2}|^2 \bk{v}^\gamma\bk{x}^{-\lambda}\dd x\dd v
	\lesssim \; { p^{ \frac{3(q-1)}{3-q} + 1} }\int_{\R^6} f^p\dd x\dd v,
\end{align}
which implies {via Gronwall's Lemma}
\begin{align*}
\sup_{t\in [0,T]}\|f(t)\|_{L^p_{x,v}}\leq C(p,T)\|f_{in}\|_{L^p_{x,v}}\quad\forall \; p>1,~~ T>0,
\end{align*}
{
with $C(p,T)\equiv \exp(C p^{\frac{3(q-1)}{6-q}}T)$.
Notice that $C(p,T)\to\infty$ as $p\to\infty$ for $q>1$, hence we cannot immediately deduce the boundedness of $f$. However, now that we have a bound for $f$ in $L^\infty(0,T; L^p(\R^3))$ for every $p<\infty$, inequality \eqref{Delta.a} yields via straightforward estimates that 
$$\|\Delta a[f]\|_{L^\infty_v(\R^6)}\lesssim 
\|f\|_{L^p(\R^3)}
+\int_{\R^6}\bk{y}^{r}f\dd y\dd w,\quad r > \frac{3|\gamma|}{3-|\gamma|},\quad p<\infty\mbox{ large enough.} $$
Therefore we deduce from \eqref{g.neg1} via Gronwall's Lemma that 
\begin{align*}
\sup_{t\in [0,T]}\|f(t)\|_{L^p_{x,v}}\leq e^{KT}\|f_{in}\|_{L^p_{x,v}}\quad\forall \; p>1,~~ T>0,
\end{align*}
for some $K>0$ independent of $p$. Hence taking the limit $p\to\infty$ in the above inequality leads to
}
\begin{comment}
Note that $C(p,T)$ stays bounded when $p\to +\infty$.  Since $f_{in}\in L^\infty_{x,v}\cap L^1_{x,v}$, we can pass to the limit $p\to + \infty$ in the above inequality and obtain: 
\end{comment}
$$
\sup_{t\in [0,T]}\|f(t)\|_{L^\infty_{x,v}} \le C(T) \|f_{in}\|_{L^\infty_{x,v}}. 
$$

This finishes the proof.
\end{proof}
Having established the $L^\infty$
  bound for moderately soft potentials, we now extend this boundedness result to the regime of very soft potentials. 

\begin{lemma}[Boundedness, $-{{\sqrt{8}}}<\gamma\le -2$] \label{lem:bounded-verysoft}
Let $f$ be a smooth solution to (\ref{1}) with initial condition $f_{in}{\in L^\infty_{x,v}}$ that satisfies the following bounds: 
\begin{align*}
    \rho_0 & = \int_{\R^6}f_{in}(x,v) \; \dd x\dd v, \\
	c_{\mm} &=  \int_{\R^6}(|x|^\mm + |v|^\mm)f_{in}(x,v)\; \dd x\dd v,\\
	h_0 & = \int_{\R^6} f_{in}\log f_{in} \; \dd x\dd v, \\
    c_{in} & = \int_{\R^6} |\nabla_x \sqrt{f_{in}}|^2 +|\nabla_v \sqrt{f_{in}}|^2\;dvdx,
\end{align*}
for
$$
\mm\geq \max \left\{  4+ 2\lambda,\;  \frac{ 3 |\gamma|}{3+\gamma},\;  
%\frac{ 2 |\gamma|}{3+\gamma}, 2|\gamma|, \; 
\frac{ 6 \lambda + 8 |\gamma|-16}{3(3+\gamma)}\; \right\}. 
$$

There exists a constant $C(T)$ that only depends on $T$, $\gamma$, $\rho_0$,$c_{\mm}$, $c_{in}$ and $h_0$ such that 
$$
\|f\|_{{L^\infty(0,T; L^{\infty}_{x,v})}} \le C(T).
$$
\end{lemma}

\begin{proof}

Recall first that $f$, by Theorem \ref{Theo_Fisher_xv}, has bounded Fisher information, and consequently the $\varepsilon$-Poincare inequality (\ref{eps_poinc_b}) hold for our $f$.

For a generic $p>1$ we employ $f^{p-1}$ as test function in \eqref{eq:Landau_Aa}, obtaining after some straightforward integrations by parts the following identity:
\begin{align}\label{h.neg}
\frac{d}{dt}\int_{\R^6}& f^p\dd x\dd v
+\frac{4(p-1)}{p}\int_{\R^6}
\nabla_{v}f^{p/2}\cdot A[f]\nabla_{v} f^{p/2}\dd x\dd v
=-(p-1)\int_{\R^6}f^{p}\Delta a[f]\dd x\dd v,
\end{align}
where 
\begin{align*}
-\Delta a[f] = 2(\gamma+d)\int_{\R^6}|v-w|^{\gamma}\kappa(x-y)f(y,w)\dd y\dd w = 2(\gamma+d)\int_{\R^3}|v-w|^{\gamma} H(x,w,t) \;dw. 
\end{align*}
	
The second integral on the left-hand side can be controlled via \eqref{aA.lb} as follows
\begin{align}\label{h.neg.1}
\int_{\R^6}
\nabla_{v}f^{p/2}\cdot A[f]\nabla_{v} f^{p/2}\dd x\dd v\geq 
\alpha\int_{\R^6} {{\bk{x}^{-\lambda}}}
\bk{v}^{\gamma}|\nabla_{v}f^{p/2}|^2\dd x\dd v. 
\end{align}

Since $f$ has bounded Fisher information,  $f$ also  satisfies the $\varepsilon$-Poincare inequality (\ref{eps_poinc_b}) thanks to  Theorem \ref{thm:eps_Poinc}, Theorem \ref{Theo_Fisher_xv}, and bounded mass and moments (Lemma \ref{ss.moments}). Hence, the right hand side of (\ref{h.neg}) is exactly the first term in the $\varepsilon$-Poincare inequality (\ref{eps_poinc_b}):

\begin{align}
-(p-1)\iint f^{p}\Delta a[f]\dd x\dd v = &  \;  2(\gamma+d)(p-1)  \iint f^p (H \ast_v|v|^{\gamma}) \;dxdv \nonumber \\
\le &\;  2(\gamma+d)(p-1) \left[  \varepsilon \iint   {{\langle x \rangle^{- \lambda}}} \left|  \nabla_v ( \langle v \rangle^{\frac{\gamma}{2}} f^{\frac{p}{2}})\right|^2 \;dvdx + D_0(T)  \iint f^p  \langle v \rangle^{{\gamma}}  \;dxdv\right] \nonumber \\
\le & \; 4 (\gamma+d)(p-1) \varepsilon  \iint   {{\langle x \rangle^{- \lambda}}} \langle v \rangle^{{\gamma}} \left|  \nabla_v (  f^{\frac{p}{2}})\right|^2 \;dxdv \nonumber \\
&\;  + 4 (\gamma+d)(p-1) \varepsilon  \iint  {{\langle x \rangle^{- \lambda}}} f^p\; | \nabla \langle v \rangle^{\frac{\gamma}{2}}|^2  \;dxdv  \nonumber\\
&\;  + 2(\gamma+d)(p-1) D_0(T)  \iint f^p  \langle v \rangle^{{\gamma}}  \;dxdv. \label{h.neg.2}
    \end{align}
Substituting (\ref{h.neg.1}) and (\ref{h.neg.2}) into (\ref{h.neg}) we obtain
\begin{align}
\frac{d}{dt}\int_{\R^6} f^p\dd x\dd v \; + \; & 4(p-1)\left[  \frac{\alpha}{p} -  \varepsilon(\gamma+d) \right] \int_{\R^6} {{\bk{x}^{-\lambda}}} \bk{v}^{\gamma}|\nabla_{v}f^{p/2}|^2\dd x\dd v \\
& \le  4(p-1) D_0(T)  \iint f^p   \;dxdv
    \end{align}
For $\varepsilon \le \min\{ 1, \frac{\alpha}{p(\gamma+d)}\}$, since $D_0(T):= C_0(T)\left(1+  \varepsilon ^{-\frac{s}{1-s}}\right)$ we obtain 
$$
\frac{d}{dt}\int_{\R^6} f^p\dd x\dd v \le c(p) \int_{\R^6} f^p\dd x\dd v
$$
with $c(p) \sim p^{M}$ with {$M>1$ suitable constant. In particular, via Gronwall's Lemma $f\in L^\infty(0,T; L^p(\R^6))$ for every $p<\infty$. At this point arguing like in the final part of the proof of Lemma \ref{lem.Linf} yields the desired statement.}

\begin{comment}
THIS IS WRONG.

a big but finite positive number, and consequently 
$$
\frac{d}{dt} \|f\|_{L^p} \le \sqrt[\leftroot{-6}\uproot{6}p]{c(p)}  \|f\|_{L^p} , \quad \forall \; p>1.
$$
Gronwall's inequality implies 
$$
 \|f\|_{L^p} \le e^{\sqrt[\leftroot{-6}\uproot{6}p]{c(p)} t}\|f_{in}\|_{L^p}
$$
and the thesis follows by taking $p \to +\infty$. 
\end{comment}

\end{proof}

A useful consequence of the above result is the following corollary.

\begin{cor}\label{cor.Linf} 
{Let $\gamma \geq -3$.}
Assume 
$$f\in L^{\infty}_{t}(0,T, L^1_{x,v}\cap L^\infty_{x,v}), \quad f\in L_t^\infty(0,T, L^1_{\mm}(\mathbb{R}^6)),$$ 
with
\begin{align}\label{lb.m}
%\mm \ge  3\lambda + 14.
{\mm > \max\left\{ 3(\lambda+5), \frac{3|\gamma|}{3+\gamma} \right\} . }
\end{align}
Furthermore assume $\kappa\in L^1(\R^3)$ if $\gamma=-3$.
We have the following bounds: 
\begin{align}\label{a.decay}
|A[f]|& \leq C\bk{x}^{-\lambda}\bk{v}^{\gamma+2}, 
\\
{}\nonumber \\
\label{Da.decay}
\quad 
|\nabla_{v}A[f]|+|\nabla_{v}a[f]|& \leq C\bk{x}^{-\lambda}\bk{v}^{\gamma+1},\\
{}\nonumber\\
\quad 
|\Delta_v a[f]|& \leq C.\label{Delta.a.bounded}
\end{align}
\end{cor}
\begin{proof}

It holds
\begin{align*}
|\nabla_{v}A[f]|+|\nabla_{v}a[f]| &\leq C\int_{\R^3}|v-w|^{\gamma+1}H(x,w)\dd w,
\end{align*}
with $H(x,v)=\int_{\R^3}\kappa(x-y)f(y,v)\dd y$.

If $\gamma+1\geq 0$ then 
$$
\int_{\R^3}|v-w|^{\gamma+1}H(x,w)\dd w \le \kappa_2 \left( \bk{v}^{\gamma+1}\rho_0 +  \|f\|_{L^1_{\mm}(\mathbb{R}^6)}\right). 
$$

If $-2\leq \gamma+1<0$, bound \eqref{Da.decay} follows from {Corollary \ref{coro.DH.est.1b}}. In fact, substitute $\epsilon=1$, $\sigma = 3$ and $ r = -\gamma -1 \in (0,2]$ in \eqref{DH.est.1b} and obtain: 
\begin{align*}
&\left|\int_{\R^3}{H(x,w)}{|v-w|^{\gamma+1}}\dd w\right|\\
&\leq C\frac{\bk{v}^{\gamma+1}}{\bk{x}^\lambda}\Big(
\|f\|_{L^{ \frac{p-\theta}{1-\theta} }(\R^6)} 
+ \int_{\R^6}\Big(
\bk{y}^{p(\lambda + 3)/\theta} + 
\bk{w}^{\frac{6(p-1)}{p-3} {\frac{p}{p-\theta}} }\Big) 
f(y,w)\dd y\dd w
\Big)
\end{align*}
for  any arbitrary choice of $p > 3$ and $0<\theta<1$.
Interpolating $L^{ \frac{p-\theta}{1-\theta} }(\R^6)$
between $L^1(\R^6)$ and $L^\infty(\R^6)$ yields
\begin{align*}
&\left|\int_{\R^3}{H(x,w)}{|v-w|^{\gamma+1}}\dd w\right|\\
&\leq C\frac{\bk{v}^{\gamma+1}}{\bk{x}^\lambda}\Big(
\|f\|_{L^{ \infty }(\R^6)} 
+ \int_{\R^6}\Big(
\bk{y}^{p(\lambda + 3)/\theta} + 
\bk{w}^{\frac{6(p-1)}{p-3}{\frac{p}{p-\theta}} }\Big) 
f(y,w)\dd y\dd w
\Big)
\end{align*}
{

Estimating $\frac{6(p-1)}{p-3}\frac{p}{p-\theta} < \frac{6p}{p-3}$ and choosing $p = \frac{3(\lambda+5)}{\lambda+3}$ yields
\begin{align*}
&\left|\int_{\R^3}{H(x,w)}{|v-w|^{\gamma+1}}\dd w\right|\\
&\leq C\frac{\bk{v}^{\gamma+1}}{\bk{x}^\lambda}\Big(
\|f\|_{L^{ \infty }(\R^6)} 
+ \int_{\R^6}\Big(
\bk{y}^{\frac{3(\lambda + 5)}{\theta}} + 
\bk{w}^{3(\lambda + 5)} \Big) 
f(y,w)\dd y\dd w
\Big)
\end{align*}
Given that \eqref{lb.m} holds, we can 
choose $\theta\geq \frac{3(\lambda + 5)}{\mm}$ to obtain the desired bound for the right-hand side of the above inequality.
}
\begin{comment}
Straightforward algebraic manipulations show that for 
$$\tilde{p}:=\frac{1}{2(\lambda+3)}\left(
3(\lambda+5)+\sqrt{ 9\lambda^2 + 66\lambda + 153 }
\right)
$$
We have $$ \tilde p(\lambda + 3) = \frac{6(\tilde p-1)}{\tilde p-3}.$$ Moreover 
$$\tilde p(\lambda+3) < 3\lambda+14.$$ 
Hence choosing $\theta\in (0,1)$ close enough to 1 yields 
$$
\frac{\tilde p(\lambda + 3)}{\theta} \le 3\lambda+14.
$$
\end{comment}
%
This finishes the proof of (\ref{Da.decay}). \\

The bound for $\Delta_v a[f]$ in the case $\gamma>-3$ is rather standard:
{for $\frac{3}{3+\gamma}<p<\frac{\mm + 3}{3}$ it holds}

%{Why is there a power $p$ in the $L^p$ norm?
%Probably need to argue like in the final part of the proof of Lemma \ref{lem.Linf}
%}
\begin{align*}
|\Delta_v a[f]| =&  \; C_\gamma \int_{\R^3} \kappa(x-y) \int_{\R^3} |v-w|^{\gamma}f(y,w) \; \dd w \dd y \\
\leq & \; C_\gamma \int_{\R^3} \kappa(x-y)   \left( \|f\|_{L_v^1(\R^3)} + \|f\|_{L_v^p(\R^3)}\right)\;dy \\
\leq & \; C_\gamma \kappa_2 \rho_0 + C_\gamma \kappa_2 
\| f\|^{(p-1)/p}_{L_{x,v}^\infty}
{\int_{\R^3}\bk{x}^{-\frac{\mm}{p}}\bk{x}^{\frac{\mm}{p}}\|f\|_{L^1_v(\R^3)}^{\frac{1}{p}}\dd x}\\
\leq & \; C_\gamma \kappa_2 \rho_0 + C_\gamma \kappa_2 
\| f\|^{(p-1)/p}_{L_{x,v}^\infty}
{
\left(
\int_{\R^3}\bk{x}^{-\frac{\mm}{p-1}}\dd x
\right)^{1-\frac{1}{p}}
\left(
\int_{\R^6}\bk{x}^{\mm}f\dd x\dd v
\right)^{\frac{1}{p}}
\leq C.}
\end{align*}
The case $\gamma=-3$ follows from the explicit formula 
$$-\Delta_v a[f](x,v)= \int_{\R^3} \kappa(x-y)f(y,v) \;dy, $$ and  the estimate 
$$
| \Delta_v a[f](x,v)| \le \| \kappa\|_{L^1(\mathbb{R}^3)} \|f\|_{L_{x,v}^\infty}. 
$$
This shows (\ref{Delta.a.bounded}) and finishes the proof of the lemma.

\end{proof}

The following is an auxiliary result which will be used in the $H^s$ regularity argument to control the contribution of the transport term.
\begin{lemma}\label{lem.trans}
For every $m\in\N$, $\alpha, \beta\in\N^3$ there exists a constant $C>0$ such that
\begin{align}\nonumber
\|&(1+|x|^2+|v|^2)^m D_v^\alpha D_x^\beta f\|_{L^2(\R^6)}\\
&\leq 
C\sum_{j=1}^3\sum_{k=1}^{|\alpha|+|\beta|}(
\|(1+|x|^2+|v|^2)^m\pa_{x_j}^k f\|_{L^2(\R^6)}
+\|(1+|x|^2+|v|^2)^m\pa_{v_j}^k f\|_{L^2(\R^6)})
\label{est.trans}
\end{align}
for every $f\in C^\infty_c(\R^6)$.
\end{lemma}
\begin{proof}
We denote with $\hat{f}$ the Fourier transform of $f$.
We have that 
\begin{align*}
\|(1+|x|^2+|v|^2)^m D_v^\alpha D_x^\beta f\|_{L^2(\R^6)}^2
&\lesssim 
\|(1+|v|^{2m}+|x|^{2m})D_v^\alpha D_x^\beta f\|_{L^2(\R^6)}^2\\ 
&= \|(1+(-\Delta_v)^{m}+(-\Delta_x)^{m})(v^\alpha x^\beta\hat{f})
\|_{L^2(\R^6)}^2 
\end{align*}
Moreover
\begin{align}\label{est.trans.1}
(1+(-\Delta_v)^{m}+(-\Delta_x)^{m})(v^\alpha x^\beta\hat{f})
= v^\alpha x^\beta (1+(-\Delta_v)^{m}+(-\Delta_x)^{m})\hat{f}
+ \mbox{l.o.t.},
\end{align}
where ``l.o.t.'' collects all the lower order terms with respect to the derivatives of $\hat{f}$, each of them having one of the two forms:
$$
v^{\alpha-\alpha'}x^\beta D_v^{2m-\alpha'}\hat{f}
\quad\mbox{or}\quad 
v^{\alpha}x^{\beta-\beta'} D_x^{2m-\beta'}\hat{f},
$$
for some $\alpha', \beta'\in\N^3$, $|\beta'|<|\beta|, 
|\alpha'|<|\alpha|$. We first estimate the highest order term;  Young's inequality yields
\begin{align*}
\|&v^\alpha x^\beta (1+(-\Delta_v)^{m}+(-\Delta_x)^{m})\hat{f}\|_{L^2(\R^6)}^2 \\
&= \int_{\R^6}
|x|^{2\beta} |v|^{2\alpha} 
|(1+(-\Delta_v)^{m}+(-\Delta_x)^{m})\hat{f}|^2\dd x\dd v\\ 
&\lesssim\int_{\R^6}(|x|^{2(|\alpha|+|\beta|)}+|v|^{2(|\alpha|+|\beta|)})
|(1+(-\Delta_v)^{m}+(-\Delta_x)^{m})\hat{f}|^2\dd x\dd v\\ 
&\lesssim\sum_{j=1}^3\left( 
\|x_j^{|\alpha|+|\beta|} (1+(-\Delta_v)^{m}+(-\Delta_x)^{m})\hat{f}\|_{L^2}^2 + \|v_j^{|\alpha|+|\beta|} (1+(-\Delta_v)^{m}+(-\Delta_x)^{m})\hat{f}\|_{L^2}^2\right).
\end{align*}
Applying Fourier transform once again leads to
\begin{align*}
\|&v^\alpha x^\beta (1+(-\Delta_v)^{m}+(-\Delta_x)^{m})\hat{f}\|_{L^2(\R^6)}^2\\ 
&\lesssim\sum_{j=1}^3\left( 
\|\pa_{x_j}^{|\alpha|+|\beta|}(1+|v|^{2m}+|x|^{2m})f\|_{L^2}^2 
+ \|\pa_{v_j}^{|\alpha|+|\beta|} (1+|v|^{2m}+|x|^{2m})f\|_{L^2}^2
\right)
\end{align*}
which is controlled by the right-hand side of \eqref{est.trans}.
All the lower order terms in \eqref{est.trans.1} are estimated in a similar way. This finishes the proof of the Lemma.
\end{proof}

\subsection{$H^s$ regularity.}We now investigate the regularity of $f$ in Sobolev spaces. In Section \ref{sec:bound}, we established that 
\begin{align*}
f\in L^\infty(0,T; L^1_{\mm}\cap L^\infty(\R^3\times\R^3)),\quad \bk{x}^{-\lambda/2}\bk{v}^{\gamma/2}f\in 
L^2(0,T; H^1(\R^3\times\R^3)).
\end{align*}
From now on, we assume that in addition to satisfying \eqref{hp.ic} and \eqref{lambda,m}, the initial datum $f_{in} \in H^s(\mathbb{R}^6)$ satisfies 
\begin{align}\label{Hkreg.ic.2}
(1+|x|^2+|v|^2)^{s_0+s-|\alpha| - |\beta|} D_v^{\alpha} D^{\beta}_x f_{in}\in L^2(\R^6)
\end{align}
for all  $\alpha, \beta\in \N_0^d$ such that $|\alpha|+|\beta| \leq s $,
{and 
\begin{align}
\label{hp.s.s0}
s_0 \geq \frac{3\lambda + 3 -\gamma}{2}.
\end{align}
}
Moreover, the spatial kernel $\kappa$ satisfies (\ref{kappa}) and 
\begin{align}\label{Hkreg.kappa}
	|D_x^\beta \kappa |\leq C_\beta \kappa ~~\mbox{in }\R^3 , \quad \forall \; |\beta|\le s . 
\end{align}

We aim at showing that
\begin{align}\label{Hkreg.goal}
(1+|x|^2+|v|^2)^{s_0+s-|\alpha| - |\beta|} D_v^{\alpha} D^{\beta}_x f\in L^\infty(0,T; L^2(\R^3\times\R^3)), \nonumber \\
\bk{v}^{\gamma/2}\bk{x}^{-\lambda/2}
(1+|x|^2+|v|^2)^{s_0+s -|\alpha| - |\beta|} D_v^{\alpha} D^{\beta}_x \nabla_{v}f\in L^2(0,T; L^2(\R^3\times\R^3)) ,
\end{align}
for every choice of $\alpha, \beta\in \N_0^d$ with $|\alpha|+|\beta|\leq s$.\\

In light of Lemma \ref{lem.trans}, it suffices to prove 
\eqref{Hkreg.goal} for the two cases 
$(\alpha,\beta)=(k e^{(i)},0)$ and $(\alpha,\beta)=(0,k e^{(i)})$ for $k=0,\ldots,s$. In other words, we can restrict our analysis to the case of pure derivatives without loss of generality. 

{
We show a preliminary bound for the spatial derivatives of the matrix $A[f]$.
\begin{lemma}\label{lem.DA.estim}
Assume \eqref{Hkreg.kappa}; then the following bound holds:
\begin{align}\label{est.DA.exp}
\forall j\leq s,~~
\exists K_{j}>0:\quad 
\sum_{i=1}^3|A[f]^{-1/2} (\pa_{x_i}^j A[f]) A[f]^{-1/2}|\leq K_j,\quad\forall x,v\in\R^3.
\end{align}

%On the other hand, under the alternative assumption of $\kappa$
%\begin{align}
%\label{hp.kappa.Gauss}
%\forall j\leq s,~~
%\exists C_j>0:\quad 
%\sum_{i=1}^3|\pa_{x_i}^j\kappa(x)|
%\leq C_j\bk{x}^j\kappa(x),\quad\forall x\in\R^3
% 

\end{lemma}
\begin{proof}
 Let \eqref{Hkreg.kappa} be fulfilled, we prove \eqref{est.DA.exp}.
Let us consider ($C_\beta$ as in \eqref{Hkreg.kappa})
\begin{align*}
C_{j} A[f] \mp \pa_{x_i}^{j} A[f] 
= \int_{\R^6}N(v-w)(C_{j} \kappa(x-y) \mp 
\pa_{x_i}^{j} \kappa(x-y))f(y,w)\dd y\dd w.
\end{align*}
Since \eqref{Hkreg.kappa} holds, it follows
\begin{align*}
0\leq \xi\cdot (C_{j} A[f] \mp  
\pa_{x_i}^{j} A[f])\xi = 
C_{j} \xi\cdot A[f]\xi \mp  \xi\cdot 
\pa_{x_i}^{j} A[f] \xi \quad\forall\xi\in\R^3,
\end{align*}
hence
\begin{align*}
-C_{j} |\xi|^2\leq 
\xi \cdot A[f]^{-1/2}\, \pa_{x_i}^{j} A[f]\, A[f]^{-1/2}\xi \leq C_{j} |\xi|^2 \quad\forall\xi\in\R^3,
\end{align*}
implying the boundedness of $A[f]^{-1/2}\, 
\pa_{x_i}^{j} A[f]\, A[f]^{-1/2}$. Thus \eqref{est.DA.exp} is proved.

\end{proof}
}

We start now the proof of (\ref{Hkreg.goal}). 
We begin with the spatial derivatives, setting $|\alpha|=0$ and  $\beta=k e^{(i)}$.
Applying the differential operator  $\pa_{x_i}^k$ to \eqref{1} yields
\begin{align}\label{Hkreg.1}
\pa_t \pa_{x_i}^k f + v\cdot\nabla_{x}\pa_{x_i}^k f = 
\Div_v(
A[f]\nabla_{v}\pa_{x_i}^k f - \pa_{x_i}^k f \, \nabla_{v}a[f])
+ E_{x_i}^{k},
\end{align}
with the lower order terms given by
\begin{align*}
E_{x_i}^{k} 
= \sum_{\ell=0}^{k-1}
\binom{k}{\ell}\Div_v\left(
\pa_{x_i}^{k-\ell} A[f]\, 
\pa_{x_i}^\ell\nabla_{v} f
- \pa_{x_i}^{k-\ell} \nabla_{v}a[f]\, 
\pa_{x_i}^\ell f
\right).
\end{align*}
Define the weight function 
\begin{align}\label{weigth}
b_m(x,v) := (1+|x|^2+|v|^2)^m,\qquad x,v\in\R^3,
\end{align}
with $ 0\le m = m(k)$ to be defined later. 

Now use $b_m^2(x,v)\pa_{x_i}^k f$ as test function in \eqref{Hkreg.1}.
We obtain
\begin{align*}
\frac{1}{2}\frac{d}{dt}&\int_{\R^6}
|\pa_{x_i}^k f|^2 b_m^2\dd x\dd v + 
\int_{\R^6}\nabla_{v}\pa_{x_i}^k f\cdot A[f]\nabla_{v}\pa_{x_i}^k f b_m^2\dd x\dd v \\
\leq &-\frac{1}{2}\int_{\R^6}|\pa_{x_i}^k f|^2\Delta_v a[f]b_m^2\dd x\dd v
+\int_{\R^6}E_{x_i}^{k}\pa_{x_i}^k f b_m^2\dd x\dd v\\
&-4m\int_{\R^6}b_{m-1}b_m v \cdot \pa_{x_i}^k f (
A[f]\nabla_{v}\pa_{x_i}^k f - \pa _{x_i}^k f \, \nabla_{v}a[f])
\dd x\dd v \\ 
&+4m\int_{\R^6}b_{m-1}b_m x\cdot v |\pa_{x_i}^k f|^2\dd x\dd v.
\end{align*}
From \eqref{Delta.a.bounded} we deduce
\begin{align}\label{Hkreg.2}
\frac{d}{dt}\int_{\R^6} & 
|\pa_{x_i}^k f|^2 b_m^2 \dd x\dd v + 
\int_{\R^6}\nabla_{v}\pa_{x_i}^k f\cdot A[f]\nabla_{v}\pa_{x_i}^k f b_m^2 
\dd x\dd v \\
\leq & C\int_{\R^6}|\pa_{x_i}^k f|^2 b_m^2 \dd x\dd v
+2\int_{\R^6}E_{x_i}^{k}\pa_{x_i}^k f b_m^2 \dd x\dd v\nonumber\\
\nonumber
&-4m\int_{\R^6}b_{m-1}b_m v \cdot \pa_{x_i}^k f 
A[f]\nabla_{v}\pa_{x_i}^k f \dd x\dd v \\
\nonumber
&+4m\int_{\R^6}b_{m-1}b_m v \cdot (\pa_{x_i}^k f)^2 \nabla_{v}a[f]\dd x\dd v\\
\nonumber
&+4m\int_{\R^6}b_{m-1}b_m x\cdot v |\pa_{x_i}^k f|^2\dd x\dd v.
\end{align}
We estimate the last three integrals. In the first one, Young's inequality yields
\begin{align*}
-4m\int_{\R^6}&b_{m-1}b_m v \cdot \pa_{x_i}^k f 
A[f]\nabla_{v}\pa_{x_i}^k f \dd x\dd v \\
\leq & \delta \int_{\R^6}\nabla_{v}\pa_{x_i}^k f\cdot A[f]\nabla_{v}\pa_{x_i}^k f b_m^2\dd x\dd v 
+ C(\delta)m^2\int_{\R^6}(\pa_{x_i}^k f)^2b_{m-1}^2 v\cdot A[f]v\,\dd x\dd v,
\end{align*}
thus \eqref{Da.decay} leads to
\begin{align*}
-4m\int_{\R^6}&b_{m-1}b_m v \cdot \pa_{x_i}^k f 
A[f]\nabla_{v}\pa_{x_i}^k f \dd x\dd v \\
\leq & \delta \int_{\R^6}\nabla_{v}\pa_{x_i}^k f\cdot A[f]\nabla_{v}\pa_{x_i}^k f b_m^2\dd x\dd v 
+ C(\delta)\int_{\R^6}(\pa_{x_i}^k f)^2\bk{v}^{\gamma}b_m^2\dd x\dd v	\\
\leq & \delta \int_{\R^6}\nabla_{v}\pa_{x_i}^k f\cdot A[f]\nabla_{v}\pa_{x_i}^k f b_m^2\dd x\dd v 
+ C(\delta)\int_{\R^6}(\pa_{x_i}^k f)^2 b_m^2\dd x\dd v .
\end{align*}
Bound \eqref{Da.decay} also implies that
\begin{align*}
4m\int_{\R^6}&b_{m-1}b_m 
v \cdot (\pa_{x_i}^k f)^2 \nabla_{v}a[f]\dd x\dd v \\
\leq & C\int_{\R^6}(\pa_{x_i}^k f)^2\bk{v}^{\gamma}b_m^2\dd x\dd v
\leq C \int_{\R^6}(\pa_{x_i}^k f)^2b_m^2\dd x\dd v.
\end{align*}
Moreover
\begin{align*}
4m\int_{\R^6}b_{m-1}b_m x\cdot v |\pa_{x_i}^k f|^2\dd x\dd v\leq 
C\int_{\R^6}b_{m}^2 |\pa_{x_i}^k f|^2\dd x\dd v.
\end{align*}
Using the last three inequalities in \eqref{Hkreg.2} and choosing $\delta>0$ small enough leads to
\begin{align}\label{Hkreg.3}
\frac{d}{dt}\int_{\R^6} & |\pa_{x_i}^k f|^2 b_m^2\dd x\dd v + \frac{1}{2}
\int_{\R^6}\nabla_{v}\pa_{x_i}^k f\cdot A[f]\nabla_{v}\pa_{x_i}^k f b_m^2\dd x\dd v \\
\leq & C_m\int_{\R^6}|\pa_{x_i}^k f|^2 b_m^2\dd x\dd v
+2\int_{\R^6}E_{x_i}^{k}\pa_{x_i}^k f b_m^2\dd x\dd v.\nonumber
\end{align}
We consider now the contribution of the lower order terms:
\begin{align*}
\int_{\R^6} & E_{x_i}^{k}\pa_{x_i}^k f b_m^2\dd x\dd v \\
=& \nonumber
-\sum_{\ell=0}^{k-1}
\binom{k}{\ell}
\int_{\R^6}\nabla_{v}\pa_{x_i}^k f\cdot \left(
\pa_{x_i}^{k-\ell} A[f]\, 
\nabla_{v}\pa_{x_i}^\ell f
- \pa_{x_i}^{k-\ell} \nabla_{v}a[f]\, 
\pa_{x_i}^\ell f
\right)b_m^2\dd x\dd v\\
\nonumber
&-4m\sum_{\ell=0}^{k-1}
\binom{k}{\ell}
\int_{\R^6}\pa_{x_i}^k f b_{m-1}b_m v\cdot \left(
\pa_{x_i}^{k\ell} A[f]\, 
\nabla_{v}\pa_{x_i}^\ell f
- \pa_{x_i}^{k-\ell} \nabla_{v}a[f]\, 
\pa_{x_i}^\ell f
\right)\dd x\dd v.
\end{align*}
Consider first the integral
\begin{align*}
&\left|\int_{\R^6} \nabla_{v}\pa_{x_i}^k f\cdot 
\pa_{x_i}^{k-\ell} A[f]\, 
\nabla_{v}\pa_{x_i}^\ell fb_m^2\dd x\dd v \right|\\
&= \left|\int_{\R^6} A[f]^{1/2}\nabla_{v}\pa_{x_i}^k f\cdot \left(A[f]^{-1/2}\pa_{x_i}^{k-\ell} A[f]\,A[f]^{-1/2}\right) 
A[f]^{1/2}\nabla_{v}\pa_{x_i}^\ell f 
b_m^2\dd x\dd v\right|\\
&\leq \delta\int_{\R^6}|A[f]^{1/2}\nabla_{v}\pa_{x_i}^k f|^2
b_m^2\dd x\dd v\\
&\qquad +C(\delta)\|
A[f]^{-1/2}\pa_{x_i}^{k-\ell} A[f]\,A[f]^{-1/2}\|_{L^\infty(\R^3\times\R^3)}^2
\int_{\R^6}|A[f]^{1/2}\nabla_{v}\pa_{x_i}^\ell f|^2 b_m^2\dd x\dd v\\
&\leq \delta\int_{\R^6}
\nabla_{v}\pa_{x_i}^k f\cdot A[f]\nabla_{v}\pa_{x_i}^k f b_m^2\dd x\dd v\\
&\qquad +C(\delta)\|
A[f]^{-1/2}\pa_{x_i}^{k-\ell} A[f]\,A[f]^{-1/2}\|_{L^\infty(\R^3\times\R^3)}
\int_{\R^6}\nabla_{v}\pa_{x_i}^\ell f\cdot A[f]\nabla_{v}\pa_{x_i}^\ell f b_{m}^2\dd x\dd v.
\end{align*}
{From \eqref{est.DA.exp} we deduce}
\begin{align}\nonumber
&\left|\int_{\R^6} \nabla_{v}\pa_{x_i}^k f\cdot 
\pa_{x_i}^{k-\ell} A[f]\, 
\nabla_{v}\pa_{x_i}^\ell f b_m^2\dd x\dd v\right|\\
\label{est.change.me}
&\leq \delta\int_{\R^6} 
\nabla_{v}\pa_{x_i}^k f\cdot A[f]\nabla_{v}\pa_{x_i}^k f
b_m^2\dd x\dd v
+C(\delta) \int_{\R^6}\nabla_{v}\pa_{x_i}^\ell f\cdot A[f]\nabla_{v}\pa_{x_i}^\ell f b_{m}^2\dd x\dd v.
\end{align}
We now turn our attention to the next error contribution. From \eqref{Hkreg.kappa}, \eqref{Da.decay} we get
\begin{align*}
&\left|\int_{\R^6}\nabla_{v}\pa_{x_i}^k f\cdot 
\pa_{x_i}^{k-\ell} \nabla_{v}a[f]\, 
\pa_{x_i}^\ell f b_m^2\dd x\dd v\right|
\leq C\int_{\R^6}|\nabla_{v}\pa_{x_i}^k f| 
|\pa_{x_i}^\ell f| 
\bk{v}^{\gamma+1}\bk{x}^{k-\ell-\lambda}b_m^2\dd x\dd v.
\end{align*}
Young's inequality and \eqref{aA.lb} then yield
\begin{align*}
&\left|\int_{\R^6}\nabla_{v}\pa_{x_i}^k f\cdot 
\pa_{x_i}^{k-\ell} \nabla_{v}a[f]\, 
\pa_{x_i}^\ell f b_m^2\dd x\dd v\right|\\
&\leq 
\delta\int_{\R^6}|\nabla_{v}\pa_{x_i}^k f|^2\bk{v}^{\gamma}\bk{x}^{-\lambda}b_m^2\dd x\dd v
+C(\delta)\int_{\R^6}|\pa_{x_i}^\ell f|^2 
\bk{v}^{\gamma+2}\bk{x}^{2(k-\ell)-\lambda}b_m^2\dd x\dd v\\
&\leq 
\delta\int_{\R^6}\nabla_{v}\pa_{x_i}^k f\cdot A[f]\nabla_{v}\pa_{x_i}^k f b_m^2\dd x\dd v
+C(\delta)\int_{\R^6}|\pa_{x_i}^\ell f|^2 b_{m+k-\ell}^2\dd x\dd v.
\end{align*}
Let us now consider
\begin{align*}
&\left|\int_{\R^6} \pa_{x_i}^k f b_{m-1}b_m v\cdot 
\pa_{x_i}^{k-\ell} A[f]\, \nabla_{v}\pa_{x_i}^\ell f
\dd x\dd v\right|\\
&\leq \int_{\R^6}|\pa_{x_i}^k f| 
|\nabla_{v}\pa_{x_i}^\ell f| 
b_m^2\bk{v}^{\gamma+1}\bk{x}^{-\lambda}\dd x\dd v\\
&\leq \delta\int_{\R^6}|\nabla_{v}\pa_{x_i}^\ell f|^2 
\bk{v}^{\gamma}
\bk{x}^{-\lambda}b_m^2\dd x\dd v
+C(\delta)\int_{\R^6}|\pa_{x_i}^k f|^2 \bk{v}^{\gamma+2}\bk{x}^{-\lambda}b_m^2\dd x\dd v,
\end{align*}
where we used once again \eqref{Da.decay}. From \eqref{aA.lb} we obtain
\begin{align*}
&\left|\int_{\R^6} \pa_{x_i}^k f b_{m-1}b_m v\cdot 
\pa_{x_i}^{k-\ell} A[f]\, \nabla_{v}\pa_{x_i}^\ell f
\dd x\dd v\right|\\
&\leq \delta\int_{\R^6}\nabla_{v}\pa_{x_i}^\ell f\cdot
A[f]\nabla_{v}\pa_{x_i}^\ell f b_m^2\dd x\dd v
+C(\delta)\int_{\R^6}|\pa_{x_i}^k f|^2 b_{m}\dd x\dd v.
\end{align*}
Finally we need to deal with
\begin{align*}
&\left|
\int_{\R^6}\pa_{x_i}^k f b_{m-1}b_m v\cdot 
\pa_{x_i}^{k-\ell} \nabla_{v}a[f]\, 
\pa_{x_i}^\ell f \dd x\dd v\right|\\
&\leq \int_{\R^6}|\pa_{x_i}^k f| |\pa_{x_i}^\ell f| 
\bk{v}^{\gamma}
\bk{x}^{-\lambda}b_m^2\dd x\dd v\\
&\leq \int_{\R^6}|\pa_{x_i}^k f|^2 b_m^2\dd x\dd v
+\int_{\R^6}|\pa_{x_i}^\ell f|^2 b_{m}^2\dd x\dd v.
\end{align*}
Using the previous estimates in \eqref{Hkreg.3} and choosing $\delta>0$ small enough lead to
\begin{align*}%\label{Hkreg.4}
\frac{d}{dt} & \int_{\R^6} |\pa_{x_i}^k f|^2 b_m^2\dd x\dd v 
%+ \frac{1}{4}
+\int_{\R^6}\nabla_{v}\pa_{x_i}^k f\cdot A[f]\nabla_{v}\pa_{x_i}^k f b_m^2\dd x\dd v \\ \nonumber
\leq & C_m\int_{\R^6}|\pa_{x_i}^k f|^2 b_m^2\dd x\dd v\\ \nonumber
&+ C_{m,k}
\sum_{\ell=0}^{k-1}
\left[
\int_{\R^6}\nabla_{v}\pa_{x_i}^\ell f\cdot A[f]\nabla_{v}\pa_{x_i}^\ell f b_{m+k-\ell}^2\dd x\dd v
+ \int_{\R^6}|\pa_{x_i}^\ell f|^2 b_{m+k-\ell}^2\dd x\dd v
\right].
\end{align*}
If $k = 0$ the above inequality reduces to 
$$
\frac{d}{dt}  \int_{\R^6} f^2 b_m^2\dd x\dd v  \le  C_m\int_{\R^6} f^2 b_m^2\dd x\dd v,
$$
and Gronwall's Lemma yields 
\begin{align}\label{beta=0}
\sup_{t\in[0,T]} \int_{\R^6} f^2 b_m^2\dd x\dd v  \le  C(T),
\end{align}
 for any $m\le s_0 + s$, using (\ref{Hkreg.ic.2}). 

At this point an induction argument on $k$ easily yields
\begin{align}\label{Hkreg.result.0}
\sup_{t\in [0,T]}\int_{\R^6}|\pa_{x_i}^k f(t)|^2 b_m^2\dd x\dd v
+\int_0^T\int_{\R^6}\nabla_{v}\pa_{x_i}^k f\cdot A[f]\nabla_{v}\pa_{x_i}^k f b_m^2\dd x\dd v\dd t < \infty,
\end{align}
for every choice of $ 0\le k \le s$ and  $m(k)\in \N$ such that
$$m(k)\leq s+s_0 - k, \qquad i=1,2,3.$$

Summarizing we have shown that if the initial data is such that 

\begin{align*}%\label{Hkreg.goal}
 \left\|b_{s_0+s - |\beta|}  D^{\beta}_x f_{in} \right\|_{L^2(\R^6)} <+\infty, \\
\left\| \bk{v}^{\gamma/2}\bk{x}^{-\lambda/2}
b_{s_0+s  - |\beta|}  D^{\beta}_x \nabla_{v}f_{in}\right\|_{L^2(\R^3\times\R^3)} <+\infty, 
\end{align*}
for any $s_0\ge 0$, then  \eqref{Hkreg.goal} holds with $|\alpha|=0$,
$\beta = k e^{(i)}$, $i=1,2,3$, $|\beta| = k$:
\begin{align*}%\label{Hkreg.goal}
\sup_{t\in [0,T]} \left\|b_{s_0+s - |\beta|}  D^{\beta}_x f \right\|_{L^2(\R^6)} <+\infty, \\
\left\| \bk{v}^{\gamma/2}\bk{x}^{-\lambda/2}
b_{s_0+s  - |\beta|}  D^{\beta}_x \nabla_{v}f\right\|_{L^2(0,T; L^2(\R^3\times\R^3))} <+\infty. 
\end{align*}

Now we want to prove \eqref{Hkreg.goal} with $\beta=0$,
$\alpha=k e^{(i)}$, $i=1,2,3$. Apply $\pa_{v_i}^k$ to \eqref{1}. We get
\begin{align}\label{Hkreg.1b}
\pa_t \pa_{v_i}^k f + v\cdot\nabla_{x}\pa_{v_i}^k f = 
\Div_v(
A[f]\nabla_{v}\pa_{v_i}^k f - \pa_{v_i}^k f \, \nabla_{v}a[f])
+ E_{v_i}^{k},
\end{align}
with the lower order terms given by
\begin{align*}
E_{v_i}^{k} = 
\sum_{\ell=0}^{k-1}
\binom{k}{\ell}\Div_v\left(
\pa_{v_i}^{k-\ell} A[f]\, 
\pa_{v_i}^\ell\nabla_{v} f
- \pa_{v_i}^{k-\ell} \nabla_{v}a[f]\, 
\pa_{v_i}^\ell f
\right) 
- k\pa_{x_i}\pa_{v_i}^{k-1}f.
\end{align*}
Employing $b_m^2\pa_{v_i}^k f$ as a test function in \eqref{Hkreg.1b} (with $m$ dependent on $k$, to be chosen later)  yields
\begin{align}\nonumber
\frac{d}{dt}\frac{1}{2}&\int_{\R^6}|\pa_{v_i}^k f|^2
b_m^2\dd x\dd v \\
=& -\int_{\R^6}\nabla_v\pa_{v_i}^k f\cdot A[f]\nabla_v\pa_{v_i}^k f b_m^2\dd x\dd v
+\int_{\R^6}\nabla_v\pa_{v_i}^k f\cdot \pa_{v_i}^k f\nabla_v a[f] b_m^2\dd x\dd v\nonumber\\
&-4m\int_{\R^6}\pa_{v_i}^k f v\cdot A[f]\nabla_v\pa_{v_i}^k f b_m b_{m-1}\dd x\dd v
+4m\int_{\R^6}\pa_{v_i}^k f v\cdot \pa_{v_i}^k f\nabla_v a[f] b_m b_{m-1}\dd x\dd v\nonumber\\ 
&+\int_{\R^6}E_{v_i}^k\pa_{v_i}^k f b_m^2\dd x\dd v.
\label{GAMM.1}
\end{align}
Once again we proceed by induction on $k$. The case $k=0$ has already been treated with (see \eqref{Hkreg.result.0}).
The terms in \eqref{GAMM.1} can be estimated in the same way as the ones for the previous case $\alpha=0$, $\beta=k e^{(i)}$; we only show how to bound the highest order term, which is
\begin{align*}
\int_{\R^6}\nabla_v\pa_{v_i}^k f\cdot f\nabla_v \pa_{v_i}^k a[f] b_m^2\dd x\dd v.
\end{align*}
The above integral is easier to estimate when $\gamma\geq -2$, hence we only consider the case $\gamma\in (-3,-2)$.
We estimate
\begin{align*}
&\left| \int_{\R^6}\nabla_v\pa_{v_i}^k f\cdot f\nabla_v \pa_{v_i}^k a[f] b_m^2\dd x\dd v
\right| \\
&\leq \|b_m\bk{v}^{\gamma/2}\bk{x}^{-\lambda/2}\nabla_v\pa_{v_i}^k f\|_{L^2(\R^6)}
\|b_m\bk{v} f\|_{L^2(\R^6)}
\|\bk{v}^{-\gamma/2-1}\bk{x}^{\lambda/2}\nabla_v \pa_{v_i}^k a[f]\|_{L^\infty(\R^6)}\\
&\leq C\|b_m\bk{v}^{\gamma/2}\bk{x}^{-\lambda/2}\nabla_v\pa_{v_i}^k f\|_{L^2(\R^6)}
\|\bk{v}^{-\gamma/2-1}\bk{x}^{\lambda/2}\nabla_v \pa_{v_i}^k a[f]\|_{L^\infty(\R^6)},
\end{align*}
from the bounds on the moments of $f$. In the last inequality we need to estimate the $L^\infty$-norm. Recall that 
$$
a[f] = 2 H[f] \ast_v|v|^{\gamma+2}. 
$$
We use \eqref{DH.est.1} with $p=2$ and $r = -\gamma-2$ and $\alpha = k+1$ to obtain the following bound:
\begin{align*}
&\|\bk{v}^{-\frac{1}{2}(\gamma+2)}\bk{x}^{\frac{\lambda}{2}}\nabla_v \pa_{v_i}^k a[f]\|_{L^\infty(\R^6)}\\
&\leq 
\left\| 
\bk{v}^{\frac{1}{2}(\gamma+2)}\bk{x}^{-\frac{\lambda}{2}}\left(
\epsilon\|\bk{y}^{\lambda+{\frac{3}{2}}^+}\nabla_v\pa_{v_i}^k f\|_{L^2(\R^6)} + C(\epsilon)\int_{\R^6}\bk{w}^{ \frac{3(k-\gamma-1)}{7+2\gamma} } {\bk{y}^{\lambda}}
f(y,w)\dd y\dd w
\right)
\right\|_{L^\infty(\R^6)}
\\ 
& \le  \epsilon\|\bk{y}^{\lambda+{\frac{3}{2}}^+}\nabla_v\pa_{v_i}^k f\|_{L^2(\R^6)} + C(\epsilon)\int_{\R^6}\bk{w}^{ \frac{3(k-\gamma-1)}{7+2\gamma} }{\bk{y}^{\lambda}}f(y,w)\dd y\dd w
 \\
 & =  \epsilon\|\bk{y}^{\lambda+{\frac{3}{2}}^+ + \frac{\lambda}{2}}\bk{v}^{-\frac{\gamma}{2}}\bk{y}^{- \frac{\lambda}{2}}\bk{v}^{\frac{\gamma}{2}}\nabla_v\pa_{v_i}^k f\|_{L^2(\R^6)} + C(\epsilon)\int_{\R^6}\bk{w}^{ \frac{3(k-\gamma-1)}{7+2\gamma} }{\bk{y}^{\lambda}}f(y,w)\dd y\dd w
 \\
&\leq C\left(
\epsilon\|b_m\bk{x}^{-\lambda/2}\bk{v}^{\gamma/2}\nabla_v\pa_{v_i}^k f\|_{L^2(\R^6)} + C(\epsilon)c_{\mm}
\right),
\end{align*}
{ 
provided that 
$$m >  \frac{3\lambda + 3 - \gamma}{2}, \quad \mm \geq \frac{3(k-\gamma-1)}{7+2\gamma} {+\lambda},$$ 
using the moments bound for $f$.

}

Hence we deduce via Young's inequality
\begin{align*}
&\left| \int_{\R^6}\nabla_v\pa_{v_i}^k f\cdot f\nabla \pa_{v_i}^k a[f] b_m^2\dd x\dd v
\right| 
\leq \epsilon \|b_m\bk{x}^{-\lambda/2}\bk{v}^{\gamma/2}\nabla_v\pa_{v_i}^k f\|_{L^2(\R^6)}^2 + C(\epsilon).
\end{align*}
The term on the right-hand side of the above inequality is controlled by 
$$-\int_{\R^6}\nabla_v\pa_{v_i}^k f\cdot A[f]\nabla_v\pa_{v_i}^k f b_m^2\dd x\dd v \lesssim  -\int_{\R^6}\bk{x}^{-\lambda}\bk{v}^{\gamma} |\nabla_v\pa_{v_i}^k f|^2 b_m^2\dd x\dd v   ,$$
thanks to \eqref{aA.lb} provided that $\epsilon>0$ is small enough. Furthermore, the term
\begin{align*}
\int_{\R^6}\nabla_v\pa_{v_i}^k f\cdot 
\pa_{v_i}^{\ell}f \nabla \pa_{v_i}^{k-\ell} a[f] b_m^2\dd x\dd v
\end{align*}
with $0<\ell<k$ can be estimated as follows
\begin{align*}
&\left| \int_{\R^6}\nabla_v\pa_{v_i}^k f\cdot 
\pa_{v_i}^{\ell}f \nabla \pa_{v_i}^{k-\ell} a[f] 
b_m^2\dd x\dd v
\right| \\
&\leq \|b_m\bk{v}^{\gamma/2}\bk{x}^{-\lambda/2}\nabla_v\pa_{v_i}^k f\|_{L^2(\R^6)}
\|b_m\bk{v} \pa_{v_i}^{\ell} f\|_{L^2(\R^6)}
\|\bk{v}^{-\gamma/2-1}\bk{x}^{\lambda/2}\nabla_v \pa_{v_i}^{k-\ell} a[f]\|_{L^\infty(\R^6)}\\
&\leq \|b_m\bk{v}^{\gamma/2}\bk{x}^{-\lambda/2}\nabla_v\pa_{v_i}^k f\|_{L^2(\R^6)}
\|b_{m+1} \pa_{v_i}^{\ell} f\|_{L^2(\R^6)}
\|\bk{v}^{-\gamma/2-1}\bk{x}^{\lambda/2}\nabla_v \pa_{v_i}^{k-\ell} a[f]\|_{L^\infty(\R^6)}.
\end{align*}
Since $\ell<k$ the term $\|b_{m+1} \pa_{v_i}^{\ell} f\|_{L^2(\R^6)}$ is bounded in $L^\infty(0,T)$ thanks to the inductive assumption. The other factors can be controlled just like in the case $\ell=0$.

The other terms in \eqref{GAMM.1} are lower order and can be estimated in similar ways.

We only mention that the transport term produces an error term in this case, namely
\begin{align*}
\int_{\R^6}\pa_{v_i}^k f\,\pa_{x_i}\pa_{v_i}^{k-1}f\, b_m^2\dd x\dd v.
\end{align*}
However, this term is easily estimated via Lemma \ref{lem.trans}.

Summarizing, Gronwall's Lemma yields the bounds
\begin{align*}
\sup_{[0,T]}\sum_{i=1}^3\|b_m\pa_{v_i}^k f\|_{L^2(\R^6)}^2 + \sum_{i=1}^3
\int_0^T \|b_m\bk{v}^{\gamma/2}\bk{x}^{-\lambda/2}\nabla_v\pa_{v_i}^k f\|_{L^2(\R^6)}^2\dd t\leq C(T)
\end{align*}
so that \eqref{Hkreg.goal} holds. 

The condition 
$m >  \frac{3\lambda + 3 - \gamma}{2}$ should hold for any derivative order $k$. This implies that the value of $s_0$ has to be 
$$
s_0 > \frac{3\lambda + 3 - \gamma}{2}.  
$$

We summarize the result in the following theorem:

\begin{theo}[$H^s$ regularity]\label{thm.reg}

Consider initial data $f_{in} \in H^s(\mathbb{R}^6)$ that satisfies the assumptions in Section \ref{sec:initial data}. Let $s_0 = \frac{3\lambda + 3 -\gamma}{2}$ .

Consider the spatial kernel $\kappa \in C^s(\R^3)$ satisfying  
\begin{align}\label{kappa_reg}
	\forall\beta\in\N_0^d,~~|\beta|\leq s,~~\exists C_\beta>0: ~~ 
	|D_x^\beta\kappa|\leq C_\beta  \kappa ~~\mbox{in }\R^3 .
\end{align}

Then any weak solution $f\ge 0$ to the fuzzy Landau equation \eqref{1} with $0\ge \gamma > -\sqrt{8}$ and $f_{in}$ and $\kappa$ as above  has the following regularity: 
\begin{align}
&(1+|x|^2+|v|^2)^{s_0 + s -|\alpha| - |\beta|}  D_v^{\alpha} D^{\beta}_x f\in L^\infty(0,T; L^2(\R^3\times\R^3)),\label{dabf} \\
&\bk{v}^{\frac{\gamma}{2} }\bk{x}^{-\lambda/2}
(1+|x|^2+|v|^2)^{s_0 + s -|\alpha| - |\beta|} D_v^{\alpha} D^{\beta}_x \nabla_{v}f\in L^2(0,T; L^2(\R^3\times\R^3)),\label{dabf-v}
\end{align}
for every choice of $\alpha, \beta\in \N_0^d$ such that
$|\alpha|+|\beta| \leq s$.

Furthermore, if $s\geq 2$ then
\begin{align}\label{dabf-t}
(1+|x|^2+|v|^2)^{s_0 + s - 2 -|\alpha| - |\beta|}  D_v^{\alpha} D^{\beta}_x \pa_ t f\in L^\infty(0,T; L^2(\R^3\times\R^3))
\end{align}
for every choice of $\alpha, \beta\in \N_0^d$ such that
$|\alpha|+|\beta|+2 \leq s$. 

All the bounds on the above norms only depends on the initial data. 

\end{theo}

We note that to obtain strong solutions to (\ref{1}) it is sufficient to choose {$s=2$. In this case 
Theorem \ref{thm.reg} guarantees that $\nabla_x f\in L^\infty(0,T; L^2(\R^3\times\R^3))$ and both 
$\pa_t f$ and $ \nabla_v^2 f$ belong to $ L^2(0,T; L^2(\R^3\times\R^3))$
}. 

To obtain classical solutions one can consider 
$s= 7$. { For this choice Theorem \ref{thm.reg} implies that  $f\in L^2(0,T; H^{s}(\R^6))$ and 
$\pa_t f\in L^2(0,T; H^{s-2}(\R^6))$. Standard Sobolev embeddings then yield $f\in C^0([0,T]; H^{s-1}(\R^6))$. Since  $H^{s-1}(\R^6)\hookrightarrow C^2(\R^6)$ for any $s > 6$, classical regularity is achieved.  }

\section{Global well posedness for 
$\gamma \in(-\sqrt{8},0]$}\label{sec:gwp}

The main results of this section are summarized in the following theorem: 

\begin{theo}\label{Thm1}
Let $\gamma \in(-\sqrt{8},0]$. 
Consider initial data $f_{in} \in H^s{\cap L^1_{\mm} \cap L^\infty}(\mathbb{R}^6)$, $s\ge 2$, as in Assumption \ref{sec:initial data}. Recall $s_0 = \frac{3\lambda + 3 -\gamma}{2}$. Let $\kappa $  satisfy (\ref{kappa}) 
%and 
%\begin{align*}
%	\forall\beta\in\N_0^d,~~|\beta|\leq s,~~\exists C_\beta>0: ~~ 
%	|D_x^\beta\kappa|\leq C_\beta \kappa ~~\mbox{in }\R^3 .
%\end{align*}
 
If $\gamma \le -2$ assume moreover that 
$$
\|\nabla \sqrt{f_{in}} \|_{L^2(\mathbb{R}^6)} \le C. 
$$
There exists a {{unique}} strong solution to (\ref{1}) with initial data $f_{in}$.  In particular 
\begin{align*}
&(1+|x|^2+|v|^2)^{s_0 + s -|\alpha| - |\beta|}  D_v^{\alpha} D^{\beta}_x f\in L^\infty(0,T; L^2(\R^3\times\R^3)),\\
&\bk{v}^{\frac{\gamma}{2} }\bk{x}^{-\lambda/2}
(1+|x|^2+|v|^2)^{s_0 + s -|\alpha| - |\beta|} D_v^{\alpha} D^{\beta}_x \nabla_{v}f\in L^2(0,T; L^2(\R^3\times\R^3)),
\end{align*}
for every multi-index $\alpha, \beta$ such that
$|\alpha|+|\beta| \leq s$ and 
\begin{align*}
(1+|x|^2+|v|^2)^{s_0 + s - 2 -|\alpha| - |\beta|}  D_v^{\alpha} D^{\beta}_x \pa_ t f\in L^\infty(0,T; L^2(\R^3\times\R^3))
\end{align*}
for every multi-index  $\alpha, \beta$ such that
$|\alpha|+|\beta| +2 \leq s$. 
\end{theo}

Theorem \ref{Thm1} can be straightforwardly  extended to $\kappa(x) = e^{-\langle x\rangle}$. Some more technical  (but still straightforward) changes are required to extend Theorem  \ref{Thm1} to $\kappa(x) = e^{-\langle x\rangle^2}$.

Our strategy for proving Theorem \ref{Thm1} proceeds in two steps: we first establish the existence of smooth solutions for a suitably regularized problem, and subsequently remove the regularizations by leveraging the uniform a priori estimates derived in Section \ref{sec:reg}.

\subsection{Existence of an approximated problem}

\subsubsection{Approximated equations.}\label{ss.approx}

We first solve the  equation: 
\begin{align}\label{1.app}
f_t + e^{-\delta|v|^2}v\cdot \nabla_x f - \eps(\Delta_x + \Delta_v)f = Q_\delta(f,f),\qquad x, v\in\R^3,~~
\end{align}
where the approximated collision operator $Q_\delta (f,g)$ is given by
\begin{align}\label{Q.app}
Q_\delta (f,g) 
&= \Div_v\int_{\R^3}\int_{\R^3}N_\delta(v-v^*)\kappa(x-x^*)(f^*\nabla_v g - g\nabla_{v^*} f^*)\dd x^*\dd v^* \\
&= \Div_v(A_\delta[f] \nabla_v g - g~b_\delta[f]) ,\label{eq:Qdel}
\end{align}
with 
\begin{comment}
the regularized spatial kernel $\kappa_\delta$ defined as 
\begin{align}\label{kappa.app}
\kappa_\delta(y) := \kappa(y)e^{-\delta \bk{y}},
\end{align}
and 
\end{comment}
the regularized Landau kernel as
{
\begin{equation}
\label{N.app}
N_\delta(w)= (\delta + |w|^2)^{\frac{\gamma}{2}}
(1+\delta |w|^2)^{-K}\; |w|^2\left(
\eye - \frac{w\otimes w}{|w|^2}\right), 
\quad \gamma \le 0,\quad K = \frac{1}{2}(\gamma+2)_+ . 
\end{equation}
}
Note that 
\begin{align}
    \textrm{div}_w N_\delta(w) & = -2 (\delta + |w|^2)^{\frac{\gamma}{2}}
    {(1+\delta |w|^2)^{-K}}
    w,\\
    {\textrm{div}_w}(\textrm{div}_w N_\delta(w)) & = -2 (\delta + |w|^2)^{\frac{\gamma}{2}}
    {(1+\delta |w|^2)^{-K}}
    \left[ 3 + \gamma \;  \frac{|w|^2}{\delta + |w|^2}
    {-2K\frac{\delta |w|^2}{1+\delta|w|^2}}
    \right].
\end{align}

Recall that 
\begin{align*}
A_\delta[f]& :=\int_{\R^6}N_\delta(v-v^*)\kappa(x-x^*)f(x^*,v^*) \dd{x}^* \dd{v}^*,\\
b_\delta[f] &:= \int_{\R^6}\Div_v N_\delta(v-v^*)\kappa(x-x^*)f(x^*,v^*) \dd{x}^* \dd{v}^*.
\end{align*}
Also we set 
\begin{align}\label{f0.app}
%f_0 \equiv \min(f_{in},\eps^{-1}e^{-|v|^2}).
{f_0(x,v) \equiv f_{in}(x,v)\Psi(\eps v),}
\end{align}
{with $\Psi\in C^\infty_c(\R^3)$: 
$\Psi(v)\equiv 1$ for $|v|\leq 1$, $0\leq\Psi(v)\leq 1$ for $1<|v|<2$, $\Psi(v)\equiv 0$ for $|v|\geq 2$. }

{
Notice that 
\begin{comment}
$\kappa_\delta$ is a Schwartz functions; in particular $\kappa_\delta$ and its derivatives of any order are bounded and integrable in $\R^3$. The function 
\end{comment}
$N_\delta\in C^\infty\cap L^\infty(\R^3)$ 
and all its derivatives of every order are bounded in $\R^3$.}

\subsection*{Solving equation \eqref{1.app}.}
To solve \eqref{1.app} with initial data (\ref{f0.app}) we iterate linear problems and show the existence of a fixed point using Leray-Schauder fixed point theorem. 

We define the mapping 
\begin{align*}
\mathcal{T} &: (\tilde{f},\sigma)\in X\times [0,1]\mapsto f\in X,\\
X&:=\{ u \in C([0,T],L^1(\R^6)),\; u\ge 0 , \; \sup_t \| u\|_{L^1_{x,v}} \le \|f_{0}\|_{L^1_{x,v}} \},
\end{align*}
where $f$ solves  the linear problem 
\begin{align}\label{1.fp}
f_t + e^{-\delta|v|^2}v\cdot \nabla_x f - \eps(\Delta_x + \Delta_v)f = \sigma Q_\delta(\tilde{f},f),\qquad f(\cdot, 0) = f_0 .
\end{align}
Equation (\ref{1.fp}) is a linear parabolic equation with bounded and smooth coefficients. Then it has a unique smooth solution. This solution is nonnegative: to see that take  $f^-$ as test function for (\ref{1.fp}), with the convention that $u = u^+ + u^-$. Using \eqref{eq:Qdel}, Lemma \ref{lem.aA.lb}, and the  bound 
\begin{align}\label{unif_b_est}
|b_\delta[\tilde{f}]|\leq \|\Div_v (N_\delta)\kappa\|_{L^\infty(\R^6)}\|\tilde{f}\|_{L^1(\R^6)}
\leq C(\delta)\|\tilde{f}\|_{L^1(\R^6)}\leq C(\delta)\|f_0\|_{L^1(\R^6)},
\end{align}
we get
\begin{align*}
\frac{1}{2}\int_{\R^6}(f^-)^2 \dd{x} \dd{v} & + \eps \int_0^T\int_{\R^6}( |\nabla_x f^-|^2 + |\nabla_v f^-|^2 )\dd{x} \dd{v}\\ \nonumber
&=  \int_0^T\int_{\R^6}f^- Q_\delta(\tilde{f},f)\dd{x} \dd{v} \\
& \le \int_0^T\int_{\R^6} f \;b_\delta[\tilde{f}] \cdot \nabla_v f^- \dd{x}\dd{v}\\
& \le C(\delta) \| f_0\|_{ L^1_{x,v}} \int_0^T \int_{\R^6} |f^-| | \nabla_v f^-| \dd{x}\dd{v} \\
& \le  \frac{\varepsilon}{2}  \int_0^T \int_{\R^6}| \nabla_v f^-|^2  \dd{x}\dd{v} \; + \; C^2_{\delta,\varepsilon} \| f_0\|^2_{ L^1_{x,v}} \int_0^T\int_{\R^6} (f^-)^2\; \dd{x}\dd{v}.
\end{align*}
 Gronwall's inequality yields 
$$
\sup_t\int_{\R^6}(f^-)^2 \dd{x} \dd{v} \le 0,
$$
which implies that 
\begin{align}
  f \ge 0. \label{posit_f_k}  
\end{align}

Next,  we look for estimates for $f$ that are independent of $\sigma$. These estimates will be used to show that $\mathcal{T}$ is a well defined and compact map.

\begin{itemize}

\item The solution $f$ has  $\sigma-$uniform bound in 
$$L^2(0,T,H^1(\R^6))\cap L^\infty(0,T,L^2(\R^6)).$$ We test \eqref{1.fp} against $f$ and get
\begin{align}\label{2.fp}
\frac{1}{2}\int_{\R^6}f^2 \dd{x} \dd{v} &- \frac{1}{2}\int_{\R^6}|f_{0}|^2 \dd{x} \dd{v} + \eps \int_0^T\int_{\R^6}( |\nabla_x f|^2 + |\nabla_v f|^2 )\dd{x} \dd{v}dt \nonumber\\ 
&=  \sigma \int_0^T\int_{\R^6}f Q_\delta(\tilde{f},f)\dd{x} \dd{v}dt.
\end{align}
It suffices to bound the collision term:
\begin{align*}
\int_{\R^6}f Q_\delta(\tilde{f},f)\dd{x} \dd{v} = 
-\int_{\R^6}\nabla_{v}f\cdot A_\delta[\tilde{f}]\nabla_{v}f \dd{x} \dd{v}
+\int_{\R^6}f \nabla_{v}f\cdot b_\delta[\tilde{f}]\dd{x} \dd{v}.
\end{align*}
The matrix $A_\delta[\tilde{f}]$ is positive semidefinite, while the vector $b_\delta[\tilde{f}]$, defined as 
\begin{align*}
b_\delta[\tilde{f}] = \int_{\R^6}\Div N_\delta(v-v^*)\kappa(x-x^*)\tilde{f}(x^*,v^*) \dd{x}^* \dd{v}^*
\end{align*}
is bounded, thanks to (\ref{unif_b_est}).
We obtain
\begin{align*}
\int_{\R^6}f Q_\delta(\tilde{f},f)\dd{x} \dd{v}\leq & \; C(\delta)\|f_0\|_{L^1(\R^6)}\|{f}\|_{L^2(\R^6)}\|\nabla_v {f}\|_{L^2(\R^6)}\\
 \le & \;\frac{ C^2(\delta)\|f_0\|^2_{L^1(\R^6)}}{\varepsilon} \|{f}\|^2_{L^2(\R^6)} + \frac{\varepsilon}{2} \|\nabla_v {f}\|^2_{L^2(\R^6)}.
\end{align*}
Substituting the above bound in  \eqref{2.fp} we get 
\begin{align*}
\frac{1}{2} \int_{\R^6}f^2 \; \dd{x} \dd{v}
+& \frac{\varepsilon}{2}\int_0^T \int_{\R^6}( |\nabla_x f|^2 + |\nabla_v f|^2 )\dd{x} \dd{v}dt
\\
 &\leq \frac{1}{2}\int_{\R^6}f_0^2 \; \dd{x} \dd{v} + \;\frac{ C^2(\delta)\|f_0\|^2_{L^1(\R^6)}}{\varepsilon} \int_0^T \|{f}\|^2_{L^2(\R^6)}\;dt.
\end{align*}
Gronwall's inequality yields the desired $\sigma-$ uniform bound for $f$: 
\begin{align}\label{stima_eps}
    \int_{\R^6}f^2 \dd{x} \dd{v}
\; + \; & {\varepsilon}\int_0^T \int_{\R^6}( |\nabla_x f|^2 + |\nabla_v f|^2 )\; \dd{x} \dd{v}dt \nonumber\\
\leq & \; C(T,\varepsilon,\delta, \|f_0\|_{L^1})\int_{\R^6}f_0^2\;  \dd{x} \dd{v} .
\end{align}

\item Next, we show the $L^1$ bound for $f$:
\begin{align}\label{app.mass}
	\int_{\R^6}f \; \dd{x} \dd{v} \le \int_{\R^6}f_0 \; \dd{x} \dd{v}.
\end{align}
For $R\geq 1$ set $\eta_R(x,v) = \eta_1(x/R,v/R)$, $\eta_1(x,v) = (1+|x|^2+|v|^2)^{-\alpha}$ for some { $\alpha>3$} to be determined later. We point out that $\eta_R\in H^1\cap L^1(\R^3\times\R^3)$ and therefore it is an admissible test function. Furthermore it holds
\begin{align}\label{etaR}
	|\nabla\eta_R|\leq C R^{-1}\eta_R,\qquad |\Delta\eta_R|\leq C R^{-1}\eta_R,\qquad 
	\nabla\in\{\nabla_x,\nabla_v\}, ~~ \Delta\in\{\Delta_x,\Delta_v\}.
\end{align}
We multiply \eqref{1.fp} by $\eta_R$ and integrate in $\mathbb{R}^6$. We get
\begin{align}\label{app.mass.R}
\int_{\R^6}\eta_R f \dd{x} \dd{v} =&  \int_{\R^6}\eta_R f_0 \dd{x} \dd{v} 
+ 
\int_0^T\int_{\R^6} f e^{-\delta |v|^2}v\cdot \nabla_x\eta_R \; \dd{x} \dd{v} \nonumber \\
& +  \eps\int_0^T\int_{\R^6}f (\Delta_x+\Delta_v)\eta_R \; \dd{x} \dd{v}
+\sigma\int_0^T \int_{\R^6}\eta_R Q_\delta(\tilde{f},f)\; \dd{x} \dd{v}.
\end{align}
Using \eqref{etaR} we have the following bounds:
\begin{align*}
\left| \int_{\R^6} f e^{-\delta |v|^2}v\cdot \nabla_x\eta_R \dd{x} \dd{v}\right|  & \leq C
R^{-1}\int_{\R^6} f \eta_R \dd{x} \dd{v},\\
\eps \int_{\R^6}f (\Delta_x+\Delta_v)\eta_R \dd{x} \dd{v} & \leq 
C R^{-1}\int_{\R^6} f \eta_R \dd{x} \dd{v}.
\end{align*}
Moreover
\begin{align*}
 \int_{\R^6}\eta_R Q_\delta(\tilde{f},f)\dd{x} \dd{v}  & = 2 \int_{\R^6} f \; b_\delta[\tilde{f}] \cdot \nabla_v\eta_R \;dxdv + \int_{\R^6} f \; \textrm{Tr} ( A_\delta[\tilde{f}] \; \textrm{Hess}_v \eta_R ) \;dxdv  \\
&\le 2 \frac{\|b_\delta[\tilde{f}]\|_{L^\infty_{x,v}} + \|A_\delta[\tilde{f}]\|_{L^\infty_{x,v}}}{R} \int_{\R^6} f \eta_R \;dxdv \\
& \le \frac{C( \delta)\|f_0\|_{L^1_{x,v}}}{R}\int_{\R^6} f \eta_R \;dxdv , 
\end{align*}
using (\ref{unif_b_est}) to estimate $\|b_\delta[\tilde{f}]\|_{L^\infty_{x,v}}$ (and similarly $\|A_\delta[\tilde{f}]\|_{L^\infty_{x,v}}$). 
Combining all the estimates above we obtain 
\begin{align*}
\int_{\R^6}\eta_R f \dd{x} \dd{v} \le & \int_{\R^6}\eta_R f_{0} \dd{x} \dd{v}   + \;  \frac{C( \delta)\|f_0\|_{L^1_{x,v}}}{R} \int_0^T \int_{\R^6}\eta_R f \dd{x} \dd{v}\\
\le & \int_{\R^6} f_{0} \dd{x} \dd{v}+ \;  \frac{C( \delta)\|f_0\|_{L^1_{x,v}}}{R} \int_0^T \int_{\R^6}\eta_R f \dd{x} \dd{v}. 
\end{align*}
First Gronwall and then dominated convergence theorem yields (\ref{app.mass}). In particular $\mathcal{T}$ is well-defined as a map from $X$ to $X$.

\item Next we prove $\sigma-$uniform bound for 
$$(1+|v|^2)^{1/2}f \quad \textrm{and} \quad (1+|x|^2)^{1/2}f $$ in $L^\infty(0,T,L^1_{x,v})$. We test {formally}\footnote{
A more rigorous way  would employ $(1+|v|^2)^{1/2}\eta_R$ then take the limit $R\to +\infty$ as in the proof of $L^1$ bound.} \eqref{1.fp} against $(1+|v|^2)^{1/2}$ and obtain
\begin{align*}
\int_{\R^6}(1+|v|^2)^{1/2}f \dd{x} \dd{v} = &
\int_{\R^6}(1+|v|^2)^{1/2}f_{0} \dd{x} \dd{v}
+\eps \int_0^T \int_{\R^6}\Delta_v((1+|v|^2)^{1/2})~ f \dd{x} \dd{v}dt\\
&+\sigma \int_0^T \int_{\R^6}(1+|v|^2)^{1/2} Q_\delta(\tilde{f},f)\dd{x} \dd{v}dt.
\end{align*}
We focus on the contribution from the collision term:
\begin{align*}
\int_{\R^6} & (1+|v|^2)^{1/2} Q_\delta(\tilde{f},f)\dd{x} \dd{v} \\
&= 
-\int_{\R^6}\int_{\R^6}\frac{v N_\delta(v-v^*)}{(1+|v|^2)^{1/2}}
\kappa(x-x^*)(
\tilde{f}^*\nabla_{v}f - f \nabla_{v^*}\tilde{f}^*)
\dd{x}^* \dd{v}^* \dd{x} \dd{v}\\
&\leq C(\delta)\|f_0\|_{L^1(\R^6)}(\|{f}\|_{L^2(\R^6)}+\|\nabla_v{f}\|_{L^2(\R^6)}).
\end{align*}
Then 
\begin{align*}
  \sigma  \int_0^T \int_{\R^6}  (1+|v|^2)^{1/2} Q_\delta(\tilde{f},f)\dd{x} \dd{v}\le& \;2 C^2(\delta) \|f_0\|^2_{L^1(\R^6)}+ \int_0^T \|{f}\|^2_{L^2(\R^3\times\R^3)} + \|\nabla_v{f}\|^2_{L^2(\R^3\times\R^3)}\;dt \\
    \le& \; 2 C^2(\delta)\|f_0\|^2_{L^1(\R^6)} + C(T,\varepsilon,\delta)\int_{\R^6}f_0^2 \dd{x} \dd{v} ,
\end{align*}
using (\ref{stima_eps}). The desired bound is summarized in the following inequality: 
\begin{align*}
\int_{\R^6}(1+|v|^2)^{1/2}f \dd{x} \dd{v} \le  &
\int_{\R^6}(1+|v|^2)^{1/2}f_{0} \dd{x} \dd{v}\\
& +4\eps \int_0^T \int_{\R^6}f_0 \dd{x} \dd{v}dt +2 C^2(\delta)\|f_0\|^2_{L^1(\R^6)} + C(T,\varepsilon,\delta)\int_{\R^6}f_0^2 \dd{x} \dd{v} .
\end{align*}
 The computations for  the  spatial moments are easier because there is no contribution from the collision operator. We have therefore
\begin{align*}
\int_{\R^6}(1+|x|^2)^{1/2}f \dd{x} \dd{v} = &
\int_{\R^6}(1+|x|^2)^{1/2}f_{0} \dd{x} \dd{v}
+\eps \int_0^T \int_{\R^6}\Delta_x((1+|x|^2)^{1/2})~ f \dd{x} \dd{v}dt\\
&- \int_0^T \int_{\R^6} (1+|x|^2)^{1/2} e^{-\delta |v|^2} v \cdot \nabla_x f \dd{x} \dd{v} \\
\le & \; C(T, \|f_0\|_{L^1_1(\R^6)} , \varepsilon, \delta),
\end{align*}
after integrating by parts in the last integral:
$$
\left| \int_{\R^6} (1+|x|^2)^{1/2} e^{-\delta |v|^2} v \cdot \nabla_x f \dd{x} \dd{v} \right|= \left|  \int_{\R^6} f e^{-\delta |v|^2}  v \cdot \nabla_x (1+|x|^2)^{1/2}  \dd{x} \dd{v} \right| \le c(\delta) \|f_0\|_{L^1(\R^6)}.
$$

\item $H^2$ estimates 
We test \eqref{1.fp} against $ -(\Delta_x + \Delta_v)f$ and get
\begin{align*}
\frac{1}{2} \int_{\R^6}|\nabla_x f|^2 + |\nabla_v f|^2 \dd{x} \dd{v} &- \frac{1}{2}\int_{\R^6}|\nabla_x f_0|^2 + |\nabla_v f_0|^2 \dd{x} \dd{v} \\
& +\eps \sum_{i} \int_0^T\int_{\R^6}( |\nabla_x \partial_{x_i} f|^2 + |\nabla_x \partial_{v_i} f|^2+ |\nabla_x \partial_{v_i} f|^2 + |\nabla_v \partial_{v_i} f|^2 )\dd{x} \dd{v}dt \nonumber\\ 
=&   -\sigma \int_0^T\int_{\R^6} (\Delta_x + \Delta_v)f \;   Q_\delta(\tilde{f},f)\dd{x} \dd{v}dt \\
&- \int_0^T \int_{\R^6} (\Delta_x + \Delta_v)f \;  e^{-\delta |v|^2} v \cdot \nabla_x f \dd{x} \dd{v}.
\end{align*}
The contribution from the transport and $\Delta_x$ is zero, while the one from $\Delta_v$ is bounded thanks to (\ref{stima_eps}): 
\begin{align*}
     \int_{\R^6} \Delta_x f \;  e^{-\delta |v|^2} v \cdot \nabla_x f \dd{x} \dd{v} &= \sum_{i,j=1}^3 \int_{\R^6} f_{x_i x_i}f_{x_j} \;  e^{-\delta |v|^2} v_j  \dd{x} \dd{v} = - \sum_{i,j=1}^3 \frac{1}{2}\int_{\R^6} \partial_{x_j}(f_{x_i})^2 \;  e^{-\delta |v|^2} v_j  \dd{x} \dd{v} =0,\\
      \int_{\R^6} \Delta_v f \;  e^{-\delta |v|^2} v \cdot \nabla_x f \dd{x} \dd{v} =& \sum_{i,j=1}^3 \int_{\R^6} f_{v_i v_i}f_{x_j} \;  e^{-\delta |v|^2} v_j  \dd{x} \dd{v} \\
      = &- {\underbrace{\sum_{i,j=1}^3 \frac{1}{2}\int_{\R^6} \partial_{x_j}(f_{v_i})^2 \;  e^{-\delta |v|^2} v_j  \dd{x} \dd{v}}_{=0}}  
      +\int_{\R^6} \int_{\R^6} f_{v_i }f_{x_j} \partial_{v_i}(e^{-\delta |v|^2} v_j)\dd{x} \dd{v} \\
      \le & \; C(T,\varepsilon,\delta, \|f_0\|_{L^1})\int_{\R^6}f_0^2\;  \dd{x} \dd{v} .
\end{align*}
Let's look at the contribution from the collision operator: 
\begin{align*}
 -\sigma \int_{\R^6} \Delta_v f \;   Q_\delta(\tilde{f},f)\dd{x} \dd{v} & =  -\sigma \sum_i \int_{\R^6} \partial_{v_i v_i} f \Div_v(A_\delta[\tilde f] \nabla f - f~b_\delta[\tilde f]) \;  \dd{x} \dd{v} \\
 & =: \;  I_1+I_2.  \\
 \end{align*}
Integration by parts yields 
\begin{align*}
I_1 =  -\sigma \int_{\R^6}  \langle \nabla_v (\partial_{v_i} f),  A_\delta[\tilde f] \nabla_v (\partial_{v_i}f) \rangle \;dxdv - \sigma \int_{\R^6}  \langle \nabla_v (\partial_{v_i} f),  \partial_{v_i} (A_\delta[\tilde f])  \nabla_v f \rangle \;dxdv
\end{align*}
Since 
$$
|\partial_{v_i} (A_\delta[\tilde f])| \le |b_\delta[\tilde{f}]|\leq \|\Div_v (N_\delta)\kappa\|_{L^\infty(\R^6)}\|\tilde{f}\|_{L^1(\R^6)} \leq C(\delta)\|f_0\|_{L^1(\R^6)},
$$
we can estimate $I_1$ as 
\begin{align*}
I_1  \le \; &   \|\partial_{v_i} (A_\delta[\tilde f])\|_{L^\infty_{x,v}} \int_{\R^6}  |\nabla_v (\partial_{v_i} f)|| \nabla_v f | \;dxdv \\
\le & \; \frac{\varepsilon}{10}  \int_{\R^6}  |\nabla_v (\partial_{v_i} f)|^2 \;dxdv + C(\delta, \varepsilon, \|f_0\|_{L^1}) \int_{\R^6}  |\nabla_v f|^2 \;dxdv \\
\le & \; \frac{\varepsilon}{10}  \int_{\R^6}  |\nabla_v (\partial_{v_i} f)|^2 \;dxdv + C(\delta, \varepsilon, \|f_0\|_{L^1}) \int_{\R^6}  f_0^2 \;dxdv, 
\end{align*}
using (\ref{stima_eps}) to bound the term with $|\nabla_v f|$.  The other term
\begin{align*}
    I_2 = \; &   \sigma \sum_i \int_{\R^6} \partial_{v_i v_i} f\;  \Div_v( f~b_\delta[\tilde f]) \;  \dd{x} \dd{v} \\
    = \; & \sigma \sum_i \int_{\R^6} \nabla_v (\partial_{v_i} f)\; \cdot  \partial_{v_i} (f~b_\delta[\tilde f]) \;  \dd{x} \dd{v} \\
     =\; &   \sigma \sum_i \int_{\R^6} \partial_{v_i} f \; \nabla_v (\partial_{v_i} f)\; \cdot  b_\delta[\tilde f]  \;  \dd{x} \dd{v} +  \sigma \int_{\R^6} f \; \nabla_v (\partial_{v_i} f)\; \cdot  \partial_{v_i} b_\delta[\tilde f]  \;  \dd{x} \dd{v} . 
\end{align*}
Using once more that 
\begin{align*}
    |\partial_{v_i} b_\delta[\tilde f]| +| b_\delta[\tilde f]|  \leq C(\delta)\|f_0\|_{L^1(\R^6)},
\end{align*}
we get 
\begin{align*}
    I_2 \le \;& \; \frac{\varepsilon}{5}  \int_{\R^6}  |\nabla_v (\partial_{v_i} f)|^2 \;dxdv + C(\delta, \varepsilon, \|f_0\|_{L^1}) \int_{\R^6}  f^2 + |\nabla_v f|^2 \;dxdv \\
    \le \;& \; \frac{\varepsilon}{5}  \int_{\R^6}  |\nabla_v (\partial_{v_i} f)|^2 \;dxdv + C(\delta, \varepsilon, \|f_0\|_{L^1})\|f_0\|^2_{L^2}. 
\end{align*}
We now consider the contribution from the collision operator and $\Delta_x$: 
\begin{align*}
 -\sigma \int_{\R^6} \Delta_x f \;   Q_\delta(\tilde{f},f)\dd{x} \dd{v} & =  -\sigma \sum_i \int_{\R^6} \partial_{x_i x_i} f \Div_v(A_\delta[\tilde f] \nabla f - f~b_\delta[\tilde f]) \;  \dd{x} \dd{v} \\
 & =: \;  J_1+J_2.  \\
 \end{align*}
 As before, integration by parts yields
 \begin{align*}
J_1 =  -\sigma \int_{\R^6}  \langle \nabla_v (\partial_{x_i} f),  A_\delta[\tilde f] \nabla_v (\partial_{x_i}f) \rangle \;dxdv - \sigma \int_{\R^6}  \langle \nabla_v (\partial_{x_i} f),  \partial_{x_i} (A_\delta[\tilde f])  \nabla_v f \rangle \;dxdv.
\end{align*}
Since 
$$
|\partial_{x_i} (A_\delta[\tilde f])| \leq \|\partial_{x_i}(\kappa)N_\delta\|_{L^\infty(\R^6)}\|\tilde{f}\|_{L^1(\R^6)} \leq C(\delta)\|f_0\|_{L^1(\R^6)},
$$
we can estimate $J_1$ as 
\begin{align*}
J_1  \le \; &   \|\partial_{x_i} (A_\delta[\tilde f])\|_{L^\infty_{x,v}} \int_{\R^6}  |\nabla_v (\partial_{x_i} f)|| \nabla_v f | \;dxdv \\
\le & \; \frac{\varepsilon}{10}  \int_{\R^6}  |\nabla_v (\partial_{x_i} f)|^2 \;dxdv + C(\delta, \varepsilon, \|f_0\|_{L^1}) \int_{\R^6}  |\nabla_v f|^2 \;dxdv \\
\le & \; \frac{\varepsilon}{10}  \int_{\R^6}  |\nabla_v (\partial_{v_i} f)|^2 \;dxdv + C(\delta, \varepsilon, \|f_0\|_{L^1}) \int_{\R^6}  f_0^2 \;dxdv, 
\end{align*}
using (\ref{stima_eps}) to bound the term with $|\nabla_v f|$. The other term 
\begin{align*}
    J_2     =\; &   \sigma \sum_i \int_{\R^6} \partial_{x_i} f \; \nabla_v (\partial_{x_i} f)\; \cdot  b_\delta[\tilde f]  \;  \dd{x} \dd{v} +  \sigma \int_{\R^6} f \; \nabla_v (\partial_{x_i} f)\; \cdot  \partial_{x_i} b_\delta[\tilde f]  \;  \dd{x} \dd{v} . 
\end{align*}
We can bound $\partial_{x_i} b_\delta [\cdot ]$ as 
\begin{align*}
    |\partial_{x_i} b_\delta[\tilde f]|   \leq C(\delta)\|f_0\|_{L^1(\R^6)},
\end{align*}
and obtain
\begin{align*}
    J_2    \le \;& \; \frac{\varepsilon}{5}  \int_{\R^6}  |\nabla_v (\partial_{x_i} f)|^2 \;dxdv + C(\delta, \varepsilon, \|f_0\|_{L^1})\|f_0\|^2_{L^2}. 
\end{align*}
Summarizing we get 
\begin{align} \label{stim_H2}
\frac{1}{2} \int_{\R^6}|\nabla_x f|^2 + |\nabla_v f|^2 \dd{x} \dd{v} & +\frac{\eps}{2} \sum_{i} 
\int_0^T\int_{\R^6}( |\nabla_x \partial_{x_i} f|^2 
+ |\nabla_x \partial_{v_i} f|^2
+ |\nabla_v \partial_{x_i} f|^2 
+ |\nabla_v \partial_{v_i} f|^2 )\dd{x} \dd{v}dt   \nonumber \\
\le& C(\delta, \varepsilon, \|f_0\|_{L^2(\R^6)}, \|f_0\|_{L^1(\R^6)}) + \frac{1}{2}\int_{\R^6}|\nabla_x f_0|^2 + |\nabla_v f_0|^2 \dd{x} \dd{v}.
\end{align}

\end{itemize}

{\em{Compactness argument.}}
Estimate (\ref{stim_H2}) guarantees that 
$$
f \in  L^\infty(0,T, H^1(\mathbb{R}^6), \quad f_t \in L^2 (0,T, L^2(\mathbb{R}^6)) . 
$$
Since $L^2(\mathbb{R}^6) \hookrightarrow
L^1_{{-4}}(\mathbb{R}^6)$, and 
$$
H^1(\mathbb{R}^6) \cap L^1_1(\mathbb{R}^6) \hookrightarrow\hookrightarrow L^1(\mathbb{R}^6) \hookrightarrow L^1_{{-4}}(\mathbb{R}^6),
$$
Aubin-Lions theorem guarantees that the functional space 
$$
W:=\{ f \in L^\infty(0,T, H^1(\mathbb{R}^6) \cap L^1_1(\mathbb{R}^6)) \; | \; f_t \in L^2 (0,T, L^1_{{-4}}(\mathbb{R}^6)) \}
$$
is compactly embedded in $C([0,T], L^1(\mathbb{R}^6))$. This implies that $\mathcal{T}$ is a compact map.  $\mathcal{T}$ is also continuous and $\mathcal{T}(\cdot,0):X\to X$ is constant. Finally, the set $\{f\in X~:~\mathcal{T}(f,\sigma)=f\}$ is bounded in $X$ uniformly in $\sigma$ thanks to mass conservation \eqref{app.mass}. Leray-Schauder's fixed point theorem implies the existence of a fixed point $f \in X$ for $\mathcal{T}(\cdot,1)$, i.e.~a solution $f$ to \eqref{1.app} that satisfies bound $(\ref{stim_H2})$.

\subsection*{Estimates for the approximated solution.}

In this section we derive estimates for the solution to the nonlinear problem \eqref{1.app}:
\begin{align}\label{1.app2}
f_t + e^{-\delta|v|^2}v\cdot \nabla_x f - \eps(\Delta_x + \Delta_v)f = Q_\delta(f,f),\qquad x, v\in\R^3.~~
\end{align}
We recall that 
\begin{align*}%\label{Q.app}
Q_\delta (f,g) 
&= \Div_v\int_{\R^3}\int_{\R^3}N_\delta(v-v^*)\kappa(x-x^*)(f^*\nabla_v g - g\nabla_{v^*} f^*)\dd x^*\dd v^* 
\end{align*}
with 
\begin{comment}
\begin{align}\label{kappa.app.2}
\kappa_\delta(y) := \kappa(y)e^{-\delta\bk{y}},
\end{align}
and 
\end{comment}
\begin{equation}
\label{N.app-delta}
N_\delta(w)= {\underbrace{(\delta + |w|^2)^{\frac{\gamma}{2}}
{(1+\delta|w|^2)^{-K}}
}_{=: \alpha_\delta}}\; |w|^2\left(
\eye - \frac{w\otimes w}{|w|^2}\right), 
\quad \gamma \le 0,\quad 
{K \equiv \frac{1}{2}(\gamma+2)_+}\, .
\end{equation}
Equation \eqref{1.app2} preserves mass and momentum of $f$:
\begin{align}%\label{app.cons.1}
\int_{\R^6}
\begin{pmatrix}
1\\ v
\end{pmatrix}
f(t) \dd{x} \dd{v} = 
\int_{\R^6}
\begin{pmatrix}
	1\\ v
\end{pmatrix}
f_{0} \dd{x} \dd{v} \qquad\forall t>0.
\end{align}
While the regularized collision operator $Q_\delta$ also preserves the second moment of $f$, the term $-\eps\Delta_{x,v}f$ yields instead some contributions:
\begin{align*}
\frac{d}{dt}\int_{\R^6}(|x|^2 + |v|^2)f\dd x\dd v &= 
2\int_{\R^6}e^{-\delta|v|^2}v\cdot x f\dd x\dd v 
+12\eps \int_{\R^6}f\dd x\dd v\\
&\lesssim  \int_{\R^6}(|x|^2 + |v|^2)f\dd x\dd v + \rho_0,
\end{align*}
which implies 
\begin{align}\label{mom.eps.delta}
\sup_{t\in [0,T]}\int_{\R^6}(|x|^2 + |v|^2)f(t)\dd x\dd v \leq C(T)\quad\forall T>0.
\end{align}
Moreover $f$ solution to \eqref{1.app2} satisfies the $H$ theorem: 
\begin{align}\label{app.S.int}
\int_{\R^6}f(t)\ln f(t)\dd x\dd v
& +4\eps\int_0^t\int_{\R^6}(
|\nabla_x\sqrt{f}|^2 + |\nabla_v \sqrt{f}|^2)\dd{x} \dd{v}\dd{t'} \\ \nonumber
&+\frac{1}{2}\int_0^t\int_{\R^6}\int_{\R^6}(
\nabla_v\log(f) - \nabla_{v^*}\log(f^*))\cdot 
N_\delta(v-v^*)\kappa(x-x^*)ff^*\\ \nonumber
&\qquad \times (
\nabla_v\log(f) - \nabla_{v^*}\log(f^*))\dd{x}^* \dd{v}^* \dd{x} \dd{v} 
\dd{t'}\\ \nonumber
&\le \int_{\R^6}f_{0}\ln f_{0}\dd x\dd v,\quad t>0.
%\leq & 0.
\end{align}

{\bf{(i) Ellipticity in $v$.}} The coercivity estimate of Lemma \ref{lem.aA.lb} should be modified to account for the $\delta$-regularization in $A_\delta[f]$. Recall that 
$$
A_\delta[f] :=\int_{\R^6}(\delta + |w|^2)^{\frac{\gamma}{2}}
{(1+\delta|w|^2)^{-K}}
\; |w|^2\left(
\eye - \frac{w\otimes w}{|w|^2}\right)\kappa(x-x^*)f(x^*,v^*) \dd{x}^* \dd{v}^*.
 $$
 This new diffusion matrix satisfies the lower bound: 
 \begin{align}\label{delta-lem.aA.lb}
     \xi\cdot A_\delta[f]\xi \geq \theta \;  
\frac{(\delta +\bk{v}^{2})^{\frac{\gamma}{2}}
{(1+\delta|v|^2)^{-K}}
}{\bk{x}^{\lambda}} \; |\xi|^2,\qquad\forall\xi\in\R^3 .
 \end{align}
The proof of (\ref{delta-lem.aA.lb}) follows from the one of Lemma \ref{lem.aA.lb}. The only modification occurs in inequality (\ref{inizion_modif}), which now reads as 
\begin{align*}
e\cdot A_\delta[f]\; e &\geq \ell(x)\int_{\mathcal{D}_R(v)}
(\delta +|w|^{2})^{\frac{\gamma}{2}} 
{(1+\delta|w|^2)^{-K}}
|w|^{2}\; e\cdot\Pi(w)e\,\dd w .
\end{align*}
Since $w\in \mathcal{D}_R(v)$ and therefore
$$
|w| \lesssim C \bk{v},
$$
we deduce that 
$$
\delta + |w|^2 \lesssim \delta + \bk{v}^2,\quad 
{
1+\delta |w|^2 \lesssim 1+\delta \bk{v}^2
}
$$
and 
$$
(\delta + |w|^2)^{\frac{\gamma}{2}}
{(1+\delta|w|^2)^{-K}}
\gtrsim (\delta + \bk{v}^2)^{\frac{\gamma}{2}}
{(1+\delta|v|^2)^{-K}}.
$$
The last inequality leads to the equivalent of (\ref{almost-there}), which reads  as 
\begin{align*}
e\cdot A_\delta[f]e &\geq \tilde{\ell}\bk{x}^{-\lambda}(\delta + \bk{v}^2)^{\frac{\gamma}{2}}
{(1+\delta|v|^2)^{-K}}
\int_{\mathcal{D}_R(v)}
(|w|^{2} - (e\cdot w)^2)\dd w.
\end{align*}
The rest of the proof is identical. \\

{\bf{(ii) Uniform in $\delta$ and $\varepsilon$ regularity estimates for $-2< \gamma \le 0$.}} 

\begin{itemize}
    \item {\em{Moments}}: Solutions to (\ref{1.app2}) satisfy the moments estimates
\begin{align}\label{mom-delta}
	\sup_{t\in [0,T]}\int_{\R^6}(|x|^{\mm}+|v|^{\mm})f(x,v,t)\dd x\dd v < C(f_{in},T) \quad\forall \; T>0,
\end{align}
where the coefficient $\mm$ is as in \eqref{hp.ic}. The proof is exactly the same as the one of Lemma \ref{ss.moments}$(i)$. The constant $C(f_{in},T)$ is independent of $\delta$. 

\item  {\em{Boundedness}}: Any smooth solution to (\ref{1.app2}) is uniformly bounded with respect to $\delta$: 
$$
\sup_{t\in [0,T]}\|f(t)\|_{L^\infty_{x,v}}\leq C(T)\|f_{in}\|_{L^\infty_{x,v}}\quad\forall \; T>0.
$$
The constant $C(T)$ is independent from $\delta$. The proof is a straightforward adaptation of the one of Lemma \ref{lem.Linf}. We start by using $f^{p-1}$ as test function in (\ref{1.app2}), and obtain, after  integrations by parts,
\begin{align}\label{g.neg}
\frac{d}{dt}\int_{\R^6}& f^p\dd x\dd v
+\frac{4(p-1)}{p}\int_{\R^6}
\nabla_{v}f^{p/2}\cdot A_\delta[f]\nabla_{v} f^{p/2}\dd x\dd v
\le -(p-1)\int_{\R^6}f^{p}\Delta a_\delta[f]\dd x\dd v,
\end{align}
with 
\begin{align*}
-\Delta a_\delta[f] = 2 \int_{\mathbb{R}^6}\frac{(\delta + |w|^2)^{\frac{\gamma}{2}}}{
{(1+\delta|w|^2)^{K}}}
\left[ 3 + \gamma \;  \frac{|w|^2}{\delta + |w|^2}
{-2K\frac{\delta |w|^2}{1+\delta|w|^2}}
\right] \\ 
\times \kappa(x-y)\;  f(y,v-w)\;dwdy. 
\end{align*}
The second integral on the left-hand side can be controlled via \eqref{delta-lem.aA.lb}: 
\begin{align*}
\int_{\R^6}
\nabla_{v}f^{p/2}\cdot A_\delta [f]\nabla_{v} f^{p/2}\dd x\dd v\geq 
\theta\int_{\R^6}\frac{(\delta +\bk{v}^{2})^{\frac{\gamma}{2}}}{\bk{x}^{\lambda}}
{(1+\delta|v|^2)^{-K}}
|\nabla_{v}f^{p/2}|^2\dd x\dd v.
\end{align*}
The integral on the right-hand side of \eqref{g.neg} will be rewritten as:
\begin{align*}
	&-\int_{\R^6} f^{p}\Delta a_\delta[f]\dd x\dd v \\
	&= -\int_{\R^6}f^{p}\frac{(\delta +\bk{v}^{2})^{\frac{\gamma}{2}}}{
    {(1+\delta|v|^2)^{K}}}
    \bk{x}^{-\lambda}\, \bk{x}^{\lambda}\frac{
    {(1+\delta|v|^2)^{K}}}{(\delta +\bk{v}^{2})^{\frac{\gamma}{2}}}\Delta a_\delta[f]\dd x\dd v.
\end{align*}
The rest of the proof follows the exact same steps as the one of Lemma \ref{lem.Linf}. 

\item  {\em{$H^s$ estimates}} Follow from Theorem \ref{thm.reg}.

 \end{itemize}

{\bf{(iii) Uniform in $\delta$ and $\varepsilon$ regularity estimates for ${{-\sqrt{8}}  <  \gamma \le -2}$}}.
\begin{itemize}
\item {\em{Fisher information.}} Solutions to (\ref{1.app2}) satisfy
\begin{align}
\sup_{t\in [0,T]}\; \int \frac{|\nabla_v f|^2}{f}+  \frac{|\nabla_x f|^2}{f} \;dxdv  + 2 \varepsilon \int_0^T \int_{\mathbb{R}^6} f | \textrm{Hess}\; ( \ln f)|^2 \; dxdv& \le C(T, f_{in}, c_0,\gamma). \label{Fisher_eps}
\end{align}
if
$$
\left\|  \kappa\left(\frac{  |\nabla  {{\kappa}}|}{{{{\kappa}}}} \right)^{-\gamma} \right\|_{L^\infty(\mathbb{R}^3)} \le c_0, 
$$
and  $\alpha_\delta $ is such that 
$$
\sup_{r\in \mathbb{R}_+} \left ( {{\frac{(\alpha_\delta')^2r^2 }{\alpha_\delta^2}}}\right) < {{8}}, \quad  
\alpha(r)  \le r^{\gamma}, \quad \gamma \le -2
%\sup_{r\ge \theta} r^2 \alpha_\delta(r) < +\infty. 
$$
 With $\alpha_\delta(r) = ( \delta + r^2)^{\frac{\gamma}{2}}$
we have
{
\begin{align*}
\frac{|\alpha_\delta'(r)|r}{\alpha_\delta(r)} = 
\left| 
\gamma \frac{r^2}{\delta + r^2} \right|
\leq |\gamma| <\sqrt{8}.
\end{align*}
}

\item {\em{Moments.}} Solutions to (\ref{1.app2}) satisfy the moments estimates
\begin{align}\label{mom-delta-verysoft}
	\sup_{t\in [0,T]}\int_{\R^6}(|x|^{\mm}+|v|^{\mm})f(x,v,t)\dd x\dd v < C(f_{in},T) \quad\forall \; T>0,
\end{align}
where the coefficient $\mm$ is as in \eqref{hp.ic}. The proof is exactly the same as the one of  Lemma \ref{ss.moments}. The constant $C(f_{in},T)$ is independent on $\delta$ and $\varepsilon$. The condition (\ref{condFL3}) is provided by the above estimate of the Fisher information (\ref{Fisher_eps}) and Sobolev embedding.

\item  {\em{Boundedness and regularity.}} The $L^\infty$ bound holds, thanks to Lemma \ref{lem:bounded-verysoft}. $H^s$-estimates follow from Theorem \ref{thm.reg}. 

\end{itemize}

The preceding steps (i)-(iii) establish regularity estimates for the solution to the approximated problem \eqref{1.app2}—specifically the bounds \eqref{dabf}, \eqref{dabf-v}, and \eqref{dabf-t}—that are entirely uniform with respect to the regularization parameters $\varepsilon$, $\delta$. These uniform bounds allow us to rigorously pass to the limit as $\varepsilon, \delta \to 0$, yielding a smooth, global-in-time solution to the original fuzzy Landau equation \eqref{1}. 

The uniqueness of solutions to \eqref{1} is proved by considering the equation satisfied by the difference $u\equiv f_1-f_2$ between two solutions $f_1$, $f_2$ with the same initial datum, and performing similar estimates to the ones carried out in the proof of Theorem \ref{thm.reg}. For more details, see \cite{Caja2026uniquenessFuzzy}. 

%In particular, in \eqref{a.decay}--\eqref{Delta.a.bounded} the dependence of the bounds for $A[f]$, $a[f]$ and their derivatives on suitable weighted norms of $f$ must be made explicit in order to effectively apply the aforementioned estimates to the uniqueness argument. Such modifications can be obtained in a rather straightforward way from Corollary \ref{coro.DH.est.1b}.

This completes the proof of Theorem \ref{Thm1}.

\section{Variational formulation - the GENERIC formalism}
\label{sec:generic}

The GENERIC formalism was introduced in \cite{GO97} and stands for General Equation for the nonequilibrium Reversible-Irreversible Coupling. It was originally used to describe the dynamics of complex fluids. This formalism provides a unified framework for understanding how certain systems evolve towards equilibrium, considering both reversible and irreversible processes. The reversible contribution is provided by an energy conservation, the irreversible by the production of an entropy. 

In the context of kinetic equations, Erbar and He for the first time adapted the GENERIC formalism to the fuzzy Boltzmann  in \cite{Erbar24, Erbar25}. Recently, Duong and He in \cite{Duong_Li_1} showed that also the fuzzy Landau equation (\ref{1}) can be written in the GENERIC formalism. 

Here we present this structure in a slightly simpler and more formal way.

An evolution equation is represented in the GENERIC framework as follows:
\begin{align}\label{GENERIC}
\pa_t f = L(f)\dd E(f) + M(f)\dd S(f)
\end{align}
where $E, S : Z\to\R$ are suitable functionals defined in the state space $Z$ representing {\em energy} and {\em entropy} (respectively), while for every given function $f\in Z$, $L(f)$, $M(f)$ are operators mapping cotangent vectors to tangent ones with the following properties:
\begin{enumerate}
\item $L=L(f)$ is skew-symmetric and satisfies the Jacobi identity:
\begin{align*}
\{\{G_1,G_2\}_L,G_3\}_L + \{\{G_2,G_3\}_L,G_1\}_L + \{\{G_3,G_1\}_L,G_2\}_L = 0
\end{align*}
for every choice of $G_1$, $G_2$, $G_3 : Z\to\R$, where the Poisson brackets $\{\cdot,\cdot\}_L$ are defined as
$$\{F,G\}_L = \langle\dd F, L\dd G\rangle \qquad \forall F,G: Z \to \R. $$
\item $M=M(f)$ is symmetric and positive semi-definite.
\item The {\em degeneracy conditions} are satisfied:
$$
L(f)\dd S(f) = 0,\qquad M(f)\dd E(f) = 0.
$$
\end{enumerate}
As a consequence of the degeneracy condition, the energy $E$ is conserved while the entropy $S$ is non-decreasing along the solutions of \eqref{GENERIC} (here $\langle \cdot,\cdot\rangle$ is the duality product between cotangent and tangent vectors, in this order):
\begin{align*}
\frac{d}{dt}E(f(t)) &= \langle \dd E(f(t)), \pa_t f\rangle \\
&= \langle \dd E(f(t)), L(f)\dd E(f)\rangle + \langle \dd E(f(t)), M(f)\dd S(f)\rangle \\
&= \langle \dd E(f(t)), L(f)\dd E(f)\rangle + \langle \dd S(f(t)), M(f)\dd E(f)\rangle \\
&= 0,\\
\frac{d}{dt}S(f(t)) &= \langle \dd S(f(t)), \pa_t f\rangle \\
&= \langle \dd S(f(t)), L(f)\dd E(f)\rangle + \langle \dd S(f(t)), M(f)\dd S(f)\rangle\\
&= -\langle \dd E(f(t)), L(f)\dd S(f)\rangle + \langle \dd S(f(t)), M(f)\dd S(f)\rangle\\
&= \langle \dd S(f(t)), M(f)\dd S(f)\rangle\\
&\geq 0.
\end{align*}
The terms $L\dd E$, $M\dd S$ describe the Hamiltonian and dissipative part of the dynamics, respectively.

We formally define the Riemannian metric induced by $M$ on the tangent and cotangent space as follows
\begin{align*}
(x,y)_{M^{-1}} = x\cdot M^{-1}y = \langle M^{-1}y , x\rangle,\quad 
(v,w)_M = v\cdot Mw = \langle v, M w\rangle.
\end{align*}
For a generic final time $T>0$, consider the functional:
\begin{align}\label{J}
J(f) = S(f_0)-S(f(T))+\frac{1}{2}\int_0^T\left(
\|\pa_t f - L(f)\dd E(f)\|_{M^{-1}}^2 + \|\dd S(f)\|_M^2
\right)dt.
\end{align}
Let us check that $J$ is nonnegative along every curve $f : [0,T]\to Z$ and $J(f)=0$ if and only if $f$ is a solution to \eqref{GENERIC}:
\begin{align*}
\frac{1}{2}\| &\pa_t f - L(f)\dd E(f) - M(f)\dd S(f)\|_{M^{-1}}^2 \\
&=\frac{1}{2}\|\pa_t f - L(f)\dd E(f)\|_{M^{-1}}^2 + \frac{1}{2}\|M(f)\dd S(f)\|_{M^{-1}}^2
-(\pa_t f - L(f)\dd E(f), M(f)\dd S(f))_{M^{-1}}\\
&=\frac{1}{2}\|\pa_t f - L(f)\dd E(f)\|_{M^{-1}}^2 + \frac{1}{2}\|\dd S(f)\|_{M}^2
-\langle \dd S(f),\pa_t f\rangle + \langle \dd S(f), L(f)\dd E(f)\rangle\\
&=\frac{1}{2}\|\pa_t f - L(f)\dd E(f)\|_{M^{-1}}^2 + \frac{1}{2}\|\dd S(f)\|_{M}^2
-\frac{d}{dt}S(f) - \langle \dd E(f), L(f)\dd S(f)\rangle,
\end{align*}
where in the last step we used the fact that $L(f)$ is skew-symmetric. Integrating the above identity in time and employing the degeneracy condition leads to the claim.

We write now \eqref{1} in the form \eqref{GENERIC}. The state space $Z$ will be a suitable subset of the set of functions $\R^{2d}\to\R$ depending on $(x,v)$.
The energy and entropy functionals are given by
\begin{align}\label{ES}
E(f) = \frac{1}{2}\int_{\R^6}|v|^2 f(x,v)\dd x\dd v,\qquad 
S(f) = -\int_{\R^6}f(x,v)\log f(x,v)\dd x\dd v.
\end{align}
In the following $f,g\in Z$ are generic state vectors. Also we define for brevity:
\begin{align*}
\mu = (x,v),\quad \mu^*=(x^*,v^*), \quad 
B(\mu-\mu^*) = \kappa(x-x^*)N(v-v^*).
\end{align*}
Now we define the operators $L$, $M$.
\begin{align*}
L(f)g &= \nabla_v \cdot (f\nabla_x g) - \nabla_x \cdot (f\nabla_v g),\\
M(f)g &= \frac{1}{2}\overline{\nabla}^\top (B(\mu-\mu^*)ff^*\overline{\nabla}g),\\
\overline{\nabla}g(\mu,\mu^*) &= \nabla_v g(\mu) - \nabla_{v^*}g(\mu^*),\\
\overline{\nabla}^\top &= \mbox{adjoint of $\overline{\nabla}$}.
\end{align*}
Let us compute $\overline{\nabla}^\top$. Notice that $\overline{\nabla}$ maps functions of $\mu$, $\R^6\to\R$, into functions of $(\mu,\mu^*)$, $\R^{12}\to\R$. Let $g = g(\mu)$, $f=f(\mu,\mu^*)$ arbitrary. Consider
\begin{align*}
\int_{\R^{12}}f\overline{\na}g \dd \mu\dd \mu^* &=
\int_{\R^{12}}f(\mu,\mu^*)\nabla_v g(\mu) \dd \mu\dd \mu^*
-\int_{\R^{12}}f(\mu,\mu^*)\nabla_{v^*} g(\mu^*) \dd \mu\dd \mu^*\\
&=
-\int_{\R^{12}}\nabla_v f(\mu,\mu^*) g(\mu) \dd \mu\dd \mu^*
+\int_{\R^{12}}\nabla_{v^*} f(\mu,\mu^*)g(\mu^*) \dd \mu\dd \mu^*\\
&=
-\int_{\R^{12}}\nabla_v f(\mu,\mu^*) g(\mu) \dd \mu\dd \mu^*
+\int_{\R^{12}}\nabla_{v} f(\mu^*,\mu)g(\mu) \dd \mu\dd \mu^*\\
&=
\int_{\R^6}g(\mu)\nabla_v\int_{\R^6}
(-f(\mu,\mu^*) + f(\mu^*,\mu))\dd\mu^* ~ \dd\mu .
\end{align*}
We obtain
\begin{align*}
\overline{\nabla}^\top f (\mu) = \nabla_v\cdot\int_{\R^6}(-f(\mu,\mu^*) + f(\mu^*,\mu))\dd\mu^*.
\end{align*}
Let us verify that the given structure really yields \eqref{1}.
Given that
\begin{align*}
\langle\dd E(f), g\rangle = \frac{1}{2}\int_{\R^6}|v|^2 g(x,v)\dd x\dd v,\qquad 
\langle\dd S(f), g\rangle = -\int_{\R^6}(\log(f)+1)g \dd x\dd v,
\end{align*}
it holds
\begin{align*}
L(f)\dd E(f) = \nabla_v \cdot (f\nabla_x \frac{|v|^2}{2}) - \nabla_x \cdot (f\nabla_v \frac{|v|^2}{2}) = -v\cdot\nabla_x f,
\end{align*}
as well as
\begin{align*}
M(f)\dd S(f) = -\frac{1}{2}\nabla_{v}\cdot
\int_{\R^6}(\psi(\mu,\mu^*)-\psi(\mu^*,\mu))\dd\mu^*,\\
\psi(\mu,\mu^*) = B(\mu-\mu^*)f(\mu)f(\mu^*)\overline{\nabla}g(\mu,\mu^*),\\
g(\mu) = -(\log(f(\mu))+1).
\end{align*}
Given that $\psi(\mu,\mu^*)=-\psi(\mu^*,\mu)$, it follows
\begin{align*}
M(f)\dd S(f) &= -\nabla_{v}\cdot\int_{\R^6}\psi(\mu,\mu^*)\dd \mu^*\\
&=\Div_v\int_{\R^6}B(\mu-\mu^*)f(\mu)f(\mu^*)(\overline{\nabla}\log(f))(\mu,\mu^*)
\dd\mu^*\\
&= \Div_v\int_{\R^3}\int_{\R^3}N(v-v^*)\kappa(x-x^*)(f^*\nabla f - f\nabla f^*)\dd x^*\dd v^* ,
\end{align*}
as desired.

\appendix

\section{A key functional inequality}\label{appendix_inequality}
In this section we prove inequality (\ref{sob_ineq}):
\begin{align} 
{{\frac{1}{\Lambda_0}}} \sum_{k=1}^3 \;   & \int  \frac{\kappa}{|v-w|^2}\;F( \sqrt{\alpha}  \tilde{b}_k  \cdot \nabla \ln F)^2 \;dz \label{sob_ineq-App} \\
  &  \le    \sum_{i,k=1}^3 \frac{1}{2} \int   \kappa  F \left(  |\tilde \xi_i \cdot  \nabla (  \sqrt{\alpha} \tilde{b}_k  \cdot \nabla \ln F)|^2  +   |\hat \xi_i \cdot  \nabla ( \sqrt{\alpha}  \tilde{b}_k  \cdot \nabla \ln F)|^2 \right)\;dz  \nonumber
\end{align}
for 
$$
{{\Lambda_0 = \frac{1}{4}}}.
$$

We borrow the notation from \cite{GS24} and \cite{Ji24a}. We start with the Bakry-
Emery $\Gamma_2$ criterion for the logarithmic Sobolev
inequality: 
\begin{align}\label{BEineq}
\int_{S^2} f \; \Gamma_2(\ln f, \ln f)\; d\sigma \ge  {{2}} \int_{S^2} f|\nabla_\sigma \ln f|^2 \; d\sigma.
\end{align}
One can check (see also Remark 9.14 in \cite{GS24}) that 
$$
 \Gamma_2(\ln f, \ln f) = \sum_{i,k=1}^3 (b_i \cdot \nabla(b_k \cdot \nabla \ln f))^2.
$$
Here  $ b_i  = b_i(\sigma)$ defined as 
    \begin{align}
    b_i(v-w) = |v-w| \; b_k(\sigma). \label{bksigma}
    \end{align}
Using the identity (see (9.3) in \cite{GS24})
$$
|\nabla_\sigma G|^2 = \sum_{k=1}^k(b_k(\sigma) \cdot \nabla G)^2,
$$
inequality (\ref{BEineq}) can be rewritten as 
\begin{align}\label{final_ineq_nonsym}
\sum_{i,k=1}^3 \int_{\mathbb{S}^2}  (b_i(\sigma) \cdot \nabla(b_k(\sigma) \cdot \nabla \ln f))^2 f \; d\sigma \ge  {{2}} \sum_{k=1}^3 \;  \int_{\mathbb{S}^2} (b_k(\sigma) \cdot \nabla \ln f)^2 f\;d\sigma.
\end{align}
Inequality (\ref{final_ineq_nonsym}) is the counterpart of the inequality in \cite[Lemma 9.13]{GS24} for symmetric functions. 

We consider now the right hand side of (\ref{sob_ineq-App}): summing over $i = 1,2,3$, we rewrite it as 
\begin{align*}
 & \sum_{k, i=1}^3    \frac{1}{2}  \int    {{\kappa}}\;  F \left(  |\tilde \xi_i \cdot  \nabla (  \sqrt{\alpha} \tilde{b}_k  \cdot \nabla \ln F)|^2  +   |\hat \xi_i \cdot  \nabla ( \sqrt{\alpha}  \tilde{b}_k  \cdot \nabla \ln F)|^2 \right)\;dz \\
&=  \sum_{k=1}^3      \int_{\mathbb{R}^6}    { {\kappa}}\int_{\mathbb{R}^6} \;  F \left(  |  \nabla_v (  \sqrt{\alpha} \tilde{b}_k  \cdot \nabla \ln F)|^2  +   | \nabla_w ( \sqrt{\alpha}  \tilde{b}_k  \cdot \nabla \ln F)|^2 \right)\;dz .
\end{align*}
We now use the inequality (see Lemma 8.3 in \cite{GS24}) 
$$
|\nabla_v h|^2 + |\nabla_w h|^2 \ge \frac{1}{2|v-w|^2} \sum_{i=1}^3 |\tilde{b}_i  \cdot \nabla h|^2
$$
to get 
\begin{align*}
& \sum_{k, i=1}^3    \frac{1}{2}  \int    {{\kappa}}\;  F \left(  |\tilde \xi_i \cdot  \nabla (  \sqrt{\alpha} \tilde{b}_k  \cdot \nabla \ln F)|^2  +   |\hat \xi_i \cdot  \nabla ( \sqrt{\alpha}  \tilde{b}_k  \cdot \nabla \ln F)|^2 \right)\;dz \\
&=  \sum_{k=1}^3      \int_{\mathbb{R}^6}    {{\kappa}}\int_{\mathbb{R}^6} \;  F \left(  |  \nabla_v (  \sqrt{\alpha} \tilde{b}_k  \cdot \nabla \ln F)|^2  +   | \nabla_w ( \sqrt{\alpha}  \tilde{b}_k  \cdot \nabla \ln F)|^2 \right)\;dz  \\
&\ge  \frac{1}{2}\sum_{i,k=1}^3 \int_{\mathbb{R}^6}    { {\kappa}}\int_{\mathbb{R}^6}\frac{F}{|v-w|^2}\left(\tilde{b}_i  \cdot \nabla (\sqrt{\alpha}  \tilde{b}_k  \cdot \nabla \ln F)\right)^2 dz. 
    \end{align*}
Hence, proving (\ref{sob_ineq-App}) reduces to prove  
    \begin{align}\label{ineq_interm}
      \frac{1}{2}\sum_{i,k=1}^3      & \int_{\mathbb{R}^6}    {{\kappa}}\left(\int_{\mathbb{R}^6}\frac{F}{|v-w|^2}\left(\tilde{b}_i  \cdot \nabla (\sqrt{\alpha}  \tilde{b}_k  \cdot \nabla \ln F)\right)^2 dvdw \right)dxdy \\
     & \ge {{\frac{1}{\Lambda_0}}}   \sum_{k=1}^3 \;    \int  \frac{\kappa}{|v-w|^2}\;F( \sqrt{\alpha}  \tilde{b}_k  \cdot \nabla \ln F)^2 \;dz,\quad {{\Lambda_0 = \frac{1}{4}}}.\nonumber
    \end{align}
    We perform a change of variables $v =  u+  r\sigma, w = u - r\sigma$ in $\mathbb{R}^6_{v,w}$: 
    $$
    dv\; dw = 8 r^2 \;d\sigma dudr 
    $$
    with $ r = \frac{1}{2}|v-w|$. The vectors $b_k(\sigma)$ defined in (\ref{bksigma}) are  
    $$
    b_k(v-w) = 2 r \; b_k(\sigma). 
    $$
   and  
    \begin{align}\label{ident_hope}
    \tilde{b}_k  \cdot \nabla G = 2 r \; b_k(\sigma) \ \cdot (\nabla_v - \nabla_w) G = 2 b_k(\sigma)\cdot \nabla_\sigma G. 
    \end{align}
   % and inequality (\ref{BEineq}) becomes 
  %  \begin{align}\label{final_ineq_nonsym_new_coord}
 %{{4^2}}  \sum_{i,k=1}^3 \int_{\mathbb{S}^2}( b_i(\sigma)\cdot \nabla_\sigma( b_k(\sigma)\cdot \nabla_\sigma \ln f))^2 f \; d\sigma \ge  {{2}} \cdot {{2^2}}\sum_{k=1}^3 \;  \int_{\mathbb{S}^2} ( b_k(\sigma)\cdot \nabla_\sigma \ln f)^2 f\;d\sigma.
%\end{align}

    The last equality holds because the vector $b_k$ belongs to the tangent space of $\mathbb{S}^2$. 
    Then 
    \begin{align*}
      \frac{1}{2}\sum_{i,k=1}^3      & \int_{\mathbb{R}^6}    { {\kappa}}\int_{\mathbb{R}^6}\frac{F}{|v-w|^2}\left(\tilde{b}_i  \cdot \nabla (\sqrt{\alpha}  \tilde{b}_k  \cdot \nabla \ln F)\right)^2 dvdw dxdy \\
       = \; &\frac{1}{2}\sum_{i,k=1}^3       \int_{\mathbb{R}^6}    {{\kappa}}\int_{\mathbb{R}^6}\frac{F \alpha }{|v-w|^2}\left(\tilde{b}_i  \cdot \nabla (  \tilde{b}_k  \cdot \nabla \ln F)\right)^2 dvdw dxdy \\
       = \; &\frac{1}{2}\sum_{i,k=1}^3       \int_{\mathbb{R}^6}    { {\kappa}}\int_{\mathbb{R}^3} \int_0^\infty \frac{ \alpha }{4 r^2}   \cdot     {{4^2}}\int_{\mathbb{S}^2} F \left({b}_i(\sigma)  \cdot \nabla_\sigma (  {b}_k(\sigma)  \cdot \nabla_\sigma \ln F)\right)^2\; d\sigma du (8 r^2 dr) dxdy \\
       \ge \; &\frac{1}{2}\sum_{k=1}^3       \int_{\mathbb{R}^6}    { {\kappa}}\int_{\mathbb{R}^3} \int_0^\infty \frac{ \alpha }{4 r^2}   \cdot     {{2}} \cdot {{4^2}} \;  \int_{\mathbb{S}^2} ( b_k(\sigma)\cdot \nabla_\sigma \ln F)^2 F\;d\sigma
       du (8 r^2dr) dxdy,
       \end{align*}
 using first (\ref{ident_hope}) twice inside the integrand and then (\ref{final_ineq_nonsym}). Applying again identity (\ref{ident_hope}) in the last integrand, we obtain 
 \begin{align*}
   \frac{1}{2}\sum_{k=1}^3  &     \int_{\mathbb{R}^6}    { {\kappa}}\int_{\mathbb{R}^3} \int_0^\infty \frac{ \alpha }{4 r^2}   \cdot     {{2}} \cdot {{4^2}} \;  \int_{\mathbb{S}^2} ( b_k(\sigma)\cdot \nabla_\sigma \ln F)^2 F\;d\sigma
       du(8 r^2 dr) dxdy   \\
  &  =\frac{1}{2}\sum_{k=1}^3       \int_{\mathbb{R}^6}    { {\kappa}}\int_{\mathbb{R}^3} \int_0^\infty \frac{ \alpha }{4 r^2}   \cdot     {{2}} \cdot {{4^2}} \cdot {{\frac{1}{4}}}\;  \int_{\mathbb{S}^2} ( \tilde{b}_k\cdot \nabla  \ln F)^2 F\;d\sigma
       du(8 r^2 dr) dxdy \\
       &=\sum_{k=1}^3       \int_{\mathbb{R}^6}    { {\kappa}}\int_{\mathbb{R}^3} \int_0^\infty \frac{ \alpha }{r^2}  \;  \int_{\mathbb{S}^2} ( \tilde{b}_k\cdot \nabla  \ln F)^2 F\;d\sigma
       du(8 r^2 dr) dxdy\\
        &=\sum_{k=1}^3       \int_{\mathbb{R}^6}    { {\kappa}}\int_{\mathbb{R}^6}  \frac{ \alpha }{r^2}  \;  ( \tilde{b}_k\cdot \nabla  \ln F)^2 F\;dvdw dxdy \\
        &= {{4}} \sum_{k=1}^3       \int_{\mathbb{R}^6}    { {\kappa}}\int_{\mathbb{R}^6}  \frac{ \alpha }{|v-w|^2}  \;  ( \tilde{b}_k\cdot \nabla  \ln F)^2 F\;dvdw dxdy 
       \end{align*}
       since $2r =|v-w| $. This proves (\ref{ineq_interm}) and consequently (\ref{sob_ineq-App}).

\section{Fisher information in $v$}\label{appendix_Fisher_Landau_v}

Proof of Lemma \ref{Lem:Fisher_Landau_v}.

\begin{proof}
We use (\ref{Equiv_v}) to rewrite 

\begin{align*}
\frac{d}{dt}  \int \frac{|\nabla_v f|^2}{f} \;dvdx = &\frac{d}{dt}  \sum_{i} \int \frac{(e_i \cdot \nabla_v f)^2 }{f} \;dvdx \\
= & \sum_{i} \; \; \frac{1}{2} \int   \frac{\tilde \xi_i \cdot  \nabla F }{F}  (\tilde \xi_i \cdot \nabla Q(F)) +   \frac{\hat \xi_i \cdot  \nabla F }{F}  (\hat \xi_i \cdot \nabla Q(F))\;dz   \nonumber \\
 &\quad  - \frac{1}{4} \int  \left( \frac{|\tilde \xi_i \cdot  \nabla F|^2 }{F^2}  +   \frac{|\hat \xi_i \cdot  \nabla F|^2 }{F^2} \right) Q(F) \;dz  \\
 & \quad - 8 \int (e_i \cdot \nabla_v \sqrt{f})(e_i \cdot  \nabla_x \sqrt{f}) \;dxdv,   \nonumber
\end{align*}
where $Q(F)$ is the linear operator that satisfies \eqref{e:lifted operator}
 \begin{align*}
  Q(F)(x,v,y,w) =\kappa(x-y)  \sum_{k=1}^3  \sqrt{\alpha}\; \tilde{b}_k \cdot \nabla(\sqrt{\alpha}  \tilde{b}_k \cdot \nabla F),
\end{align*} 
 and $F(x,v,y,w) := f(x,v)f(y,w)$. 

We consider separately the integrals with $\tilde \xi_i$ and $\hat \xi_i$. 
We rewrite 
\begin{align*}
\int   \frac{1}{2} & \frac{\tilde{\xi}_i \cdot  \nabla F }{F}  (\tilde{\xi}_i\cdot \nabla Q(F))dz -  \frac{1}{4} \int    \frac{(\tilde{\xi}_i \cdot \nabla F )^2 }{F^2} Q(F) \;dz \\
& =  \frac{1}{4} \int   Q(F) (  \tilde{\xi}_i \cdot \nabla \ln F)^2 + {2} F (\tilde{\xi}_i \cdot \nabla \ln F) \left(  \tilde{\xi}_i \cdot \nabla \left(\frac{Q(F)}{F}\right)\right)\;dz. 
\end{align*}
Similarly for the terms with $\hat \xi_i$. 

Next, we rewrite  $ \frac{Q(F)}{F}$ as
$$
\frac{Q(F)}{F}  = Q(\ln F) + \kappa\sum_k | \sqrt{\alpha} \; \tilde{b}_k \cdot \nabla \ln F|^2,
$$
and obtain
\begin{align} 
\frac{d}{dt}  \int \frac{|\nabla_v f|^2}{f} \;dvdx = \frac{1}{4}\sum_i \int   Q(F) |\hat{\xi}_i \cdot   \nabla \ln F|^2  \;dz & + \int  2 F (\hat{\xi}_i \cdot \nabla \ln F) \left(\hat{\xi}_i  \cdot \nabla Q(\ln F)\right) \;dz \nonumber \\
  +\sum_k  \int 2 \kappa  F (\hat{\xi}_i \cdot \nabla \ln F) &\left(\hat{\xi}_i  \cdot \nabla |  \sqrt{\alpha} \;  \tilde{b}_k \cdot \nabla \ln F|^2 \right)\;dz \nonumber \\
 +\frac{1}{4}\sum_i \int   Q(F) |\tilde{\xi}_i \cdot   \nabla \ln F|^2  \;dz & + \int  2 F (\tilde{\xi}_i \cdot \nabla \ln F) \left(\tilde{\xi}_i  \cdot \nabla Q(\ln F)\right) \;dz \nonumber \\
  + \sum_k \int 2\kappa  F (\tilde{\xi}_i \cdot \nabla \ln F) &\left(\tilde{\xi}_i  \cdot \nabla |  \sqrt{\alpha} \;  \tilde{b}_k \cdot \nabla \ln F|^2 \right)\;dz \nonumber \\
 =: \sum_i \frac{1}{4}(\hat{I}_1 + \hat{I}_2+ \hat{I}_3) &+\frac{1}{4}(\tilde{I}_1 + \tilde{I}_2+ \tilde{I}_3) \label{I1I2I3}.
\end{align}
Here we abuse the notation and hide the subindex $i$ in $I_1, I_2, I_3$. 
We start with $\hat{I}_2$. Since $\textrm{div}(  \kappa \sqrt{\alpha} \;    \tilde{b}_k)=0$, we have
\begin{align*}
\hat{I}_2 =  \sum_{k=1}^3 & \: 2    \int  \kappa  F (\hat{\xi}_i \cdot \nabla \ln F) \left(\hat{\xi}_i  \cdot \nabla (  \sqrt{\alpha} \;  \tilde{b}_k \cdot \nabla( \sqrt{\alpha} \;  \tilde{b}_k \cdot \nabla \ln F) )\right) \;dz  \\
 =  \sum_{k=1}^3  &\: 2  \int \kappa   (\hat{\xi}_i \cdot \nabla  F) \left(  \sqrt{\alpha} \;  \tilde{b}_k   \cdot \nabla (  \hat{\xi}_i \cdot \nabla( \sqrt{\alpha} \;  \tilde{b}_k \cdot \nabla \ln F) )\right) \;dz  \\
 &+2   \int  \kappa   (\hat{\xi}_i \cdot \nabla  F) [\hat{\xi}_i,\sqrt{\alpha} \;  \tilde{b}_k]\cdot \nabla ( \sqrt{\alpha} \;  \tilde{b}_k \cdot \nabla \ln F) \;dz  \\
 = \sum_{k=1}^3  & -2     \int \kappa\sqrt{\alpha} \;  \tilde{b}_k \cdot \nabla (\hat{\xi}_i \cdot \nabla  F)  (  \hat{\xi}_i \cdot \nabla( \sqrt{\alpha} \;  \tilde{b}_k \cdot \nabla \ln F) ) dz\\
 &+2    \int  \kappa  (\hat{\xi}_i \cdot \nabla F) [\hat{\xi}_i,\sqrt{\alpha} \;  \tilde{b}_k]\cdot \nabla ( \sqrt{\alpha} \;  \tilde{b}_k \cdot \nabla \ln F) \;dz \\
 =\sum_{k=1}^3  &  -2    \int  \kappa   \hat{\xi}_i \cdot \nabla (\sqrt{\alpha} \;  \tilde{b}_k  \cdot \nabla  F)  (  \hat{\xi}_i \cdot \nabla(\sqrt{\alpha} \;   \tilde{b}_k \cdot \nabla \ln F) )\;dz  \\
 &- 2 \int  {{-\kappa [\hat{\xi}_i,\sqrt{\alpha} \;  \tilde{b}_k]}}\cdot \nabla F  (  \hat{\xi}_i \cdot \nabla(\sqrt{\alpha} \;   \tilde{b}_k \cdot \nabla \ln F) ) \;dz   \\ 
 &+2   \int    \kappa (\hat{\xi}_i \cdot \nabla  F) [\hat{\xi}_i,\sqrt{\alpha} \;  \tilde{b}_k]\cdot \nabla ( \sqrt{\alpha} \;  \tilde{b}_k \cdot \nabla \ln F) \;dz \\
 = \sum_{k=1}^3 & - 2   \int  \kappa   \hat{\xi}_i \cdot \nabla ( F \sqrt{\alpha} \;  \tilde{b}_k  \cdot \nabla \ln F)  (  \hat{\xi}_i \cdot \nabla(\sqrt{\alpha} \;   \tilde{b}_k \cdot \nabla \ln F) )dz \\
 &- 2 \int   {{-\kappa [\hat{\xi}_i,\sqrt{\alpha} \;  \tilde{b}_k]}}\cdot \nabla F  (  \hat{\xi}_i \cdot \nabla( \sqrt{\alpha} \;  \tilde{b}_k \cdot \nabla \ln F) ) \;dz  \\
 & +2   \int   \kappa  (\hat{\xi}_i \cdot \nabla  F) [\hat{\xi}_i,\sqrt{\alpha} \;  \tilde{b}_k]\cdot \nabla ( \sqrt{\alpha} \;  \tilde{b}_k \cdot \nabla \ln F) \;dz \\
 =  \sum_{k=1}^3  &-2     \int  \kappa F  (\hat{\xi}_i \cdot \nabla (  \sqrt{\alpha} \;  \tilde{b}_k  \cdot \nabla \ln F))  (  \hat{\xi}_i \cdot \nabla( \sqrt{\alpha} \;  \tilde{b}_k \cdot \nabla \ln F) ) dz\\
 &-2    \int   \kappa (\hat{\xi}_i \cdot \nabla F) ( \sqrt{\alpha} \;   \tilde{b}_k  \cdot \nabla \ln F)  (  \hat{\xi}_i \cdot \nabla( \sqrt{\alpha} \;  \tilde{b}_k \cdot \nabla \ln F) ) dz \\
 &- 2 \int   {{- \kappa[\hat{\xi}_i,\sqrt{\alpha} \;  \tilde{b}_k]}}\cdot \nabla F  (  \hat{\xi}_i \cdot \nabla( \sqrt{\alpha} \;  \tilde{b}_k \cdot \nabla \ln F) ) \;dz  \\
 &+2   \int  \kappa   (\hat{\xi}_i \cdot \nabla  F) [\hat{e}_i,\sqrt{\alpha} \;  \tilde{b}_k]\cdot \nabla ( \sqrt{\alpha} \;  \tilde{b}_k \cdot \nabla \ln F) \;dz \\
 = \sum_{k=1}^3 & -2    \int  \kappa F  (\hat{\xi}_i \cdot \nabla ( \sqrt{\alpha} \;   \tilde{b}_k  \cdot \nabla \ln F))^2 \;dz -     \int \kappa   (\hat{\xi}_i \cdot \nabla F)( \hat{\xi}_i \cdot \nabla|  \sqrt{\alpha} \;  \tilde{b}_k  \cdot \nabla \ln F|^2) \;dz \\
 & - 2 \int    {{- \kappa[\hat{\xi}_i,\sqrt{\alpha} \;  \tilde{b}_k]}}\cdot \nabla F  (  \hat{\xi}_i \cdot \nabla( \sqrt{\alpha} \;  \tilde{b}_k \cdot \nabla \ln F) ) \;dz \\
 &  +2   \int  \kappa  (\hat{\xi}_i \cdot \nabla  F) [\hat{\xi}_i,\sqrt{\alpha} \;  \tilde{b}_k]\cdot \nabla ( \sqrt{\alpha} \;  \tilde{b}_k \cdot \nabla \ln F) \;dz . \nonumber
 \end{align*}
Note that the second term equals $-\frac{1}{2} \hat{I}_3$.  Hence 
\begin{align} 
\hat{I}_1+\hat{I}_2+\hat{I}_3 = &\int   Q(F) |\hat{\xi}_i \cdot   \nabla \ln F|^2  \;dz  -2  \sum_{k=1}^3  \int \kappa   F  (\hat{\xi}_i \cdot \nabla ( \sqrt{\alpha} \;   \tilde{b}_k  \cdot \nabla \ln F))^2 \;dz  \nonumber   \\
&+  \sum_{k=1}^3   \int \kappa    (\hat{\xi}_i \cdot \nabla F)( \hat{\xi}_i \cdot \nabla|  \sqrt{\alpha} \;  \tilde{b}_k  \cdot \nabla \ln F|^2) \;dz \label{I1+I_2+I_3} \\
& - 2 \sum_{k=1}^3  \int    {{- \kappa[\hat{\xi}_i,\sqrt{\alpha} \;  \tilde{b}_k]}}\cdot \nabla F  (  \hat{\xi}_i \cdot \nabla( \sqrt{\alpha} \;  \tilde{b}_k \cdot \nabla \ln F) ) \;dz   \nonumber \\
&+2  \sum_{k=1}^3   \int \kappa    (\hat{\xi}_i \cdot \nabla  F) [\hat{\xi}_i,\sqrt{\alpha} \;  \tilde{b}_k]\cdot \nabla ( \sqrt{\alpha} \;  \tilde{b}_k \cdot \nabla \ln F) \;dz  .\nonumber
\end{align}
Substitute $Q(F)$ into $\hat I_1$ and use integration by parts to obtain: 
\begin{align*} 
\hat I_1 =& \sum_{k=1}^3 \int \kappa  \sqrt{\alpha} \;  \tilde{b}_k \cdot \nabla( \sqrt{\alpha} \;  \tilde{b}_k \cdot \nabla F) |\hat{\xi}_i \cdot   \nabla \ln F|^2  \;dz \\
=& - \sum_{k=1}^3 \int \kappa ( \sqrt{\alpha} \;  \tilde{b}_k \cdot \nabla F)\sqrt{\alpha} \;   \tilde{b}_k \cdot \nabla (  |\hat{\xi}_i \cdot   \nabla \ln F|^2) \;dz \\
=& -2 \sum_{k=1}^3 \int \kappa  ( \sqrt{\alpha} \;  \tilde{b}_k \cdot \nabla F) (\hat{\xi}_i \cdot   \nabla \ln F)  \sqrt{\alpha} \;  \tilde{b}_k \cdot \nabla (  \hat{\xi}_i \cdot   \nabla \ln F) \;dz  \\
=  \sum_{k=1}^3 & -2 \int \kappa  ( \sqrt{\alpha} \;  \tilde{b}_k \cdot \nabla \ln F) (\hat{\xi}_i \cdot   \nabla F)  \hat{\xi}_i  \cdot \nabla (  \sqrt{\alpha} \;  \tilde{b}_k  \cdot   \nabla \ln F) \;dz \\
&-2  \int \kappa  ( \sqrt{\alpha} \;  \tilde{b}_k \cdot \nabla \ln F) (\hat{\xi}_i \cdot   \nabla F)  {{ [\sqrt{\alpha} \;  \tilde{b}_k,\hat{\xi}_i]}}  \cdot   \nabla \ln F \;dz \\
=  \sum_{k=1}^3  &  -\int \kappa ( \hat{\xi}_i  \cdot \nabla F) \hat{\xi}_i \cdot \nabla|  \sqrt{\alpha} \;  \tilde{b}_k  \cdot \nabla \ln F|^2 \;dz \\
&+2 \int \kappa  ( \sqrt{\alpha} \;  \tilde{b}_k \cdot \nabla \ln F) (\hat{\xi}_i \cdot   \nabla F)   {{ [\hat{\xi}_i,\sqrt{\alpha} \;  \tilde{b}_k]}}  \cdot   \nabla \ln F \;dz .
\end{align*}

The first term of $\hat I_1$ cancels with the third term in (\ref{I1+I_2+I_3}). We get then 
\begin{align*}
\hat{I}_1+\hat{I}_2+\hat{I}_3 =  \sum_{k=1}^3  \; &\; 2 \int \kappa  ( \sqrt{\alpha} \;  \tilde{b}_k \cdot \nabla \ln F) (\hat{\xi}_i \cdot   \nabla F)   {{ [\hat{\xi}_i, \sqrt{\alpha} \;  \tilde{b}_k]}}  \cdot   \nabla \ln F \;dz\\
&-2  \int \kappa   F  (\hat{\xi}_i \cdot \nabla ( \sqrt{\alpha} \;   \tilde{b}_k  \cdot \nabla \ln F))^2 \;dz  \nonumber   \\
& + 2  \int    {{\kappa[\hat{\xi}_i,\sqrt{\alpha} \;  \tilde{b}_k]}}\cdot \nabla F  (  \hat{\xi}_i \cdot \nabla( \sqrt{\alpha} \;  \tilde{b}_k \cdot \nabla \ln F) ) \;dz   \nonumber \\
&+2    \int \kappa    (\hat{\xi}_i \cdot \nabla  F) [\hat{\xi}_i,\sqrt{\alpha} \;  \tilde{b}_k]\cdot \nabla ( \sqrt{\alpha} \;  \tilde{b}_k \cdot \nabla \ln F) \;dz  .\nonumber
    \end{align*}
In the first, third and last term above we substitute for $[\hat{\xi}_i,\sqrt{\alpha} \;  \tilde{b}_k]$ the expression (see Lemma \ref{lem_commutators})
$$[\hat{\xi}_i,\sqrt{\alpha} \;  \tilde{b}_k] = \sqrt{\alpha} \; [\hat{\xi}_i,  \tilde{b}_k] + {\frac{\alpha'}{|v-w|\sqrt{\alpha}}} (v_i-w_i)\tilde{b}_k. $$  
For each $k = 1,2,3$, the contributions coming from $\sqrt{\alpha} \; [\hat{\xi}_i,  \tilde{b}_k]$ are zero {(after summation over $i$)}:   
\begin{align*}
 \sum_{i=1}^3 \; 
 & 2  \int \kappa \;   (\sqrt{\alpha} \;  \tilde{b}_k \cdot \nabla \ln F) (\hat{\xi}_i \cdot   \nabla F)  \sqrt{\alpha}  {{ [\hat{\xi}_i,  \tilde{b}_k]}} \cdot   \nabla \ln F \;dz  \\
 &+ 2 \int  \kappa  \sqrt{\alpha} {{ [\hat{\xi}_i, \;  \tilde{b}_k]}} \cdot \nabla F  (  \hat{\xi}_i \cdot \nabla( \sqrt{\alpha} \; \tilde{b}_k \cdot \nabla \ln F) ) \;dz  \\
& +2   \int   \kappa  (\hat{\xi}_i \cdot \nabla  F)\sqrt{\alpha}  [\hat{\xi}_i, \;  \tilde{b}_k]\cdot \nabla ( \sqrt{\alpha} \;  \tilde{b}_k \cdot \nabla \ln F) \;dz \\
&= 0. 
\end{align*}
The contribution arising from ${\frac{\alpha'}{|v-w|\sqrt{\alpha}}} (v_i-w_i)\tilde{b}_k$ reads as 
\begin{align}
2   &\int \kappa \;  {\frac{\alpha'}{|v-w|\sqrt{\alpha}}} (\sqrt{\alpha} \;  \tilde{b}_k \cdot \nabla \ln F) (\hat{\xi}_i \cdot   \nabla F)   {{ (v_i-w_i) \tilde{b}_k}} \cdot   \nabla \ln F \;dz\nonumber \\ 
+2 &\int \kappa{\frac{\alpha'}{|v-w|\sqrt{\alpha}}} {{ (v_i-w_i)\tilde{b}_k}} \cdot \nabla F  (  \hat{\xi}_i \cdot \nabla( \sqrt{\alpha} \; \tilde{b}_k \cdot \nabla \ln F) ) \;dz  \nonumber  \\
& +2    \int \kappa   {\frac{\alpha'}{|v-w|\sqrt{\alpha}}} (\hat{\xi}_i \cdot \nabla  F) (v_i-w_i)\tilde{b}_k \cdot \nabla ( \sqrt{\alpha} \;  \tilde{b}_k \cdot \nabla \ln F) \;dz. \label{main_push} 
\end{align}
We will show that for each $k$ the sum over $i=1,2,3 $ of (\ref{main_push}) will equal 
$$
2 \int \kappa {F} {{\frac{(\alpha ')^2}{ \alpha^2} }} (\sqrt{\alpha} \;  \tilde{b}_k \cdot \nabla \ln F)^2 \;dz. 
$$
Let
\begin{align*}
\hat n 
:=\left( \begin{array}{ccc} 0\\v-w \\0\\- (v-w)
\end{array}  \right) 
= \sum_i  \left( \begin{array}{ccc} 0\\(v_i-w_i)e_i\\0\\- (v_i-w_i)e_i
\end{array}  \right)
=\sum_i (v_i - w_i) \hat{\xi}_i
\end{align*}

and $ h_k:= \sqrt{\alpha} \;  \tilde{b}_k \cdot \nabla \ln F$ .  The sum over $i =1,2,3$ of (\ref{main_push}) can be  rewritten as 
\begin{align*}
2\int \kappa & {\frac{\alpha'}{|v-w|\sqrt{\alpha}}}\; h_k\; (\hat{n} \cdot   \nabla F)  ( {{  \tilde{b}_k}} \cdot   \nabla \ln F) \;dz \\
& +2\int \kappa{\frac{\alpha'}{|v-w|\sqrt{\alpha}}} ( {{ \tilde{b}_k}} \cdot \nabla F)  (  \hat{n} \cdot \nabla h_k ) \;dz  \\
& +2\int \kappa  {\frac{\alpha'}{|v-w|\sqrt{\alpha}}} (\hat{n} \cdot \nabla  F)( \tilde{b}_k \cdot \nabla h_k) \;dz\\
& = :2A_1+2A_2+2A_3.
\end{align*}

Since $\textrm{div}(\tilde{b}_k)=0$ we write $A_3$ as
\begin{align*}
A_3 =& \int \kappa {\frac{\alpha'}{|v-w|\sqrt{\alpha}}}  (\hat{n} \cdot \nabla  F) \textrm{div}(\tilde{b}_k h_k) \;dz\\
= &- \int \kappa   h_k \tilde{b}_k \cdot  \nabla\left({\frac{\alpha'}{|v-w|\sqrt{\alpha}}} \;  \hat{n} \cdot \nabla  F  \right) dz\\
= & - \int \kappa    h_k\;  {\frac{\alpha'}{|v-w|\sqrt{\alpha}}} \; \hat{n}  \cdot  \nabla(\tilde{b}_k \cdot \nabla  F  ) dz - \int \kappa   h_k \;\left[\tilde{b}_k ,{\frac{\alpha'}{|v-w|\sqrt{\alpha}}}  \hat{n} \right]\cdot \nabla  F dz \\
= & - \int \kappa   h_k \; {\frac{\alpha'}{|v-w|\sqrt{\alpha}}} \; \hat{n}  \cdot  \nabla( \frac{F}{\sqrt{\alpha}}h_k  )  dz - \int \kappa   h_k \left[\tilde{b}_k ,{\frac{\alpha'}{|v-w|\sqrt{\alpha}}}  \hat{n} \right]\cdot \nabla  F \; dz \\
= & - \int \kappa   h_k^2 \; {\frac{\alpha'}{|v-w|\sqrt{\alpha}}} \hat{n}  \cdot  \nabla( \frac{F}{\sqrt{\alpha}}  )\;dz  - \int \kappa   {\frac{\alpha'}{|v-w|{\alpha}}} h_k  {F} (\hat{n}  \cdot  \nabla h_k) \;dz    - \int \kappa  h_k [\tilde{b}_k ,{\frac{\alpha'}{|v-w|\sqrt{\alpha}}}  \hat{n} ]\cdot \nabla  F \;dz.
\end{align*}
The second term cancels with $A_2$ since 
$$
A_2 = \int \kappa {\frac{\alpha'}{|v-w|\sqrt{\alpha}}} ( {{ \tilde{b}_k}} \cdot \nabla F)  (  \hat{n} \cdot \nabla h_k ) dz  = \int \kappa {\frac{\alpha'}{|v-w|{\alpha}}} ( \sqrt{\alpha}{{ \tilde{b}_k}} \cdot \nabla F)  (  \hat{n} \cdot \nabla h_k ) dz = \int  \kappa  {\frac{\alpha'}{|v-w|{\alpha}}} h_k  {F} (\hat{n}  \cdot  \nabla h_k)dz.
$$ 
Note also that 
$$A_1 =\int \kappa {\frac{\alpha'}{|v-w|{\alpha}}}\;{h_k^2}\; (\hat{n} \cdot   \nabla F) \;dz. $$   Hence 
\begin{align*}
A_1+A_2+A_3 = - \int \kappa  h_k^2 \; {\frac{\alpha'}{|v-w|\sqrt{\alpha}}}\; \hat{n}  \cdot  \nabla( \frac{F}{\sqrt{\alpha}}  ) dz   - \int   \kappa h_k \left[\tilde{b}_k ,{\frac{\alpha'}{|v-w|\sqrt{\alpha}}}  \hat{n} \right]\cdot \nabla  F dz +\int  \kappa{\frac{\alpha'}{|v-w|{\alpha}}}\;{h_k^2}\; (\hat{n} \cdot   \nabla F) \;dz.
\end{align*}
Taking into account that
$$ \left[\tilde{b}_k ,{\frac{\alpha'}{|v-w|\sqrt{\alpha}}}  \hat{n} \right] = {\frac{\alpha'}{|v-w|\sqrt{\alpha}}} \underbrace{[\tilde{b}_k ,  \hat{n} ]}_{=0} + \underbrace{\tilde{b}_k \cdot \nabla\left( {\frac{\alpha'}{|v-w|\sqrt{\alpha}}}  \right)}_{=0} \hat{n} = 0, $$
%{ We need to check that the commutator $[\tilde{b}_k ,  \hat{n} ]$ is zero. } 
and 
$$\nabla \alpha \cdot \hat{n} = \alpha ' \frac{ \hat{n}}{|v-w|} \cdot \hat{n} = 2 \alpha ' |v-w|,$$
 we get
\begin{align*}
A_1+A_2+A_3 =&   - \int  \kappa h_k^2 \; {\frac{\alpha'}{|v-w|\sqrt{\alpha}}}\; \hat{n}  \cdot  \nabla( \frac{F}{\sqrt{\alpha}}  )dz    +\int \kappa {h_k^2}{\frac{\alpha'}{|v-w|{\alpha}}}\;\; (\hat{n} \cdot   \nabla F) \;dz \\
= & - \int \kappa   F h_k^2 \; {\frac{\alpha'}{|v-w|\sqrt{\alpha}}}\;  \hat{n}  \cdot  \nabla( \frac{1}{\sqrt{\alpha}}  )  \;dz\\
= & \;  \int \kappa  F h_k^2 {{\frac{(\alpha ')^2}{ \alpha^2} }}\;dz.
\end{align*}

Summarizing, 
\begin{align}
\sum_i \hat{I}_1+\hat{I}_2+\hat{I}_3 =  \sum_{k,i=1}^3 & -2   \int \kappa  F  (\hat{\xi}_i \cdot \nabla (  \sqrt{\alpha} \; \tilde{b}_k  \cdot \nabla \ln F))^2 \;dz  + 2  \int \kappa {{\frac{(\alpha ')^2}{ \alpha^2} }}  F h_k^2 \;dz.\label{1a}
\end{align}

We now perform the same computations for the terms with $\tilde{\xi}_i$. In this case the terms with the commutators are identically zero, since $[\tilde \xi_i, \tilde b_k]=0=[\tilde \xi_i, \sqrt{\alpha}\; \tilde b_k]$.  This  yields 
\begin{align}
\tilde{I}_1+\tilde{I}_2+\tilde{I}_3 =  -2  \sum_{i,k=1}^3  \int \kappa   F  (\tilde{\xi}_i \cdot \nabla ( \sqrt{\alpha} \;\tilde{b}_k  \cdot \nabla \ln F))^2 \;dz.\label{2a}
\end{align}
The claim follows after substituting (\ref{1a}) and (\ref{2a}) into (\ref{I1I2I3}). 
\end{proof}

\bibliography{landaurefs_proposal}
\bibliographystyle{plain}

\end{document}